%% file: paper.tex
\documentclass[english,10pt]{article}

\usepackage[english]{babel}
\usepackage[utf8x]{inputenc}
\usepackage[T1]{fontenc}
\usepackage[formats]{listings}
\usepackage{geometry}
\usepackage{enumitem}

\usepackage{graphicx}
\usepackage[colorinlistoftodos]{todonotes}
\usepackage{setspace}
\usepackage{placeins}
\usepackage{enumitem}
\usepackage{titlesec}
\usepackage{cite}
\usepackage{comment}
\usepackage{caption}
\usepackage{subcaption}

\usepackage[colorlinks=true, allcolors=blue]{hyperref}

\usepackage{amsmath}
\usepackage{amsthm}
\usepackage{amssymb}
\usepackage{amsfonts}
\usepackage{bbm}
\usepackage{bm}
\usepackage{nicefrac}
\usepackage{mathtools}

\usepackage{tikz}

\usepackage{graphicx}
\usepackage{fullpage}
\usepackage{float}
\usepackage{wrapfig}
\usepackage{caption}

\usepackage{algpseudocode}
\usepackage{algorithm}
\usepackage{listings}

\usepackage{amsthm}\usepackage{dsfont}\usepackage{array}\usepackage{mathrsfs}\usepackage{cite}\usepackage{comment}\usepackage{mathrsfs}

\makeatother

\usepackage{babel}
\usepackage{color}
\usepackage{centernot}

\usepackage{babel}
\usepackage{sansmath}

\definecolor{ejc}{RGB}{255,0,0}
\definecolor{yxc}{RGB}{0,0,200}
\definecolor{ps}{RGB}{0,150,0}
\newcommand{\ejc}[1]{\textcolor{ejc}{[EJC: #1]}}
\newcommand{\yxc}[1]{\textcolor{yxc}{[YC: #1]}}

\newcommand{\ps}[1]{\textcolor{ps}{PS: #1}}

\date{June  2017}

\title{The Likelihood Ratio Test in High-Dimensional Logistic Regression Is Asymptotically a \emph{Rescaled} Chi-Square}

\author{Pragya Sur\thanks{Department of Statistics, Stanford
    University, Stanford, CA 94305, U.S.A.} \and 
  Yuxin Chen\thanks{Department of Electrical Engineering, Princeton
    University, Princeton, NJ 08544, U.S.A.}  \and Emmanuel J. Cand\`es\footnotemark[1] \thanks{Department of Mathematics, Stanford
    University, Stanford, CA 94305, U.S.A.} }

\theoremstyle{plain}\newtheorem{lemma}{\textbf{Lemma}}\newtheorem{theorem}{\textbf{Theorem}}\newtheorem{corollary}{\textbf{Corollary}}\newtheorem{definition}{\textbf{Definition}}\newtheorem{proposition}{\textbf{Proposition}}

\theoremstyle{definition}

\theoremstyle{definition}\newtheorem{remark}{\textbf{Remark}}

\newcommand{\bbeta}{\bm{\beta}}
\newcommand{\hbbeta}{\hat{\bm{\beta}}}
\newcommand{\tbbeta}{\tilde{\bm{\beta}}}
\newcommand{\tbbetaimin}{\tilde{\bm{\beta}}_{[-i]}}
\newcommand{\hbbetaimin}{\hat{\bm{\beta}}_{[-i]}}
\newcommand{\hbbetaonemin}{\hat{\bm{\beta}}_{[-1]}}
\newcommand{\hbbetatwomin}{\hat{\bm{\beta}}_{[-2]}}
\newcommand{\hbbetaonetwomin}{\hat{\bm{\beta}}_{[-12]}}

\newcommand{\bX}{\bm{X}}
\newcommand{\tbX}{\tilde{\bm{X}}}
\newcommand{\bzero}{\bm{0}}

\newcommand{\dber}{{\mathsf{Bernoulli}}}
\newcommand{\dunif}{{\mathsf{Unif}}}
\newcommand{\dnorm}{{\mathcal{N}}}

\newcommand{\tp}{\top}
\newcommand{\iid}{\stackrel{ \text{i.i.d.} }{\sim} }
\newcommand{\bd}{\bm{d}}
\newcommand{\bu}{\bm{u}}
\newcommand{\bv}{\bm{v}}
\newcommand{\bM}{\bm{M}}
\newcommand{\bO}{\bm{Q}}
\newcommand{\R}{\mathbb{R}}
\newcommand{\spS}{\mathbb{S}}
\newcommand{\orO}{\mathbb{O}}
\newcommand{\eqd}{\stackrel{\mathrm{d}}{=}}
\newcommand{\Id}{\bm{I}}
\newcommand{\hbpi}{\hat{\bm{\pi}}}
\newcommand{\bs}{b_{\ast}}
\newcommand{\taus}{\tau_{\ast}}

\newcommand{\tb}{\tilde{b}}
\newcommand{\Nu}{\mathcal{V}}
\newcommand{\tNu}{\tilde{\mathcal{V}}}
\newcommand{\prox}{\mathsf{prox}}
\newcommand{\by}{\bm{y}}
\newcommand{\bx}{\bm{x}}
\newcommand{\btheta}{\bm{\theta}}

\newcommand{\etab}{\bm{\eta}}
\newcommand{\bnu}{\bm{\nu}}
\newcommand{\bS}{\bm{s}}
\newcommand{\bA}{\bm{A}}
\newcommand{\bB}{\bm{B}}
\newcommand{\bh}{\bm{h}}
\newcommand{\bb}{\bm{b}}
\newcommand{\mb}{\bm{m}}
\newcommand{\bq}{\bm{q}}

\newcommand{\bW}{\bm{Z}}
\newcommand{\Lcal}{\ell}
\newcommand{\bZ}{\bm{Z}}
\newcommand{\bD}{\bm{D}}
\newcommand{\bG}{\bm{G}}

\newcommand{\A}{\mathcal{A}}
\newcommand{\br}{\bm{r}}
\newcommand{\brt}{\tilde{\bm{r}}}
\newcommand{\bw}{\bm{w}}
\newcommand{\Xcone}{\bX_{\cdot 1}}
\newcommand{\Dtbeta}{\bD_{\tilde{\bbeta}}}
\newcommand{\Dhbeta}{\bD_{\hat{\bbeta}}}

\newcommand{\tbG}{\tilde{\bm{G}}}
\newcommand{\tbGi}{\tilde{\bm{G}}_{(i)}}
\newcommand{\tbGimin}{\tilde{\bm{G}}_{[-i]}}
\newcommand{\bGimin}{\bm{G}_{[-i]}}
\newcommand{\bGonetwomin}{\bm{G}_{[-12]}}
\newcommand{\bH}{\bm{H}}
\newcommand{\Dinterm}{\bD_{{\hbbeta}, {\tbb}} }
\newcommand{\Ginterm}{\bm{G}_{\hbbeta,\tbb}}
\newcommand{\Ltilde}{\tilde{\Lcal}}

\newcommand{\tr}{\mathrm{Tr}}
\newcommand{\talpha}{\tilde{\alpha}}

\newcommand{\convd}{\stackrel{\mathrm{d}}{\rightarrow}}
\newcommand{\bV}{\bm{V}}
\newcommand{\bU}{\bm{U}}
\newcommand{\bSigma}{\bm{\Sigma}}
\newcommand{\hbb}{\hat{\bm{b}}}
\newcommand{\hbbonemin}{\hat{\bm{b}}_{[-1]}}
\newcommand{\alphas}{\alpha_{\ast}}
\newcommand{\Lcalimin}{\Lcal_{[-i]}}
\newcommand{\convP}{\stackrel{\opP}{\rightarrow}}
\newcommand{\convLtwo}{\stackrel{\mathrm{L}_2}{\rightarrow}}

\newcommand{\tq}{\tilde{q}}
\newcommand{\tdelta}{\tilde{\delta}}
\newcommand{\bone}{\bm{1}}
\newcommand{\bXonetwomin}{\bX_{[-12]}}

\newcommand{\ty}{\tilde{y}}
\newcommand{\be}{\bm{e}}

\newcommand{\bt}{\bm{t}}

\newcommand{\ba}{\boldsymbol{a}}

\newcommand{\tbb}{\tilde{\bm{b}}}

\newcommand{\opP}{\mathbb{P}}

\makeatletter
\renewenvironment{proof}[1][\proofname] {
	\par\pushQED{\qed}\normalfont
	\topsep6\p@\@plus6\p@\relax
	\trivlist\item[\hskip\labelsep\bfseries#1\@addpunct{:}]
 	\ignorespaces
} {
	\popQED\endtrivlist\@endpefalse
}
\makeatother

\DeclareMathOperator{\E}{\mathbb{E}}

\begin{document}
\maketitle

\begin{abstract}
   Logistic regression is used thousands of times a day to fit data,
  predict future outcomes, and assess the statistical significance of
  explanatory variables. When used for the purpose of statistical
  inference, logistic models produce p-values for the regression
  coefficients by using an approximation to the distribution of the
  likelihood-ratio test. Indeed, Wilks' theorem asserts that whenever
  we have a fixed number $p$ of variables, twice the
  log-likelihood ratio (LLR) $2 \Lambda$ is distributed as a
  $\chi^2_k$ variable in the limit of large sample sizes $n$; here,
  $\chi^2_k$ is a chi-square with $k$ degrees of freedom and $k$ 
  the number of variables being tested.  In this paper, we prove that
  when $p$ is not negligible compared to $n$, Wilks' theorem does not
  hold and that the chi-square approximation is grossly incorrect; in
  fact, this approximation produces p-values that are far too small
  (under the null hypothesis).

  Assume that $n$ and $p$ grow large in such a way that
  $p/n \rightarrow \kappa$ for some constant $\kappa < 1/2$. We prove that for
  a class of logistic models, the LLR converges to a {\em rescaled}
  chi-square, namely, $2\Lambda ~\stackrel{\mathrm{d}}{\rightarrow}~ \alpha(\kappa) \chi_k^2$,
  where the scaling factor $\alpha(\kappa)$ is greater than one as
  soon as the dimensionality ratio $\kappa$ is positive. Hence, the
  LLR is larger than classically assumed. For instance, when
  $\kappa = 0.3$, $\alpha(\kappa) \approx 1.5$. In general, we show
  how to compute the scaling factor by solving a nonlinear system of
  two equations with two unknowns.  Our mathematical arguments are
  involved and use techniques from approximate message passing
  theory, from non-asymptotic random matrix theory and from convex
  geometry. We also complement our mathematical study by showing that
  the new limiting distribution is accurate for finite sample sizes.

  Finally, all the results from this paper extend to some other regression
  models such as the probit regression model.

\end{abstract}


\input{introduction}

\input{main}

\input{numerics}

\input{preliminaries}

\input{stat-dim}

\input{AMP}

\input{LRT}
\input{discussion}

\section*{Acknowledgements}

E.~C.~was partially supported by the Office of Naval Research under
grant N00014-16-1-2712, and by the Math + X Award from the Simons
Foundation. Y.~C.~and P.~S.~are grateful to Andrea
Montanari for his help in understanding AMP and
\cite{donoho2013high}. Y.~C.~thanks Kaizheng Wang and Cong Ma for helpful discussion about \cite{el2015impact}, and  P.~S.~thanks Subhabrata Sen for
several helpful discussions regarding this project. E.~C.~would like
to thank Iain Johnstone for a helpful discussion as well.

\input{appendix}

\bibliographystyle{plain}
\bibliography{bibfileLR}

\end{document}

%% file: introduction.tex
\section{Introduction}\label{sec: introduction}

Logistic regression is by far the most widely used tool for relating a
binary response to a family of explanatory variables.  This model is
used to infer the importance of variables and nearly all standard
statistical softwares have inbuilt packages for obtaining p-values for
assessing the significance of their coefficients. For instance, one
can use the snippet of $\mathrm{R}$ code below to fit a logistic
regression model from a vector $\verb|y|$ of binary responses and a
matrix \verb|X| of covariates:
\begin{verbatim}
    fitted <- glm(y ~ X+0, family = `binomial')
    pvals  <- summary(fitted)$coefficients[,4]
\end{verbatim}
The vector \verb|pvals| stores p-values for testing whether a variable
belongs to a model or not, and it is well known that the underlying
calculations used to produce these p-values can also be used to
construct confidence intervals for the regression coefficients. Since
logistic models are used hundreds of times every day for inference
purposes, it is important to know whether these calculations---e.g.~these p-values---are accurate and can be trusted.

\subsection{Binary regression}\label{subsec: probset}

Imagine we have $n$
samples of the form $(y_i,\bX_i)$, where $y_i \in \{0,1 \}$ and
$\bX_i \in \R^p$. In a generalized linear model, one postulates the
existence of a link function $\mu(\cdot)$ relating the conditional
mean of the response variable to the linear predictor
$\bm{X}_i^{\top}\bm{\beta}$, 
\begin{equation}
	 \E[y_i|\bX_i] = \mu(\bX_i^{\tp} \bbeta), 
\end{equation}
where
$\bbeta = [\beta_1, \beta_2, \ldots, \beta_p]^{\top} \in \mathbb{R}^{p}$ is
an unknown vector of parameters. We focus here on the two most
commonly used binary regression models, namely, the logistic and the
probit models for which
\begin{equation}
	\mu(t):=\begin{cases}
e^t/(1+e^t) \quad & \text{in the logistic model},\\
\Phi(t) & \text{in the probit model};
\end{cases}
\end{equation}
here, $\Phi$ is the cumulative distribution function (CDF) of a
standard normal random variable.
In both cases, the {\em Symmetry Condition}
\begin{equation}
\label{eq: symmetry}
\mu(t) + \mu(-t) =1
\end{equation}
holds, which says that the two types $y_i = 0$ and $y_i=1$ are treated
in a symmetric fashion.
Assuming that the observations are independent, 
the negative log-likelihood function is given by
\cite[Section 4.1.2]{agresti2011categorical}
\begin{equation*}
\ell\left(\bm{\beta}\right):= - \sum_{i=1}^{n}  \left\{ y_{i} \log \left(\frac{\mu_i}{1- \mu_i} \right)  + \log \left(1-\mu_i \right)  \right\}, \qquad \mu_i := \mu(\bX_i^{\tp} \bbeta). 
\end{equation*}
Invoking the symmetry condition, a little algebra reveals an equivalent
expression
\begin{equation}
	\ell\left(\bm{\beta}\right):=  \sum\nolimits_{i=1}^{n}  \rho\big( -\tilde{y}_i \bm{X}_i^{\top} \bm{\beta}  \big),
	 \label{eq:log-likelihood-alternative}
\end{equation}
where
%
\begin{equation}
	\tilde{y}_i :=\begin{cases}
1\quad & \text{if }y_i=1,\\
-1 & \text{if }y_i=0, 
\end{cases}
	\qquad \text{and} \qquad
\rho(t):=\begin{cases}
\log\left(1+e^{t}\right)\quad & \text{in the logistic case},\\
-\log\Phi\left(-t\right) & \text{in the probit case}.
\end{cases}
	\label{eq:link-function-both-models}
\end{equation}
Throughout we refer to this function $\rho$ as the {\em effective
  link}.

\subsection{The likelihood-ratio test and  Wilks' phenomenon}

Researchers often wish to
 determine which covariates are of importance, or more precisely,  to test whether the $j{\text{th}}$ variable belongs to the model or not:  formally, we wish to test the hypothesis 
\begin{equation}\label{eq: singlecoeff}
	H_j: \quad \beta_j = 0 \quad \text{versus} \quad \beta_j \neq 0. 
\end{equation}
Arguably, one of the most commonly deployed techniques for testing $H_j$ 
is the likelihood-ratio test (LRT), 
which is based on the log-likelihood ratio (LLR) statistic
\begin{equation}
	\label{eq:LLR-stat}
	\Lambda_j := \ell\big( \hbbeta_{(-j)} \big) - \ell \big( \hbbeta \big).
\end{equation}
Here, $\hbbeta$ and $\hbbeta_{(-j)}$ denote respectively the maximum likelihood estimates (MLEs) under the full model and the reduced model on dropping the $j{\text{th}}$ predictor; that is, 
\[
	\hbbeta = \arg \min _{\bbeta \in \R^p}\Lcal(\bbeta) \qquad \text{and} \qquad \hbbeta_{(-j)} = \arg \min _{\bbeta \in \R^p, \beta_j=0}\Lcal(\bbeta).
\]
%
%
Inference based on such log-likelihood ratio statistics has been
studied extensively in prior literature
\cite{wilks1938large,chernoff1954distribution,mccullagh1989generalized}. Arguably,
one of the most celebrated results in the large-sample regime is the
Wilks' theorem.


 To describe the Wilk's phenomenon,
imagine we have a sequence of observations $(y_i, \bX_i)$ where
$y_i \in \{0,1 \}$, $\bX_i \in \R^p$ with $p$ fixed. Since we are
interested in the limit of large samples, we may want to assume that
the covariates are i.i.d.~drawn from some population with
non-degenerate covariance matrix so that the problem is fully
$p$-dimensional. As before, we assume a conditional logistic model for the
response.
{In this setting, Wilks' theorem \cite{wilks1938large} calculates the
  asymptotic distribution of $\Lambda_j(n)$ when $n$ grows to
  infinity:}
{
  \setlist{rightmargin=\leftmargin}
\begin{itemize}
\item[] (Wilks' phenomenon) {\em Under suitable regularity conditions which, for instance, guarantee that the MLE exists and is unique,\footnote{Such conditions would also typically imply asymptotic normality of the MLE.}
    the LLR statistic for testing $H_j : \beta_j=0$ vs.~$\beta_j \neq 0$ has asymptotic
    distribution under the null given by
		\begin{equation}\label{eq:Wilks-phenomenon} 
			2 \Lambda_j(n) ~\convd~ \chi^2_1,  \qquad\quad \text{as } n\rightarrow \infty.
		\end{equation}
		}
\end{itemize}
}
\noindent This fixed-$p$ large-$n$ asymptotic result, which is a
consequence of asymptotic normality properties of the MLE\cite[Theorem
5.14]{van2000asymptotic}, applies to a much broader class of testing
problems in parametric models; for instance, it applies to the probit
model as well. We refer the readers to \cite[Chapter
12]{lehmann2006testing} and \cite[Chapter 16]{van2000asymptotic} for a
thorough exposition and details on the regularity conditions under
which Wilks' theorem holds. Finally, there is a well-known extension
which states that if we were to drop $k$ variables from the model,
then the LLR would converge to a chi-square distribution with $k$
degrees of freedom under the hypothesis that the reduced model is
correct.

\subsection{Inadequacy of Wilks' theorem in high dimensions}

The chi-square approximation to the distribution of the LLR statistic
is used in standard statistical softwares to provide p-values for the
single or multiple coefficient likelihood ratio tests.
Here, we perform a simple experiment on synthetic data to study the
accuracy of the chi-square approximation when $p$ and $n $ are both
decently large. Specifically, we set $\bbeta = \bzero$ and test
$\beta_1=0$ vs.~$\beta_1 \neq 0$ using the LRT in a setting where
$p=1200$. In each trial, $n=4000$ observations are produced with $y_i \iid \dber(1/2)$, and $\bm{X}:=[\bX_1,\cdots,\bX_n]^{\top}\in \mathbb{R}^{n\times p}$ is obtained by generating a random
matrix composed of i.i.d.~$\mathcal{N}(0,1)$ entries.  
We fit a logistic regression of $\by$ on $\bX$ using R, and extract the p-values for each coefficient.   Figure \ref{fig: classical} plots the pooled histogram that aggregates $4.8 \times 10^5$ p-values in total (400 trials with $1200$ p-values obtained in each trial).

\begin{figure}
\begin{center}
	\begin{tabular}{ccc}
	\includegraphics[scale=0.28,keepaspectratio]{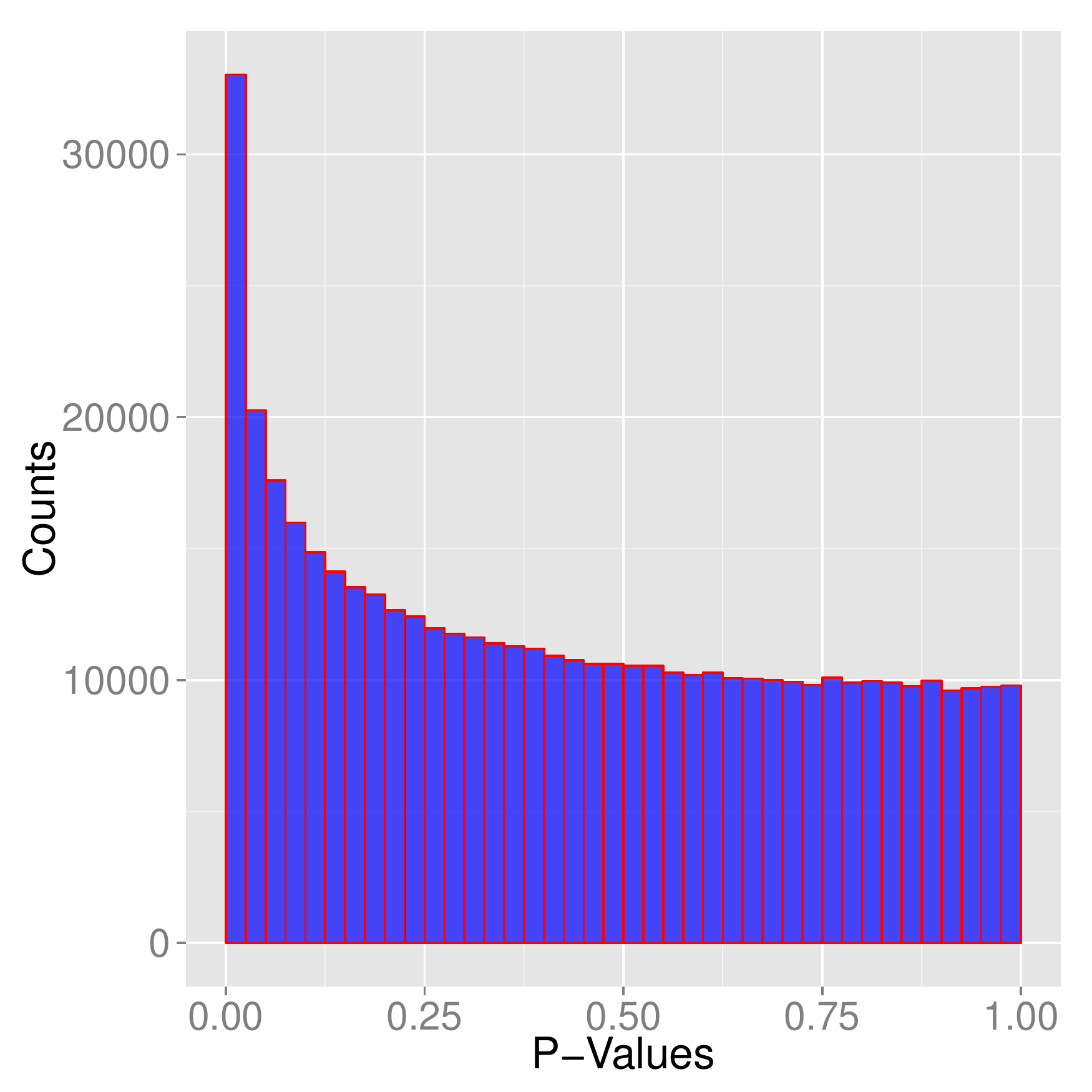}  
		& \includegraphics[scale=0.28,keepaspectratio]{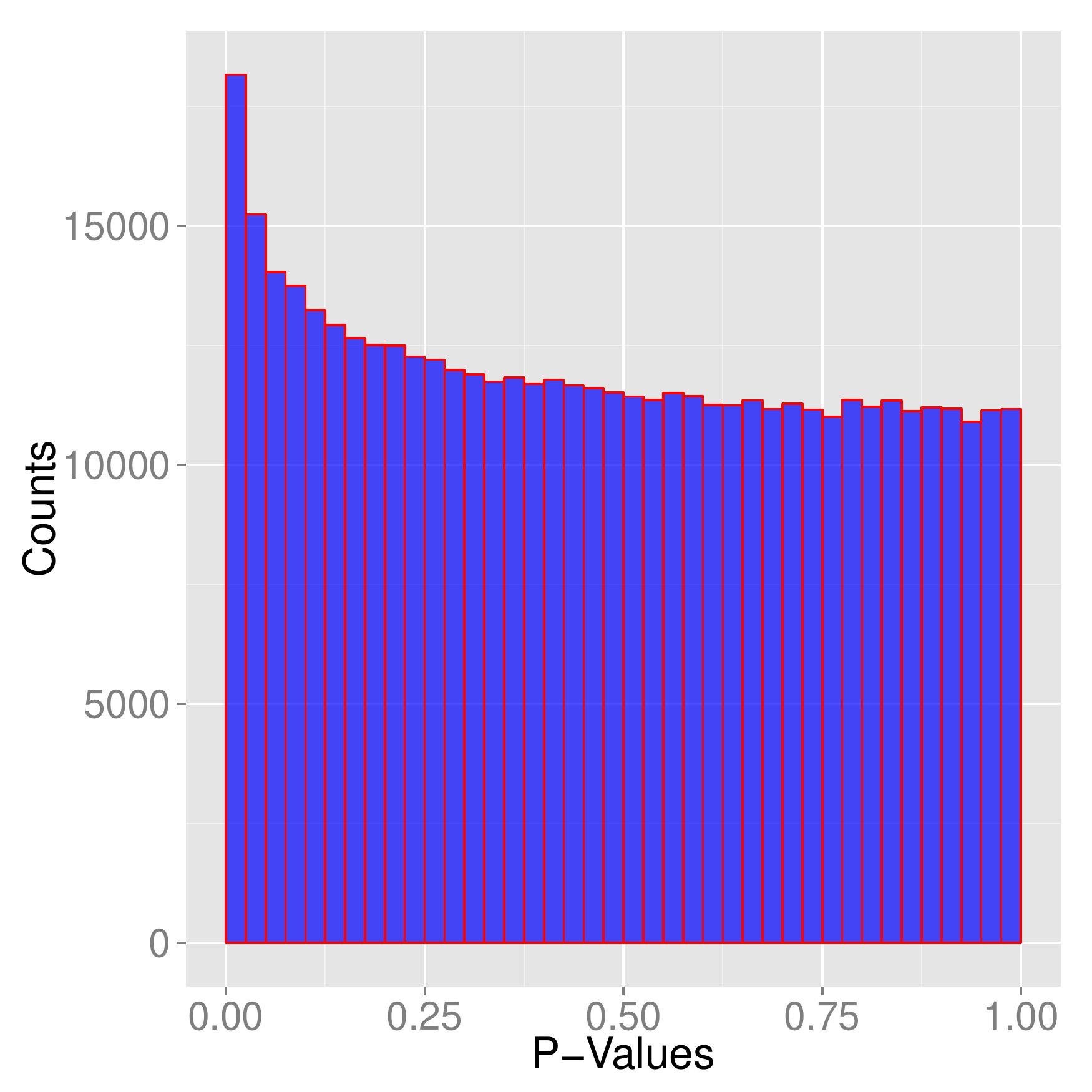}
		& \includegraphics[scale=0.28,keepaspectratio]{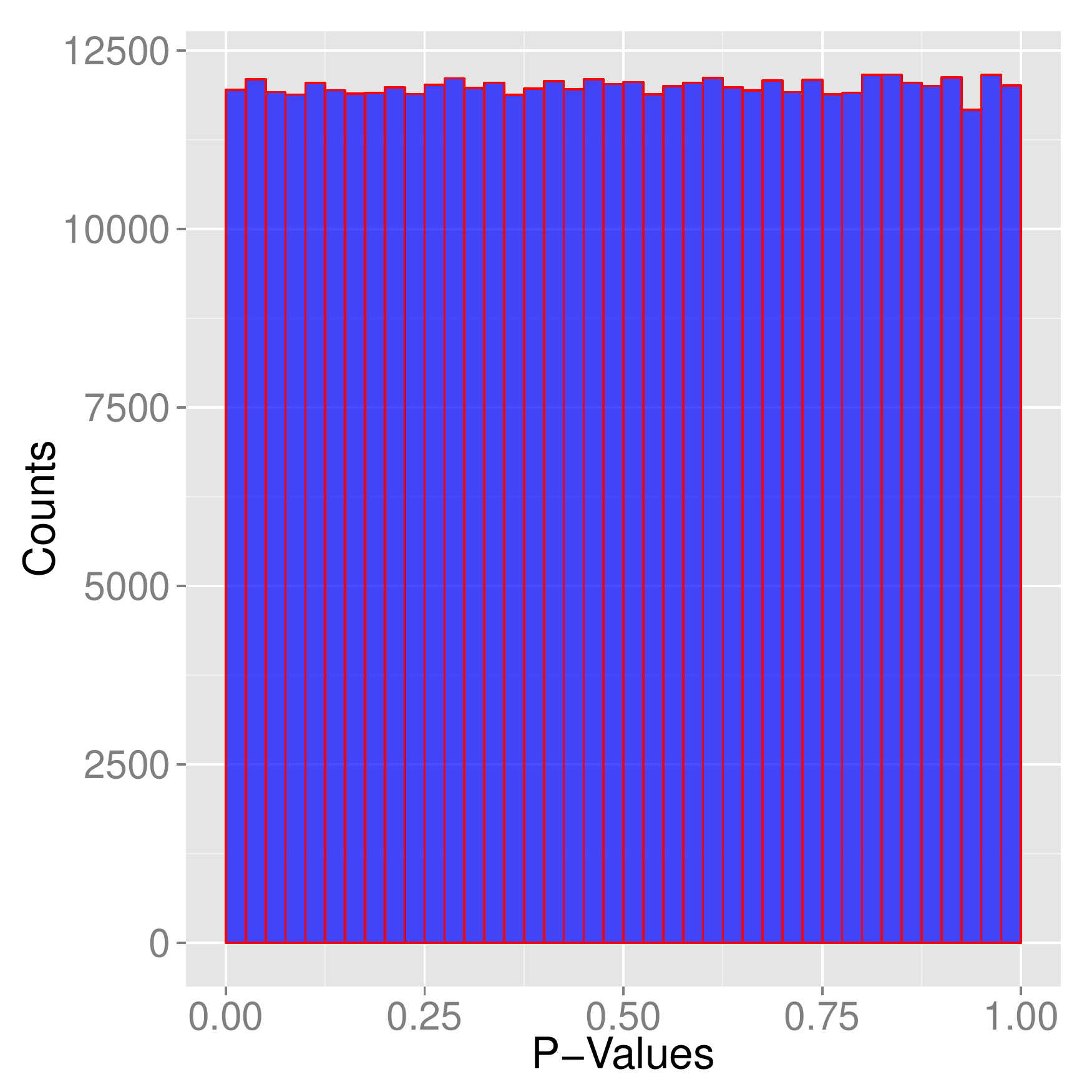} \tabularnewline
		(a) & (b) & (c) \tabularnewline
	\end{tabular}
\end{center}
\caption{Histogram of p-values for logistic regression under i.i.d.~Gaussian design, when $\bbeta=\bzero$, $n=4000$, $p=1200$, and
	$\kappa = 0.3$: (a) classically computed p-values; (b) Bartlett-corrected p-values; (c) adjusted p-values. 
  }\label{fig: classical}
\end{figure}



If the $\chi^2_1$ approximation were true,
then we would expect to observe uniformly distributed p-values. The
histrogram from Fig.~\ref{fig: classical} is, however, far from
uniform. This is an indication of the inadequacy of Wilks' theorem
when $p$ and $n$ are both large. The same issue was also 
reported in \cite{candes2016panning}, where the authors observed that
this discrepancy is highly problematic since the distribution is
skewed towards smaller values. Hence, such p-values cannot be trusted
to construct level-$\alpha$ tests and the problem is increasingly severe when we turn attention to smaller p-values as in
large-scale multiple testing applications.

\subsection{The Bartlett correction?}

A natural question that arises immediately is whether the observed
discrepancy could be an outcome of a finite-sample effect. It has been
repeatedly observed that the chi-square approximation does not yield
accurate results with finite sample size.  One correction to the LRT that is
widely used in finite samples is the Bartlett correction, which dates
back to Bartlett \cite{bartlett1937properties} and has been
extensively studied over the past few decades
(e.g.~\cite{box1949general,lawley1956general,bickel1990decomposition,cordeiro1995bartlett,cordeiro1983improved}). In
the context of testing for a single coefficient in the logistic
model, this correction can be described as follows 
\cite{moulton1993bartlett}: compute the expectation of
the LLR statistic up to terms of order $1/n^2$; that is, compute
a parameter $\alpha$ such that
\[
 \mathbb{E}[2\Lambda_j] = {1+\frac{\alpha}{n} + O\left(\frac{1}{n^2}\right)},\]
which suggests a corrected LLR statistic
\begin{equation}
	\label{eq:Bartlett-corrected}
	\frac{2 \Lambda_j}{1+\frac{\alpha_n}{n}}
\end{equation}
with  $\alpha_n$ being an estimator of $\alpha$.  With a proper choice of $\alpha_n$, one can ensure 
\[\E\left[\frac{2 \Lambda_j}{1+\frac{\alpha_n}{n}}\right] = 1+O\left(\frac{1}{n^2}\right)  \] 
in the classical setting where $p$ is fixed and $n$ diverges. 
In expectation, this corrected statistic is closer to a $\chi_1^2$ distribution than the original LLR for finite samples. Notably, the correction factor may in general be a function of the unknown $\bbeta$ and, in that case, must be estimated from the null model via maximum likelihood estimation.

In the context of GLMs, Cordeiro \cite{cordeiro1983improved} derived a
general formula for the Bartlett corrected LLR statistic, see
\cite{cribari1996bartlett,cordeiro2014introduction} for a detailed
survey.
In the case where there is no signal ($\bbeta = \bzero$), one can
compute $\alpha_n$ for the logistic regression model following
\cite{cordeiro1983improved} and\cite{moulton1993bartlett}, which
yields
\begin{equation}
	\label{eq:alphan}
	\alpha_n = \frac{n}{2} \left[ \tr\left(\bD_p^2\right) - \tr\left(\bD_{p-1}^2\right)  \right]. 
\end{equation}
Here, $\bD_p$ is the diagonal part of
$\bX(\bX^{\tp}\bX)^{-1}\bX^{\tp}$ and $\bD_{p-1}$ is that of
$\bX_{(-j)} \big( \bX_{(-j)}^{\tp}\bX_{(-j)}
\big)^{-1}\bX_{(-j)}^{\tp}$ in which $\bX_{(-j)}$ is the design matrix
$\bm{X}$ with the $j{\text{th}}$ column removed.
Comparing the adjusted LLRs to a $\chi^2_1$ distribution yields
adjusted p-values. In the setting of Fig.~\ref{fig: classical}(a), the
histogram of Bartlett corrected p-values is shown in Fig.~\ref{fig:
  classical}(b). As we see, these p-values are still far from
uniform.


If the mismatch is not due to finite sample-size effects, what is the
distribution of the LLR in high dimensions? Our main contribution is
to provide a very precise answer to this question; below, we derive
the high-dimensional asymptotic distribution of the log-likelihood
ratios, i.e.~in situations where the dimension $p$ is not necessarily
negligible compared to the sample size $n$.

%% file: main.tex
\section{Main results}
\subsection{Modelling assumptions}\label{subsec: assumptions}
In this paper, we focus on the high-dimensional regime where the sample size is not
much larger than the number of parameters to be estimated---a setting
which has attracted a flurry of activity in recent years. In
particular, we assume that the number $p(n)$ of covariates grows
proportionally with the number $n$ of observations; that is,
\begin{equation}
  \label{defn:kappa}
  \lim_{n \rightarrow \infty} \, \frac{p(n)}{n} = \kappa,
\end{equation}
where $\kappa >0 $ is a fixed constant independent of $n$ and $p(n)$. In fact, we shall also assume $\kappa < 1/2$ for both the logistic and the probit models, as the MLE is otherwise at $\infty$; see Section \ref{sec:separation}.  

To formalize the notion of high-dimensional asymptotics when both $n$ and $p(n)$ diverge, we consider a sequence of instances $\{\bX(n),\by(n) \}_{n \geq 0}$ such that for any $n$, 
\begin{itemize}
\item $\bX(n) \in \R^{n \times p(n)}$ has i.i.d.~rows
	$\bX_i(n) \sim \dnorm(0,\bSigma)$, where $\bSigma \in \mathbb{R}^{p(n)\times p(n)}$ is 
  positive definite; 
  
  \item $y_i(n) \, | \, \bX(n) \sim y_i(n) \, | \, \bX_i(n) \stackrel{\mathrm{ind.}}{\sim} \dber\left(\mu(\bX_i(n)^{\tp}\bbeta(n))\right)$, where $\mu$ satisfies the Symmetry Condition;
  \item we further assume $\bbeta(n) = \bzero$. From the Symmetry
    Condition it follows that $\mu(0)=1/2$, which directly implies
    that $\by(n)$ is a vector with i.i.d~$\dber(1/2)$ entries.
\end{itemize}
The MLE is denoted by $\hbbeta(n)$ and there are $p(n)$ LLR statistics
$\Lambda_j(n)$ ($1\leq j\leq p(n)$), one for each of the $p(n)$ regression coefficients. In
the sequel, the dependency on $n$ shall be suppressed whenever it is
clear from the context.


\subsection{When does the  MLE exist?}
\label{sec:separation}
Even though we are operating in the regime where $n>p$, the existence
of the MLE cannot be guaranteed for all $p$ and $n$. Interestingly,
the norm of the MLE undergoes a sharp phase transition in the sense
that 
$$\|\hat{\bm{\beta}}\|= \infty \quad \text{if }\kappa > 1/2 \qquad \text{and} \qquad \|\hat{\bm{\beta}}\| < \infty \quad \text{if } \kappa < 1/2.$$
%

Here, we develop some understanding about this phenomenon. 
Given that $\rho(t)\geq\rho(-\infty)=0$ for both the logistic and
probit models, each summand in (\ref{eq:log-likelihood-alternative})
is minimized if
$\tilde{y}_i \bm{X}_{i}^{\top}\bm{\beta} = \infty$, which
occurs when
$\text{sign}(\bm{X}_{i}^{\top}\bm{\beta}) = \text{sign}(\tilde{y}_i)$
and $\|\bm{\beta}\| = \infty$.  As a result, if there exists a
nontrivial ray $\bm{\beta}$ such that
\begin{equation}
  \bm{X}_{i}^{\top}\bm{\beta} >0 \quad \text{ if }~ \tilde{y}_i = 1  \qquad\quad \text{and} \qquad\quad \bm{X}_{i}^{\top}\bm{\beta} <0 \quad \text{ if } ~\tilde{y}_i = -1  \label{eq:beta-sign}
\end{equation}
for any $1\leq i\leq n$, then pushing $\|\bm{\beta}\|$ to infinity
leads to an optimizer of (\ref{eq:log-likelihood-alternative}).  In
other words, the solution to (\ref{eq:log-likelihood-alternative})
becomes unbounded (the MLE is at $\infty$) whenever there is a
hyperplane perfectly separating the two sets of samples
$\{ i \mid \tilde{y}_i = 1 \}$ and $\{ i \mid \tilde{y}_i = -1 \}$.

Under the assumptions from Section \ref{subsec: assumptions},
$\tilde{y}_i$ is independent of $\bm{X}$ and the distribution of $\bX$
is symmetric. Hence, to calculate the chance that there exists a 
separating hyperplane, we can assume $\tilde{y}_i = 1$
$(1\leq i\leq n)$ without loss of generality. 
In this case, the event
(\ref{eq:beta-sign}) becomes
\begin{equation}
 \left\{ \bm{X}\bm{\beta}\mid\bm{\beta}\in\mathbb{R}^{p}\right\} \cap\mathbb{R}_{++}^{n}\neq\emptyset,
 \label{eq:Xbeta-R++}
\end{equation}
 where $\mathbb{R}_{++}^{n}$ is the positive orthant. Write
 $\bX = \bZ\bSigma^{1/2}$ so that $\bZ$ is an $n\times p$ matrix with
 i.i.d.~standard Gaussian entries, and $\btheta =
 \bSigma^{1/2}\bbeta$. Then the event (\ref{eq:Xbeta-R++}) is
 equivalent to 
\begin{equation}
\left\{ \bZ\btheta \mid\btheta\in\mathbb{R}^{p}\right\}
\cap\mathbb{R}_{++}^{n}\neq\emptyset.  
  \label{eq:Ztheta-R++}
\end{equation}
Now the probability that (\ref{eq:Ztheta-R++}) occurs is the same as
that
\begin{equation}
\left\{ \bZ\btheta \mid\btheta\in\mathbb{R}^{p}\right\} \cap\mathbb{R}_{+}^{n}\neq\left\{ \bm{0}\right\} 
  \label{eq:Ztheta-R+}
\end{equation}
occurs, where $\mathbb{R}_{+}^{n}$ denotes the non-negative orthant.
From the approximate kinematic formula \cite[Theorem
I]{amelunxen2014living} in the literature on convex geometry, the event
(\ref{eq:Ztheta-R+}) happens with high probability if and only if the
total statistical dimension of the two closed convex cones exceeds the
ambient dimension, i.e.
\begin{equation}
\delta\left(\left\{ \bZ\btheta\mid\btheta\in\mathbb{R}^{p}\right\} \right)+\delta\left(\mathbb{R}_{+}^{n}\right)>n+o(n).\label{eq:stat-dim-PT}
\end{equation}
Here, the statistical dimension of a closed convex cone $\mathcal{K}$ is
defined as 
\begin{equation}
\delta(\mathcal{K}):=\mathbb{E}_{\bm{g}\sim\mathcal{N}(\bm{0},\bm{I})}\left[\|\Pi_{\mathcal{K}}\left(\bm{g}\right)\|^{2}\right]\label{eq:defn-statistical-dimension}
\end{equation}
with $\Pi_{\mathcal{K}}\left(\bm{g}\right):=\arg\min_{\bm{z}\in\mathcal{K}}\|\bm{g}-\bm{z}\|$
the Euclidean projection. Recognizing that\cite[Proposition 2.4]{amelunxen2014living}
{
\[
\delta\left(\left\{ \bZ\btheta\mid\btheta\in\mathbb{R}^{p}\right\} \right)=p\quad\text{and}\quad\delta(\mathbb{R}_{+}^{n})=n/2,
\]
}
we reduce the condition (\ref{eq:stat-dim-PT}) to 
\[
p+n/2>n+o(n)\qquad\text{or}\qquad p/n>1/2+o(1),
\]
thus indicating that $\|\hat{\bm{\beta}}\| = \infty$
when $\kappa = \lim p/n>1/2$. 

This argument only reveals that  $\| \hat{\bm{\beta}} \| = \infty$ in
the regime where $\kappa >1/2$. If $\kappa=p/n<1/2$, then $\| \bm{\Sigma}^{1/2} \hat{\bm{\beta}} \| = O(1)$
with high probability, a fact we shall
prove in Section \ref{sec: MLEbounded}. 
In light of these observations
we work with the additional condition 
\begin{equation}
\label{eq: kappacondition}
	 \kappa < 1/2. 
\end{equation}

\subsection{The high-dimensional limiting distribution of the LLR}

In contrast to the
classical Wilks' result, our findings reveal that the LLR statistic
follows a {\em rescaled} chi-square distribution with a rescaling
factor that can be explicitly pinned down through the solution to a
system of equations.

\subsubsection{A system of equations}
We start by setting up the crucial system of equations. Before proceeding, 
we first recall the proximal operator
\begin{equation}\label{eq:prox}
	\prox_{b\rho}(z) := \arg \min_{x \in \R} \left\{ b\rho(x) + \frac{1}{2}(x-z)^2 \right\}
\end{equation}
defined for any $b>0$ and convex function $\rho(\cdot)$. As in
\cite{donoho2013high}, we introduce the operator
\begin{equation}
  \label{eq:defn-psi}
  \Psi(z;b) := b \rho'(\prox_{b\rho}(z)),
\end{equation}
which is simply the proximal operator of the conjugate $(b\rho)^*$ of $b\rho$.\footnote{The conjugate $f^*$ of a function $f$ is defined as $f^{*} (x) = \sup_{u \in \mathrm{dom}(f)} \{\langle u, x \rangle - f(u) \}$.
} 
To see this, we note that $\Psi$ satisfies the relation \cite[Proposition 6.4]{donoho2013high} 
\begin{equation}\label{eq:proxrelations}
  \Psi(z;b) + \prox_{b\rho}(z) = z.
\end{equation}
The claim that $\Psi(\cdot;b)=\prox_{(b\rho)^*}(\cdot)$ then follows from the Moreau
decomposition
\begin{equation}
\prox_f(z) + \prox_{f^{*}}(z) = z, \qquad \forall z, 
\end{equation}
which holds for a closed convex function $f$ \cite[Section
2.5]{parikh2014proximal}.
Interested readers are referred to \cite[Appendix 1]{donoho2013high} for more properties of $\prox_{b\rho}$ and $\Psi$.

We are now in position to present the system of equations that plays a crucial role in determining the distribution of the LLR statistic in high dimensions:
\begin{align}
\tau^{2} & = \frac{1}{\kappa}\mathbb{E}\left[\left(\Psi\left(\tau Z; \hspace{0.2em} b\right)\right)^{2}\right] ,  \label{eq: sysofeq-tau}\\
\kappa &  = \mathbb{E}\big[\Psi'\left(\tau Z; \hspace{0.2em} b\right)\big] , \label{eq: sysofeq-b}
\end{align}
where $Z \sim \dnorm(0,1)$, and $\Psi'\left(\cdot,\cdot\right)$
denotes differentiation with respect to the first variable.  The fact that this
system of equations would admit a unique solution in $\mathbb{R}_+^2$
is not obvious {\em a priori}. We shall establish this for the logistic and the probit models later in Section \ref{sec:
  AMP}. 
  

\subsubsection{Main result}\label{subsec: mainres}


\begin{theorem}\label{thm: thm1}
  Consider a logistic or probit regression model under the assumptions
  from Section \ref{subsec: assumptions}.  If $\kappa \in (0,1/2)$,
  then for any $1\leq j\leq p$, the log-likelihood ratio statistic
  $\Lambda_j$ as defined in (\ref{eq:LLR-stat}) obeys
\begin{equation}
	2\Lambda_j ~\convd~ \frac{\taus^2}{\bs}\chi_1^2,\qquad \text{as } n\rightarrow \infty,
\end{equation}
where $(\taus, \bs)\in \mathbb{R}_+^2$ is the unique solution to the
system of equations \eqref{eq: sysofeq-tau} and \eqref{eq:
  sysofeq-b}. Furthermore, the LLR statistic obtained by dropping $k$
variables for any fixed $k$ converges to $({\taus^2}/{\bs})\chi_k^2$.
Finally, these results extend to all binary regression models with
links obeying the assumptions listed in Section \ref{sec:rho}. 
\end{theorem}
Hence, the limiting distribution is a rescaled chi-square with a rescaling factor ${\taus^2}/{\bs}$ that only depends on the aspect ratio $\kappa$. 
Fig.~\ref{fig: rescale_cst_logistic} 
illustrates the dependence of the rescaling factor on the limiting
aspect ratio $\kappa$ for logistic regression. The figures for the
probit model are similar as the rescaling constants actually differ by
very small values.

\begin{figure}
\centering
\begin{subfigure}{.5\textwidth}
  \centering
  \includegraphics[scale=.4,keepaspectratio]{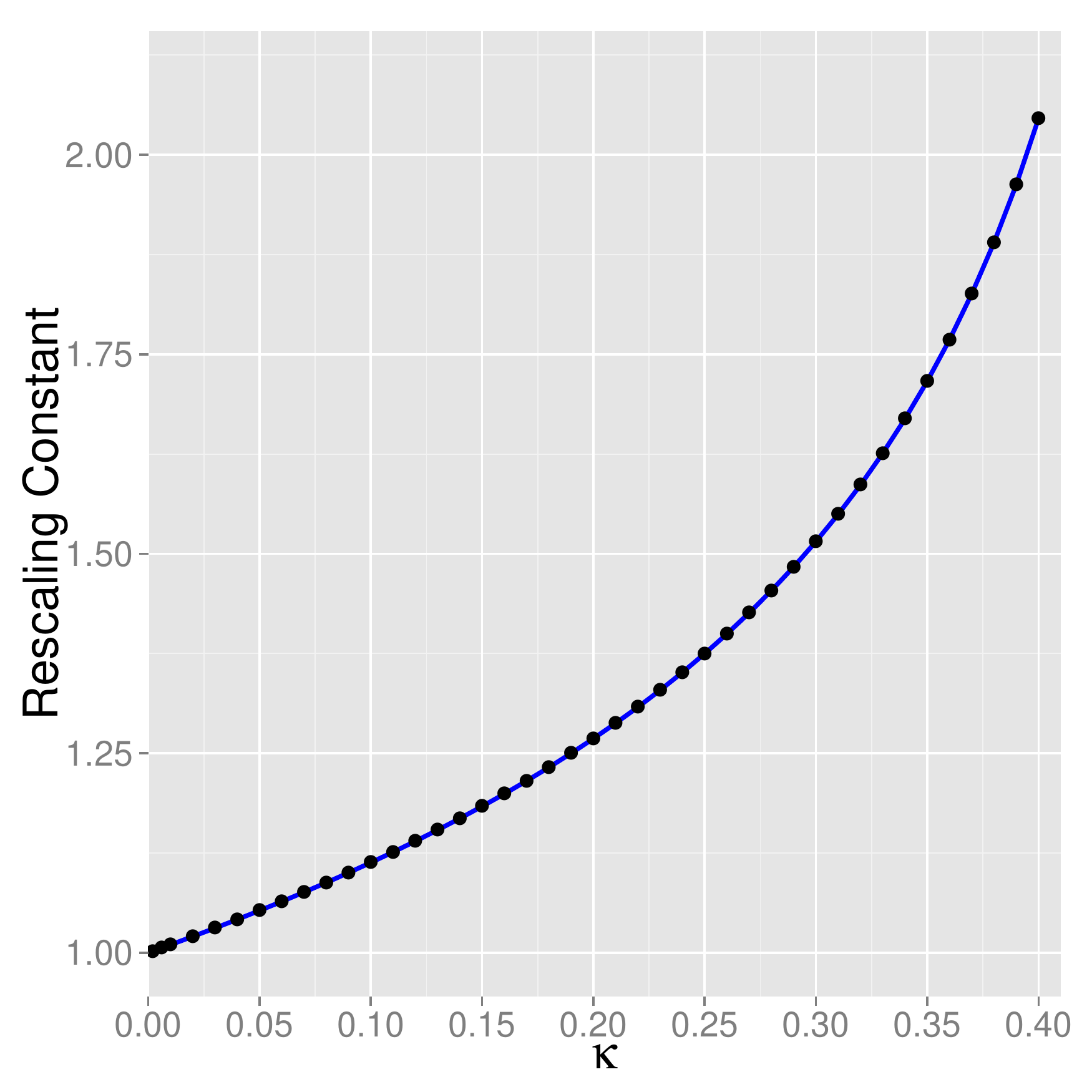}
\end{subfigure}%
\begin{subfigure}{.5\textwidth}
  \centering
  \includegraphics[scale=0.4,keepaspectratio]{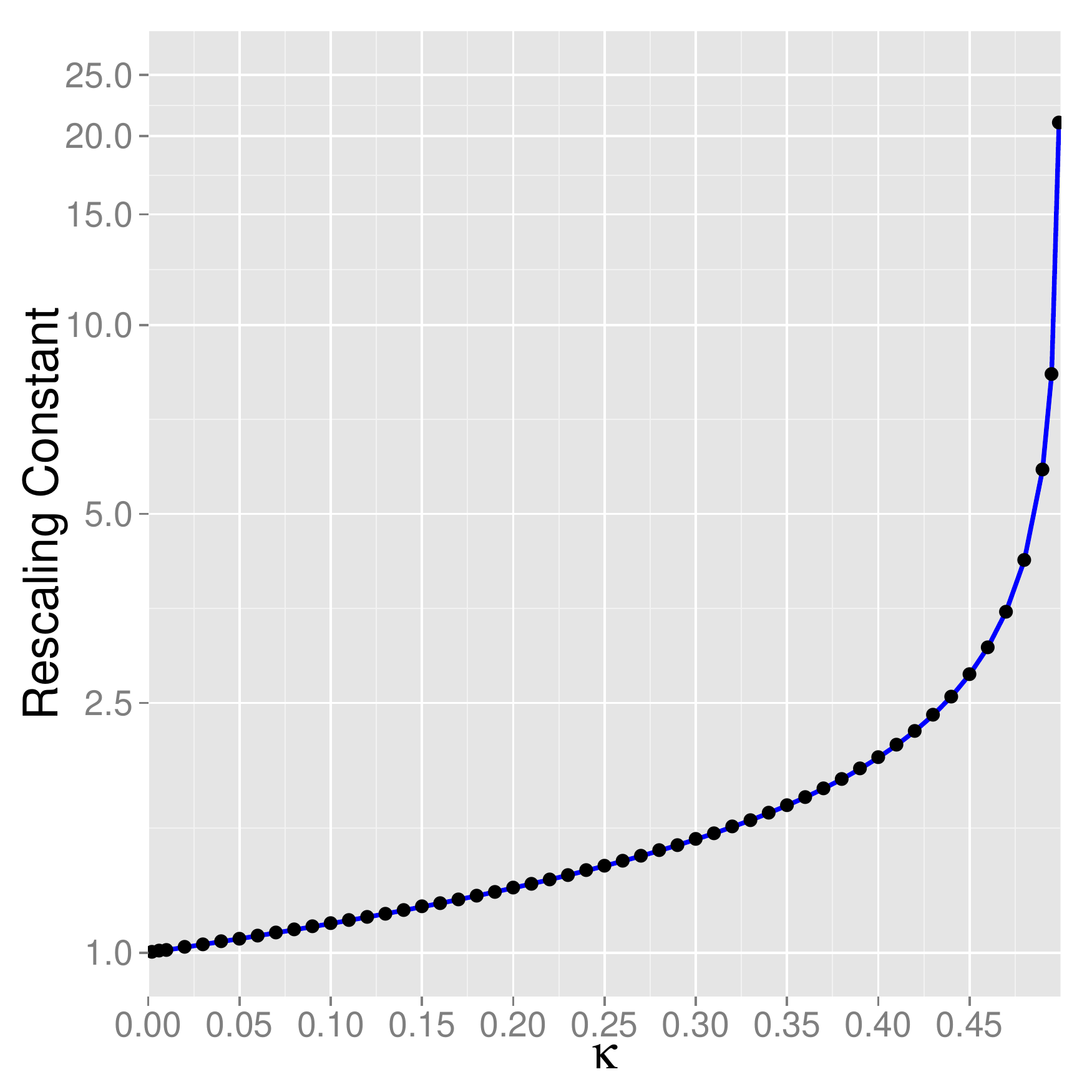}
\end{subfigure}
\caption{Rescaling constant $\taus^2/\bs$ as a function of $\kappa $
  for the logistic model. Note the logarithmic scale in the right
  panel. The curves for the probit model are nearly identical.}
\label{fig: rescale_cst_logistic}
\end{figure}


To study the quality of approximation for finite samples, we
repeat the same numerical experiments as before but now obtain the
p-values by comparing the LLR statistic with the rescaled chi-square
suggested by Theorem \ref{thm: thm1}. For a particular run of the
experiment ($n=4000, p=1200, \kappa =0.3$), we compute the adjusted
LLR statistic $({2\bs}/{\taus^2})\Lambda_j$ for each coefficient 
and obtain the p-values based on the $\chi^2_1$
distribution. The pooled histogram that aggregates $4.8 \times 10^5$ p-values in total is  
shown in
Fig.~\ref{fig: classical}(c).

As we clearly see, the p-values are much closer to a uniform
distribution now. One can compute the chi-square goodness of fit
statistic to test the closeness of the above distribution to
uniformity. To this end, we divide the interval $[0,1]$ into $20$
equally spaced bins of width $0.05$ each. For each bin we compute the
observed number of times a p-value falls in the bin out of the
$4.8 \times 10^5$ values.  Then a chi-square goodness of fit statistic
is computed, noting that the expected frequency is $24000$ for each
bin. The chi-square statistic in this case is $16.049$, which gives a
p-value of $0.654$ in comparison with a $\chi^2_{19}$ variable. The
same test when performed with the Bartlett corrected p-values
(Fig.~\ref{fig: classical}(b)) yields a chi-square statistic $5599$
with a p-value of $0$. \footnote{ Note that the p-values obtained at
  each trial are not exactly independent. However, they are
  exchangeable, and weakly dependent (see the proof of Corollary
  \ref{cor: marginal} for a formal justification of this fact).
  Therefore, we expect the goodness of fit test to be an approximately
  valid procedure in this setting.} Thus, our correction gives the
desired uniformity in the p-values when the true signal
$\bbeta = \bzero$.


Practitioners would be concerned about the validity of p-values when
they are small---again, think about multiple testing applications. In
order to study whether our correction yields valid results for small
p-values, we compute the proportion of times the p-values (in all the
three cases) lie below $5\%, 1\% , 0.5\% , 0.1\%$ out of the $4.8 \times 10^5$
times. The results are summarized in Table \ref{tab:pvalfits}. This
further illustrates the deviation from uniformity for the classical
and Bartlett corrected p-values, whereas the ``adjusted'' p-values
obtained invoking Theorem \ref{thm: thm1} are still valid.

\begin{table}[h!]
\centering
    \begin{tabular} {c|c|c|c}
    \hline
     & Classical & Bartlett-corrected & Adjusted \\ \hline
    $\opP\{\text{p-value}  \leq 5 \% \}$ & $11.1044 \% (0.0668 \%)$  & $6.9592 \% (0.0534 \%)$ &  $5.0110 \% (0.0453 \%)$\\ 
    \hline
        $\opP\{\text{p-value} \leq 1 \% \}$ & $3.6383\% (0.038 \%)$  & $1.6975 \% (0.0261 \%)$  & $0.9944 \% ( 0.0186 \%)$    \\ 
        \hline
        $\opP\{\text{p-value}  \leq 0.5 \% \}$& $2.2477 \% (0.0292 \%)$  & $0.9242 \% (0.0178 \%)$   &  $0.4952 \% (0.0116 \%)$ \\ 
        \hline
        $\opP\{\text{p-value}  \leq 0.1 \% \}$ &  $0.7519 \% (0.0155 \%)$ & $0.2306\%(0.0078\%)$ & $0.1008 \% (0.0051\%)$  \\ \hline
        $\opP\{\text{p-value}  \leq 0.05 \% \}$ &  $0.4669 \% (0.0112 \%)$ & $0.124\%(0.0056\%)$ & $0.0542\% (0.0036\%)$  \\ \hline
        $\opP\{\text{p-value}  \leq 0.01 \% \}$ &  $0.1575 \% (0.0064 \%)$ & $ 0.0342 \%(0.0027\%)$ & $0.0104 \% (0.0014\%)$  \\ \hline
    \end{tabular} \\ 
    \caption{Estimates of p-value probabilities with estimated Monte Carlo standard errors in parentheses under i.i.d.~Gaussian design.}
    \label{tab:pvalfits}   
\end{table}


\subsubsection{Extensions}
\label{sec:rho}

As noted in Section \ref{subsec: probset}, the Symmetry Condition
(\ref{eq: symmetry}) allows to express the negative log-likelihood in
the form \eqref{eq:log-likelihood-alternative}, which makes use of the
effective link $\rho(\cdot)$. Theorem \ref{thm: thm1} applies to any
$\rho(\cdot)$ obeying the following properties: 
\begin{enumerate}
\item $\rho $ is non-negative, has up to three derivatives, and obeys
  $\rho(t) \geq t$.
\item $\rho'$ may be unbounded but it should grow sufficiently slowly, in particular, we assume $|\rho'(t)| = O(|t|)$ and $\rho'(\prox_{c \rho}(Z))$ is a sub-Gaussian random variable for any constant $c>0$ and any $Z \sim \dnorm(0,\sigma^2)$ for some finite $\sigma>0$. 
\item $\rho''(t)>0$ for any $t$ which implies that $\rho$ is convex,
  and $\sup_t \rho''(t) < \infty$.
\item $\sup_t |\rho'''(t)| < \infty$. 
\item Given any $\tau > 0$ ,the equation \eqref{eq: sysofeq-b} has a unique solution in $b$.  
\item The map $\Nu(\tau^2)$ as defined in \eqref{eq: varmapone} has a fixed point. 
\end{enumerate}
It can be checked that the effective links for both the logistic and
the probit models \eqref{eq:link-function-both-models} obey {\em all}
of the above.  The last two conditions are assumed to ensure existence
of a unique solution to the system of equations \eqref{eq:
  sysofeq-tau} and \eqref{eq: sysofeq-b} as will be seen in Section
\ref{sec: AMP}; we shall justify these two conditions for the logistic
and the probit models in Section \ref{subsec: SE}.


{\subsection{Reduction to independent covariates} 

  In order to derive the asymptotic distribution of the LLR
  statistics, it in fact suffices to consider the special case
  $\bSigma = \Id_p$.
\begin{lemma}
  Let $\Lambda_j(\bX)$ be the LLR statistic based on the design matrix $\bX$, where the rows of $\bX$ are
  i.i.d.~$\dnorm({\bm 0}, \bSigma)$ and $\Lambda_j(\bZ)$ that
  where the rows are i.i.d.~$\dnorm({\bm 0}, \Id_p)$. Then
\[
  \Lambda_j(\bX) ~\eqd~ \Lambda_j(\bZ). 
\]
\end{lemma}
\begin{proof}
Recall from \eqref{eq:log-likelihood-alternative} that the LLR statistic for testing the $j$th coefficient can be expressed as
\[
\Lambda_j(\bX) = \min_{\bbeta}\sum_{i=1}^n
\rho(-\ty_i \be_i^\top \bX \bbeta) - \min_{\bbeta: \beta_j =0}\sum_{i=1}^n
 \rho(-\ty_i \be_i^\top \bX\bbeta).  
\]
 Write $\bW' = \bX \bSigma^{-1/2}$ so that the rows of $\bW'$ are
 i.i.d.~$\dnorm({\bm 0}, \Id_p)$ and set
 $ \btheta' = \bSigma^{1/2}\bbeta$. With this reparameterization,
 we observe that the constraint $\beta_j = 0$ is equivalent to $\ba_j^\top \btheta' = 0$
 for some non-zero vector $\ba_j \in \R^p$. This gives
\[
\Lambda_j(\bX) = \min_{\btheta'}\sum_{i=1}^n
\rho(-\ty_i \be_i^\top \bW' \btheta') - \min_{\btheta': \ba_j^{\tp}\btheta' =0}\sum_{i=1}^n
 \rho(-\ty_i \be_i^\top \bW'\btheta'). 
\]
Now let $\bO$ be an orthogonal matrix mapping $\ba_j \in \R^p$ into
the vector $\|\ba_j \|\be_j \in \R^p$,
i.e.~$\bO \ba_j =\|\ba_j \| \be_j$. Additionally, set $\bW = \bW' \bO$ (the rows of
$\bW$ are still i.i.d.~$\dnorm({\bm 0}, \Id_p)$) and 
$\btheta = \bO \btheta'$. Since $\ba_j^\top \btheta' = 0$ occurs if and only
if $\theta_j = 0$, we obtain
\[
\Lambda_j(\bX) = \min_{\btheta}\sum_{i=1}^n
\rho(-\ty_i \be_i^\top \bW \btheta) - \min_{\btheta: \theta_j =0}\sum_{i=1}^n
 \rho(-\ty_i \be_i^\top \bW\btheta) = \Lambda_j(\bZ),  
\]
which proves the lemma.
 \end{proof} 
 In the remainder of the paper we, therefore, assume
 $\bSigma = \Id_p$.

\subsection{Proof architecture}

This section presents the main steps for proving Theorem \ref{thm:
  thm1}. We will only prove the theorem for $\{\Lambda_j\}$, the LLR
statistic obtained by dropping a single variable. The analysis for the
LLR statistic obtained on dropping $k$ variables (for some fixed $k$)
follows very similar steps and is hence omitted for the sake of
conciceness.  As discussed before, we are free to work with any
configuration of the $y_i$'s. For the two steps below, we will adopt
two different configurations for convenience of presentation.

\subsubsection{Step 1: characterizing the asymptotic distributions of $\hat{\beta}_j$}
\label{sub:AMP-step}

Without loss of generality, we assume here that $y_i=1$ (and hence
$\tilde{y}_i=1$) for all $1\leq i\leq n$ and, therefore, the MLE problem
reduces to
\[
  \text{minimize}_{\bm{\beta}\in \mathbb{R}^p} \quad \sum\nolimits_{i=1}^n \rho(-\bm{X}_i^{\top}\bm{\beta}).
\]
We would first like to characterize the marginal distribution of $\hat{\bm{\beta}}$, which is crucial in understanding the LLR statistic. To this end, our analysis follows by a reduction to the setup of \cite{el2013robust,donoho2013high,karoui2013asymptotic,el2015impact}, with certain modifications that are called for due to the specific choices of $\rho(\cdot)$ we deal with here. 
Specifically, consider the linear model
\begin{equation}
	\bm{y} = \bm{X} \bm{\beta} + \bm{w},	\label{eq:linear-regression}
\end{equation}
and prior work \cite{el2013robust,donoho2013high,karoui2013asymptotic,el2015impact} investigating the associated M-estimator
\begin{equation}
	\text{minimize}_{\bm{\beta}\in \mathbb{R}^p} \quad \sum\nolimits_{i=1}^n \rho(y_i - \bm{X}_i^{\top} \bm{\beta}). 	\label{eq:M-estimation}
\end{equation}
Our problem  reduces to (\ref{eq:M-estimation})  on setting $\bm{y}=\bm{w}=\bm{0}$  in (\ref{eq:M-estimation}). When $\rho(\cdot)$ satisfies certain assumptions (e.g.~strong convexity), the asymptotic distribution of $\| \hat{\bm{\beta}} \|$ has been studied in a series of works \cite{el2013robust,karoui2013asymptotic,el2015impact} using a leave-one-out analysis and independently in \cite{donoho2013high} using approximate message passing (AMP) machinery. An outline of their main results is described in Section \ref{subsec: priorart}. 
However, the function $\rho(\cdot)$ in our cases has vanishing
curvature and, therefore, lacks the essential strong convexity
assumption that was utilized in both the aforementioned lines of work. To
circumvent this issue, we propose to invoke the AMP machinery as in
\cite{donoho2013high}, in conjunction with the following critical
additional ingredients:
\begin{itemize}
	\item (\textit{Norm Bound Condition}) We utilize results from the conic geometry literature (e.g.~\cite{amelunxen2014living}) to establish that $$\| \hat{\bm{\beta}} \| = O(1)$$ with high probability as long as $\kappa < 1/2$. This will be elaborated in Theorem \ref{thm: normbound-MLE}.  
  \item (\textit{Likelihood Curvature Condition}) We establish some regularity conditions on the Hessian of the log-likelihood function, generalizing the strong convexity condition, which will be detailed in Lemma \ref{lem:min-eigenvalue-bound}. 
  \item (\textit{Uniqueness of the Solution to \eqref{eq: sysofeq-tau} and \eqref{eq: sysofeq-b}}) We establish that for both the logistic and the probit case, the system of equations \eqref{eq: sysofeq-tau} and \eqref{eq: sysofeq-b} admits a unique solution. 
\end{itemize}
We emphasize that these elements are not straightforward, require significant effort and a number of novel ideas,  
which form our primary technical contributions for this step. 

These ingredients enable the use of the AMP machinery even in the absence of strong convexity on $\rho(\cdot)$, finally leading to the following theorem: 
\begin{theorem}\label{thm: MLE-norm}
Under the conditions of Theorem \ref{thm: thm1}, 
\begin{equation}
\lim_{n\rightarrow\infty} \|\hat{\bm{\beta}}\|^{2}=_{\mathrm{a.s.}} {\tau_{*}^{2}}.
\end{equation}
\end{theorem}

This theorem immediately implies that the marginal distribution of $\hat{{\beta}}_j$ is normal. 

\begin{corollary}\label{cor: marginal}
Under the conditions of Theorem \ref{thm: thm1}, for every $1\leq j\leq p$, it holds that 
\begin{equation}
  \sqrt{p}\hat{\beta}_j ~\convd~ \dnorm(0,\taus^2), \qquad \text{ as } n\rightarrow \infty.
\end{equation}
\end{corollary}
\begin{proof} From the rotational invariance of our i.i.d.~Gaussian
  design, it can be easily verified that $\hbbeta / \|\hbbeta \|$ is
  uniformly distributed on the unit sphere $\spS^{p-1}$ and is independent
  of $\|\hbbeta \|$.  Therefore, $\hat{\beta}_j$ has the same
  distribution as $\|\hbbeta \| Z_j/\|\bZ \|$, where
  $\bZ=(Z_1, \hdots, Z_p) \sim \dnorm(\bm{0},\bm{I}_p)$ independent of
  $\|\hbbeta \|$. Since $ \sqrt{p} \|\hbbeta \|/\|\bZ \|$ converges in
  probability to $\taus$, we have, by Slutsky's theorem, that
  $\sqrt{p}\hat{\beta}_j$ converges to $\dnorm(0, \taus^2)$ in
  distribution.
\end{proof}

\subsubsection{Step 2: connecting $\Lambda_j$ with $\hat{\beta}_j$}
\label{sub:step2-LRT}

Now that we have derived the asymptotic distribution of $\hat{\beta}_j$,  the next step involves a reduction of the LLR statistic to a function of the relevant coordinate of the MLE. Before continuing, we note that the distribution of $\Lambda_j$  is the  same for all $1\leq j\leq p$ due to exchangeability. As a result, going forward we will only analyze $\Lambda_1$ without loss of generality.  In addition, we introduce the following convenient notations and assumptions:
\begin{itemize}
\item the design matrix on dropping the first column is written as
  $\tbX$ and the MLE in the corresponding reduced model as $\tbbeta$;
	\item write $\bX=[\bX_1, \cdots, \bX_n]^{\top} \in \mathbb{R}^{n\times p}$ and $\tbX=[\tbX_1, \cdots, \tbX_n]^{\top} \in \mathbb{R}^{n\times (p-1)}$;
	\item without loss of generality, assume that $\tilde{y}_i=-1$ for all $i$ in this subsection, and hence the MLEs under the full and the reduced models reduce to 
\begin{eqnarray}
	&\hat{\bm{\beta}} = \arg\min_{\bm{\beta}\in \mathbb{R}^p} &\ell(\bm{\beta}) := \sum\nolimits_{i=1}^n \rho(\bX_i^{\top}\bm{\beta}), \\
	&\tilde{\bm{\beta}} = \arg\min_{\bm{\beta}\in \mathbb{R}^{p-1}}  &\tilde{\ell}(\bm{\beta}) :=  \sum\nolimits_{i=1}^n \rho(\tbX_i^{\top}\bm{\beta}). \label{eq:defn-tbeta}
\end{eqnarray}
\end{itemize}
   With the above notations in place, the LLR statistic for testing $\beta_1=0$ vs.~$\beta_1 \neq 0$ can be expressed as
\begin{equation} 
\label{eq: LRT}
\Lambda_1 ~:=~ \tilde{\ell}(\tbbeta)- \ell(\hbbeta) 
 ~=~ \sum_{i=1}^n \left\{ \rho (\tbX_i^ \tp \tbbeta)-\rho(\bX_i^ \tp \hbbeta)  \right\} . 
\end{equation}
To analyze $\Lambda_1$, we invoke Taylor expansion to reach
\begin{multline}
\qquad\qquad\qquad \Lambda_1  ~=~ \underset{:=Q_{\mathrm{lin}}}{\underbrace{\sum_{i=1}^n\rho'\left(\bX_i ^ \tp \hbbeta\right) \left(\tbX_i^ \tp \tbbeta - \bX_i ^ \tp\hbbeta\right)}} +  \frac{1}{2} \sum_{i=1}^n\rho''\left(\bX_i  ^ \tp \hbbeta\right) \left(\tbX_i ^ \tp \tbbeta - \bX_i ^ \tp \hbbeta\right)^2 \\
 + \frac{1}{6} \sum_{i=1}^n \rho'''(\gamma_i) \left(\tbX_i ^ \tp \tbbeta - \bX_i ^ \tp \hbbeta \right)^3, \qquad\qquad\qquad\qquad
\label{eq:L-4term}
\end{multline}
where $\gamma_i$ lies between $\tbX^{\tp}_i\tbbeta$ and $\bX^{\tp}_i\hbbeta$.
A key observation is that the linear term $Q_{\mathrm{lin}}$ in the above equation vanishes. To see this, note that the first-order optimality conditions for the MLE $\hat{\bbeta}$ 
is given by
%
\begin{align}
\label{eq: betaesteq}
\sum\nolimits_{i=1}^n  \rho'(\bX_i ^ \tp \hat{\bbeta}) \bX_i = \bm{0}. 
\end{align}
Replacing $\tilde{\bm{X}}_{i}^{\top}\tilde{\bm{\beta}}$ with $\bm{X}_{i}^{\top}\left[\begin{array}{c}
0\\
\tilde{\bm{\beta}}
\end{array}\right]$ in $Q_{\text{lin}}$ and using the optimality condition, we obtain 
\begin{eqnarray*}
	Q_{\text{lin}} & = & \left(\sum\nolimits_{i=1}^{n}\rho'\left(\bm{X}_{i}^{\top}\hat{\bm{\beta}}\right)\bm{X}_{i}\right)^{\top} \left( \left[\begin{array}{c}
0\\
\tilde{\bm{\beta}}
	\end{array}\right]- \hat{\bm{\beta}} \right) ~=~ 0.
\end{eqnarray*}
Consequently, $\Lambda_1$ simplifies to the following form
\begin{equation}
 \Lambda_1
 =~\frac{1}{2} \sum_{i=1}^n\rho''(\bX_i ^ \tp \hbbeta) \left(\tbX_i ^ \tp \tbbeta - \bX_i ^ \tp \hbbeta \right)^2 + \frac{1}{6} \sum_{i=1}^n \rho'''(\gamma_i) \left(\tbX_i ^ \tp \tbbeta - \bX_i ^ \tp \hbbeta\right)^3.  
 \label{eq:LRT-simplified}
\end{equation}
Thus, computing the asymptotic distribution of
$\Lambda_1$ boils down to analyzing $\bX_i^{\tp}\hbbeta -
\tbX_i^{\tp}\tbbeta$. Our argument is inspired by the
leave-one-predictor-out approach developed in
\cite{karoui2013asymptotic,el2015impact}.

We re-emphasize that our setting is not covered by that of
\cite{karoui2013asymptotic,el2015impact}, due to the violation of
strong convexity and some other technical assumptions. We sidestep
this issue by utilizing the \textit{Norm Bound Condition} and the
\textit{Likelihood Curvature Condition}. In the end, our analysis
establishes the equivalence of $\Lambda_1$ and
$\hat{\beta}_1$ up to some explicit multiplicative factors modulo negligible
error terms. This is summarized as follows.

\begin{theorem}
	\label{thm: Lambda-marginal}
Under the assumptions of Theorem \ref{thm: thm1}, 
\begin{equation}
	2\Lambda_1 - \frac{p }{\bs} \hat{\beta}_1^2 ~\convP~ 0 , \qquad \text{ as } n\rightarrow \infty.
\end{equation}
\end{theorem}
Theorem \ref{thm: Lambda-marginal} reveals a simple yet surprising connection between the LLR statistic $\Lambda_1$ and the MLE $\hbbeta$.  As we shall see in the proof of the theorem, 
the quadratic term in \eqref{eq:LRT-simplified} is $\frac{1}{2}\frac{p }{\bs} \hat{\beta}_1^2 + o(1)$, while the remaining third-order term of \eqref{eq:LRT-simplified} is vanishingly small. Finally, putting Corollary \ref{cor: marginal} and Theorem  \ref{thm: Lambda-marginal} together directly establishes Theorem \ref{thm: thm1}. 

\subsection{Comparisons with the classical regime}

We pause to shed some light on the interpretation of the correction 
factor $\taus^2/\bs$ in Theorem \ref{thm: thm1} and understand the
differences from classical results.  Classical theory (e.g.~\cite{huber1973robust,huber2011robust}) asserts that
when $p$ is fixed and $n$ diverges, the MLE for a fixed design $\bm{X}$ is asymptotically normal,
namely,
\begin{equation}
	\sqrt{n}(\hbbeta - \bbeta) ~\convd~ \dnorm(0,\mathcal{I}_{\bbeta}^{-1}), \label{eq:asymp-normality}
\end{equation}
where 
\begin{equation}
	\mathcal{I}_{\bbeta} = \frac{1}{n} \bX^{\tp} \bm{D}_{{\bbeta}}\bX \qquad \text{with}\quad
	\bm{D}_{{\bbeta}} := \left[\begin{array}{ccc}
\rho''\left(\bX_{1}^{\top}{\bbeta}\right)\\
 & \ddots\\
 &  & \rho''\left(\bX_{n}^{\top}{{\bbeta}}\right)
\end{array}\right]
\end{equation}
is the normalized Fisher information at the true value
$\bbeta$.  
In particular, under the global null and i.i.d.~Gaussian design, this
converges to 
\[
	\mathbb{E}_{\bm{X}}[\mathcal{I}_{\bbeta}] = \begin{cases} \frac{1}{4}\Id, \quad & \text{for the logistic model} \\  \frac{2}{\pi} \Id, & \text{for the probit model} 	\end{cases}
\]  
as $n$ tends to infinity \cite[Example 5.40]{van2000asymptotic}. 

The behavior in high dimensions is different. In particular,
Corollary \ref{cor: marginal} states that under the global null, we
have 
\begin{equation}
	\sqrt{p} (\hat{\beta}_j - \beta_j) ~\convd~ \dnorm(0, \taus^2).
\end{equation}
Comparing the variances in the logistic model, we have that
\[
	\lim_{n\rightarrow \infty} \mathsf{Var} \left( \sqrt{p} \hat \beta_j \right) =
	\begin{cases} 4\kappa, \qquad & \text{in classical large-sample theory}; \\
		\taus^2, & \text{in high dimensions}.
	\end{cases}
\]
Fig.~\ref{fig: taus_by_kappa} illustrates the behavior of the ratio
$\taus^2/\kappa $ as a function of $\kappa$. Two observations are
immediate:
\begin{itemize}
	\item First, in Fig.~\ref{fig: taus_by_kappa}(a) we have $\taus^2 \ge 4\kappa$ for all $\kappa \geq 0$. This indicates an inflation in variance
  or an ``extra Gaussian noise'' component that appears in high dimensions, as discussed in
  \cite{donoho2013high}. The variance of the ``extra Gaussian
  noise'' component increases as $\kappa$ grows. 
\item Second, as $\kappa \rightarrow 0$, we have
  $\taus^2/4\kappa \rightarrow 1$ in the logistic model, which indicates that classical
  theory becomes accurate in this case. In other words, our theory recovers the
  classical prediction in the regime where $p = o(n)$. 
\end{itemize}
 

Further, for the testing problem considered here, the LLR statistic in
the classical setup can be expressed, through Taylor expansion, as
\begin{equation}\label{eq: classicalLRT}
  2\Lambda_1 = n (\hbbeta-\tbbeta)^{\tp} \mathcal{I}_{\bbeta} (\hbbeta - \tbbeta) + o_P(1), 
\end{equation}
where $\tbbeta$ is defined in (\ref{eq:defn-tbeta}). 
In the high-dimensional setting, we will also establish a quadratic
approximation of the form 
\[
2\Lambda_1 = n (\hbbeta-\tbbeta)^{\tp} \bG 
(\hbbeta - \tbbeta) + o_P(1), \qquad \bG= \frac{1}{n} \bX^{\tp}\Dhbeta \bX. 
\]
In Theorem \ref{thm:mainthm}, we shall see that $\bs$ is the limit of
$\frac{1}{n}\tr(\bG^{-1})$, the Stieltjes transform of the empirical
spectral distribution of $\bG$ evaluated at $0$. Thus, this quantity
in some sense captures the spread in the eigenvalues of $\bG$ one
would expect to happen in high dimensions. 


\begin{figure}
\centering
\begin{subfigure}{.5\textwidth}
  \centering
  \includegraphics[scale=.4,keepaspectratio]{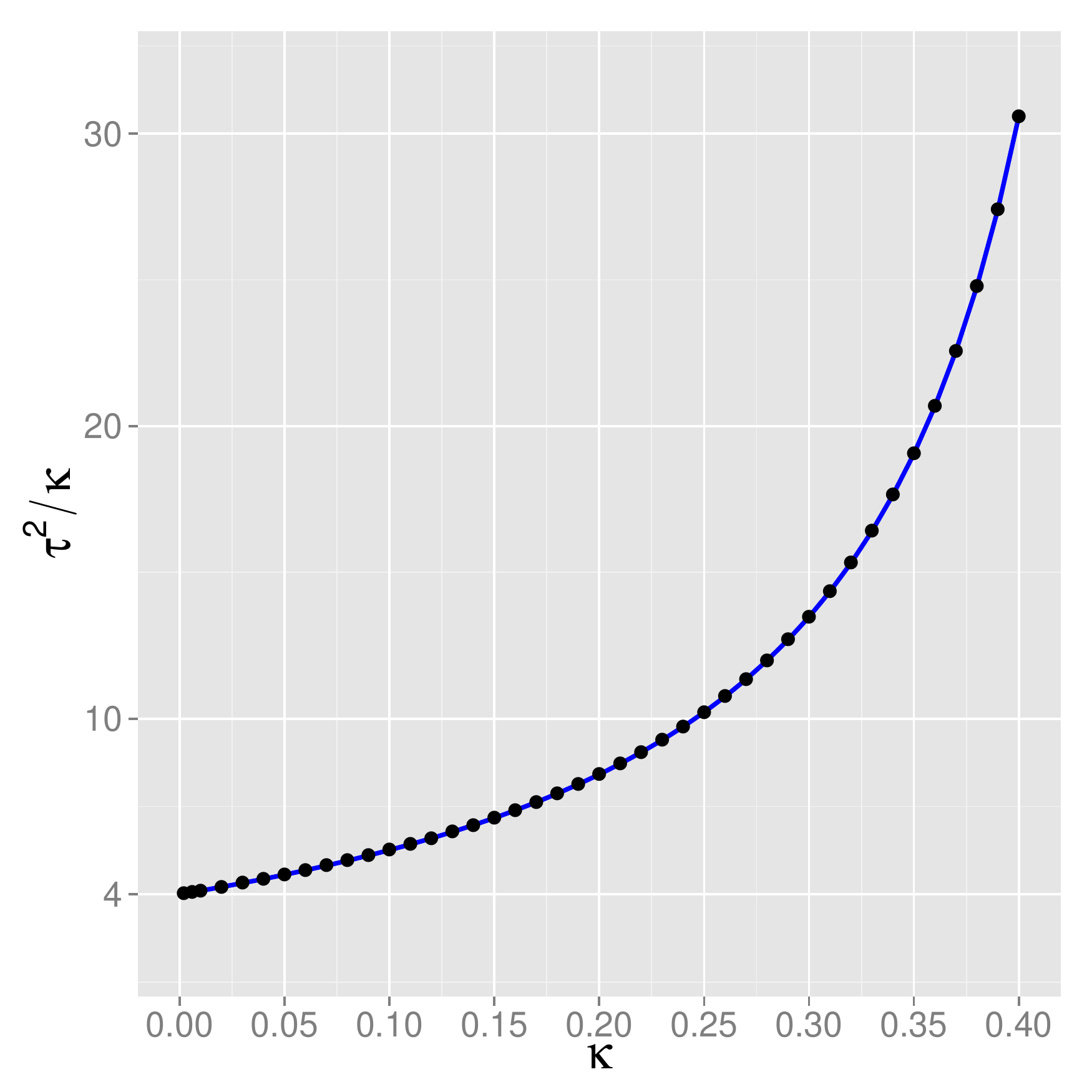}
  \caption{logistic regression}
\end{subfigure}%
\begin{subfigure}{.5\textwidth}
  \centering
  \includegraphics[scale=0.4,keepaspectratio]{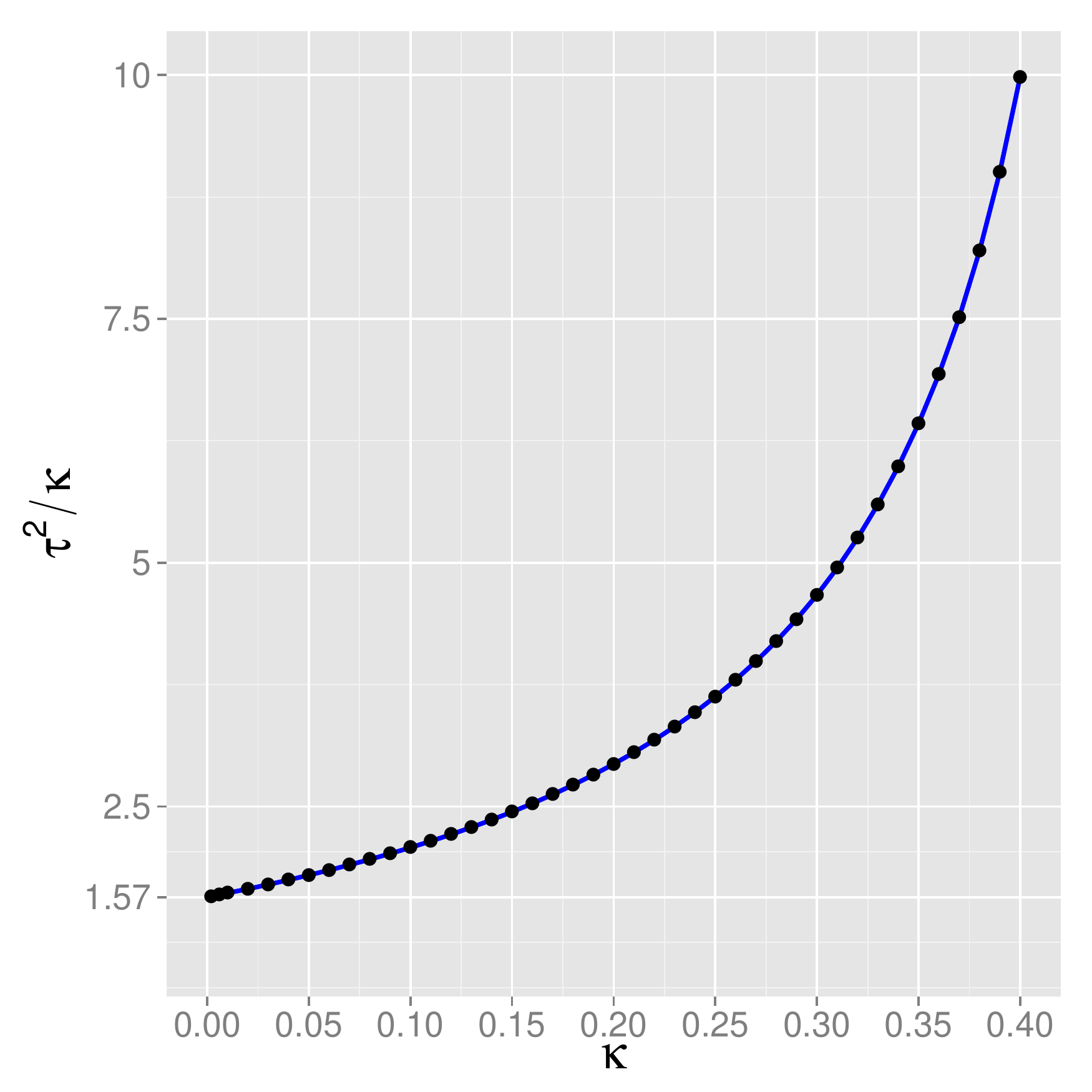}
 \caption{probit regression}
\end{subfigure}
\caption{Ratio of asymptotic variance and dimensionality factor
  $\kappa$ as a function of $\kappa$.}
\label{fig: taus_by_kappa}
\end{figure}


\subsection{Prior art}\label{subsec: priorart}

Wilks' type of phenomenon in the presence of a diverging dimension $p$
has received much attention in the past. For instance, Portnoy
\cite{portnoy1988asymptotic} investigated simple hypotheses in regular
exponential families, and established the asymptotic chi-square
approximation for the LLR test statistic as long as
$p^{3/2}/n \rightarrow 0$. This phenomenon was later extended in 
\cite{spokoiny2012penalized}  to accommodate the MLE with a
quadratic penalization, and in \cite{yan2012high} to account for
parametric models underlying several random graph models. Going beyond
parametric inference, Fan et
al.~\cite{fan2001generalized,fan2007nonparametric} explored extensions
to infinite-dimensional non-parametric inference problems, for which
the MLE might not even exist or might be difficult to derive. While
the classical Wilks' phenomenon fails to hold in such settings, Fan et
al.~\cite{fan2001generalized,fan2007nonparametric} proposed a
generalization of the likelihood ratio statistics based on suitable
non-parametric estimators and characterized the asymptotic
distributions. Such results have further motivated Boucheron and
Massart \cite{boucheron2011high} to investigate the non-asymptotic
Wilks' phenomenon or, more precisely, the concentration behavior of the
difference between the excess empirical risk and the true risk, from a
statistical learning theory perspective. The Wilks' phenomenon for
penalized empirical likelihood has also been established
\cite{tang2010penalized}. However, the precise asymptotic behavior of
the LLR statistic in the regime that permits $p$ to grow proportional
to $n$ is still beyond reach.


On the other hand, as demonstrated in Section \ref{sub:AMP-step}, the MLE here under the global null can be viewed as an M-estimator for a linear regression problem. Questions regarding the behavior of robust linear regression estimators in high dimensions---where $p$ is allowed to grow with $n$—--were raised in Huber \cite{huber1973robust}, and have been
 extensively studied in subsequent works,
 e.g.~\cite{portnoy1984asymptotic,portnoy1985asymptotic,portnoy1986asymptotic,mammen1989asymptotics}. When
 it comes to logistic regression, the behavior of the MLE was studied
 for a diverging number of parameters by \cite{he2000parameters},
 which characterized the squared estimation error of the MLE if $(p
 \log p)/n \rightarrow 0$. In addition, the asymptotic normality properties of
 the MLE and the penalized MLE for logistic regression have been established by \cite{liang2012maximum} and \cite{fan2011nonconcave}, respectively. 
A very recent paper by Fan et al.~\cite{fan2017nonuniform} studied the logistic model under the global null $\bm{\beta} = \bm{0}$, and investigated the classical asymptotic normality  as given in (\ref{eq:asymp-normality}). It was discovered in \cite{fan2017nonuniform} that the convergence property (\ref{eq:asymp-normality}) breaks down even in terms of the marginal distribution,  
namely, 
$$ \frac{ \sqrt{n}  \hat{{\beta}}_i }{ \big( \mathcal{I}_{\bm{\beta}} \big)^{-1/2}_{i,i} } ~\overset{\mathrm{d}}{\nrightarrow}~ \mathcal{N}\left({0}, 1 \right), \qquad  \mathcal{I}_{\bm{\beta}} = \frac{1}{4n}\bm{X}^{\top}\bm{X},  $$
as soon as $p$ grows at a rate exceeding $n^{2/3}$. 
This result, however, does not imply the asymptotic distribution of the likelihood-ratio statistic in this regime. In fact, our theorem implies that LLR statistic $2\Lambda_j$ goes to $\chi_1^2$ (and hence Wilks phenomenon remains valid) when $\kappa = p/n\rightarrow 0$. 

The line of work that is most relevant to the present paper was initially started by El Karoui et al.~\cite{el2013robust}. Focusing on the regime where $p$ is comparable to $n$, the authors uncovered, via a non-rigorous argument, that the asymptotic $\ell_2$ error of the MLE could be characterized by a system of nonlinear equations. This seminal result was later made rigorous independently by Donoho et al.~\cite{donoho2013high} under i.i.d.~Gaussian design and by El Karoui \cite{karoui2013asymptotic,el2015impact} under more general i.i.d.~random design as well as certain assumptions on the error distribution. Both approaches rely on strong convexity on the function $\rho(\cdot)$ that defines the M-estimator, which does not hold in the models considered herein.

\subsection{Notations}
%
We adopt the standard notation $f(n)= O \left(g(n)\right)$ or
    $f(n)\lesssim g(n)$ which  means that there exists a constant $c>0$ such
    that $\left|f(n)\right|\leq c|g(n)|$. Likewise, 
    $f(n)=\Omega\left(g(n)\right)$ or $f(n)\gtrsim g(n)$ means that
    there exists a constant $c>0$ such that
    $|f(n)|\geq c\left|g(n)\right|$, $f(n)\asymp g(n)$ means that
    there exist constants $c_{1},c_{2}>0$ such that
    $c_{1}|g(n)|\leq|f(n)|\leq c_{2}|g(n)|$, and $f(n)= o( g(n) )$
    means that $\lim_{n\rightarrow \infty} \frac{f(n)}{g(n)} = 0$.
    Any mention of $C$, $C_i$, $c$, $c_i$ for $i \in \mathbb{N}$
    refers to some positive universal constants whose value may change
    from line to line. For a square symmetric matrix $\bM$, the
    minimum eigenvalue is denoted by $\lambda_{\min}(\bM)$.
    Logarithms are base $e$.


%% file: numerics.tex
\section{Numerics}\label{sec: numerics}


\begin{figure}
\begin{center}
	\begin{tabular}{ccc}
	\includegraphics[scale=0.28,keepaspectratio]{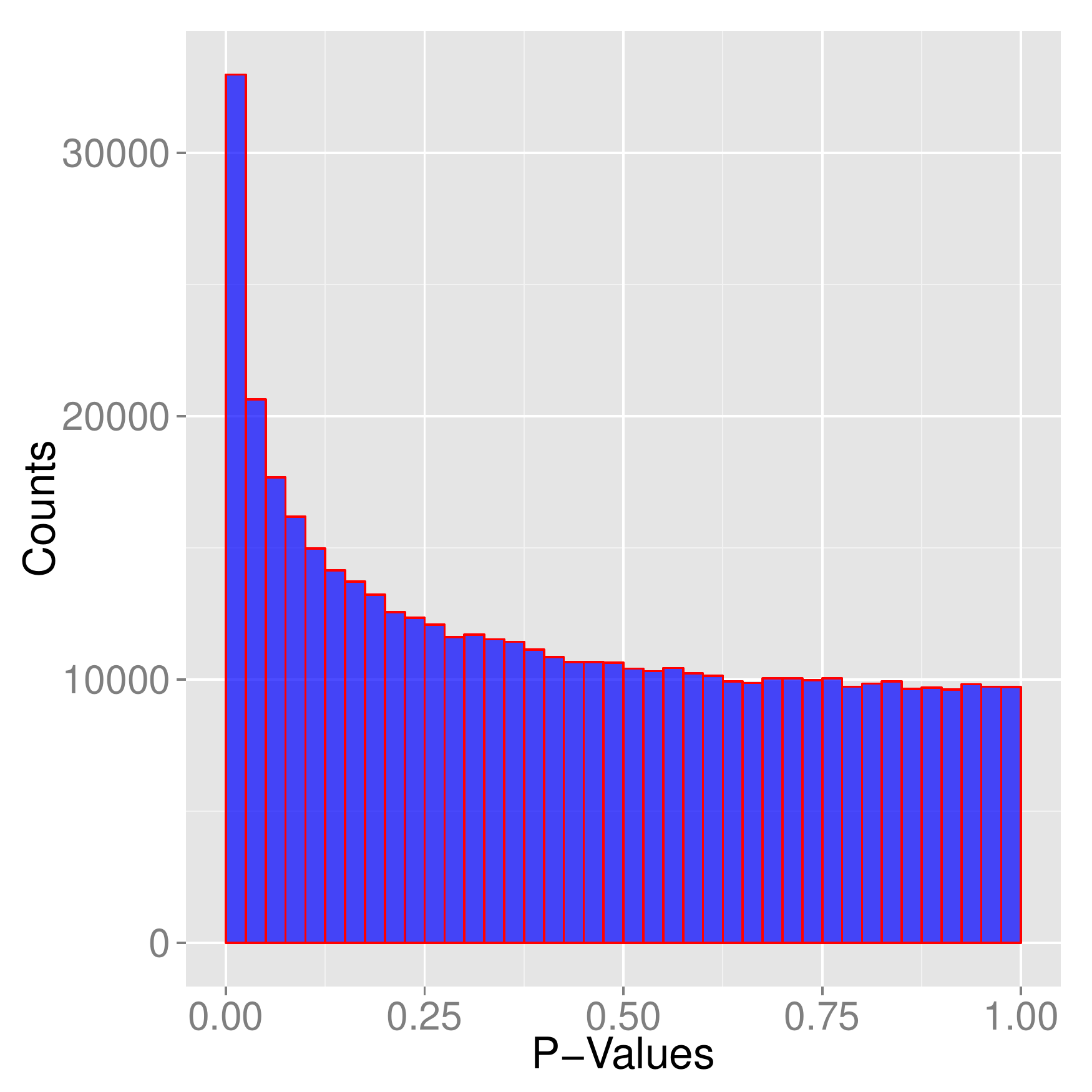}  
		& \includegraphics[scale=0.28,keepaspectratio]{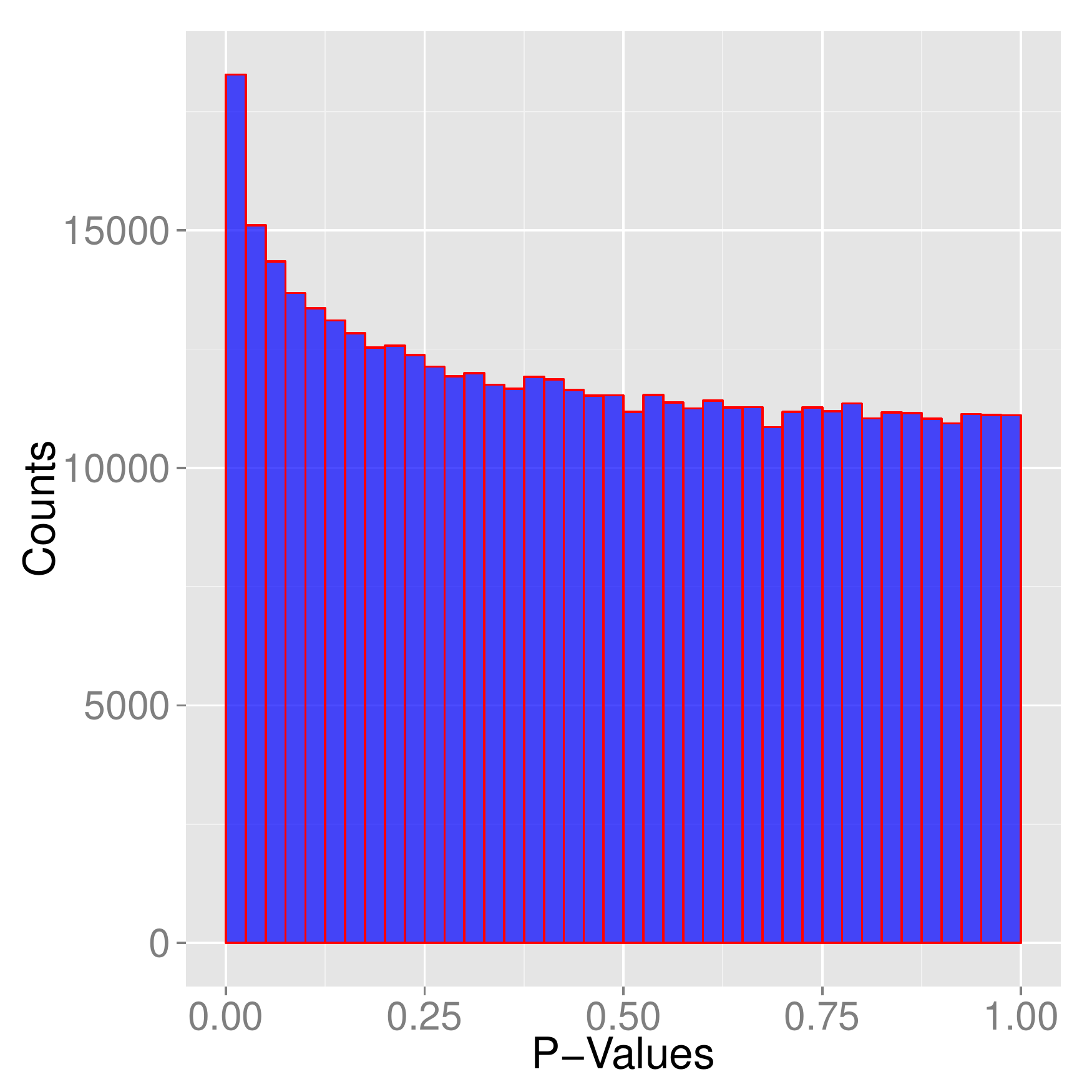}
		& \includegraphics[scale=0.28,keepaspectratio]{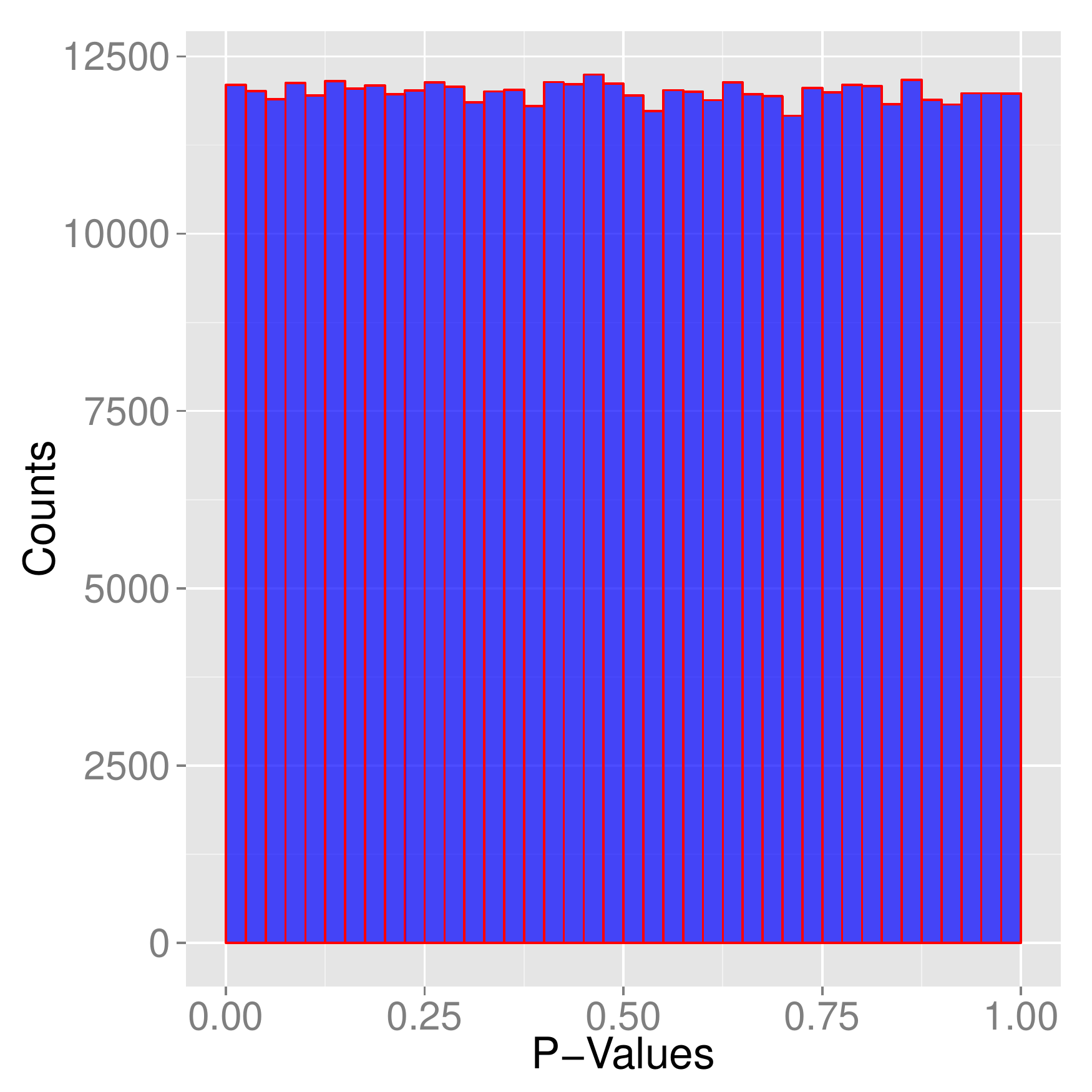} \tabularnewline
		(a) & (b) & (c) \tabularnewline
	\end{tabular}
\end{center}
\caption{Histogram of p-values for logistic regression under i.i.d.~Bernoulli design, when $\bbeta=\bzero$, $n=4000$, $p=1200$, and
	$\kappa = 0.3$: (a) classically computed p-values; (b) Bartlett corrected p-values; (c) adjusted p-values. 
  }\label{fig: bernoulli}
\end{figure}

\subsection{Non-Gaussian covariates}

In this section we first study the sensitivity of our result to the
Gaussianity assumption on the design matrix. To this end, we consider
a high dimensional binary regression set up with a Bernoulli design
matrix. We simulate $n=4000$ i.i.d.~observations $(y_i,\bX_i)$ with
$y_i \iid \dber(1/2)$, and $\bX_i$ generated independent of $y_i$,
such that each entry takes on values in $\{1,-1 \}$ w.p.~$1/2$. At
each trial, we fit a logistic regression model to the data and obtain
the classical, Bartlett corrected and adjusted p-values (using the
rescaling factor $\taus^2/\bs$).  Figure \ref{fig: bernoulli} plots
the histograms for the pooled p-values, obtained across $400$ trials.

It is instructive to compare the histograms to that obtained in the Gaussian case (Figure \ref{fig: classical}). The classical and Bartlett corrected p-values exhibit similar deviations from uniformity as in the Gaussian design case, whereas our adjusted p-values continue to have an approximate uniform distribution. We test for deviations from uniformity using a formal chi-squared goodness of fit test as in Section \ref{subsec: mainres}. For the Bartlett corrected p-values, 
 the chi-squared statistic turns out to be $5885$, with a p-value  $0$. For the adjusted p-values,the chi-squared statistic is $24.1024$, with a p-value  $0.1922$.\footnote{Recall our earlier footnote about the use of a $\chi^2$ test.}
 
 Once again, the Bartlett correction fails to provide valid p-values
 whereas the adjusted p-values are consistent with a uniform
 distribution. These findings indicate that the distribution of the
 LLR statistic under the i.i.d.~Bernoulli design is in agreement to
 the rescaled $\chi^2_1$ derived under the Gaussian design in Theorem
 \ref{thm: thm1}, suggesting that the distribution is not too
 sensitive to the Gaussianity assumption. Estimates of p-value probabilities for our method are provided in Table \ref{tab:pvalfits_bernoulli}. 
 
 \begin{table}[h!]
\centering
    \begin{tabular} {c|c}
    \hline
     & Adjusted \\ \hline
    $\opP\{\text{p-value}  \leq 5 \% \}$ &  $5.0222 \% (0.0412 \%)$\\ 
    \hline
        $\opP\{\text{p-value} \leq 1 \% \}$ &  $1.0048\% ( 0.0174 \%)$    \\ 
        \hline
        $\opP\{\text{p-value}  \leq 0.5 \% \}$&   $0.5123 \% (0.0119 \%)$ \\ 
        \hline
        $\opP\{\text{p-value}  \leq 0.1 \% \}$ &  $0.1108 \% (0.005\%)$  \\ \hline
        $\opP\{\text{p-value}  \leq 0.05 \% \}$ & $0.0521\% (0.0033\%)$  \\ \hline
        $\opP\{\text{p-value}  \leq 0.01 \% \}$ &   $0.0102 \% (0.0015\%)$  \\ \hline
    \end{tabular} \\ 
    \caption{Estimates of p-value probabilities with estimated Monte Carlo standard errors in parentheses under i.i.d.~Bernoulli design.}
    \label{tab:pvalfits_bernoulli}   
  \end{table}

  \subsection{Quality of approximations for finite sample sizes}
    
In the rest of this section, we report some numerical experiments which study the applicability of our theory in finite sample setups.
    
{\bf Validity of tail approximation} The first experiment explores the
efficacy of our correction for extremely small p-values.  This is
particularly important in the context of multiple comparisons, where
practitioners care about the validity of exceedingly small
p-values. To this end, the empirical cumulative distribution of the
adjusted p-values is estimated under a standard Gaussian design with
$n=4000$, $p=1200$ and $4.8 \times 10^5$ p-values. The range
$[0.1/p, 12/p]$ is divided into points which are equi-spaced with a
distance of $1/p$ between any two consecutive points. The estimated
empirical CDF at each of these points is represented in blue in Figure
\ref{fig: ecdf1}. The estimated CDF is in near-perfect agreement with
the diagonal, suggesting that the adjusted p-values computed using the
rescaled chi-square distribution are remarkably close to a uniform,
even when we zoom in at very small resolutions as would be the case
when applying Bonferroni-style corrections.

{\bf Moderate sample sizes} The final experiment studies the accuracy
of our asymptotic result for moderately large samples. This is
especially relevant for applications where the sample sizes are not
too large. We repeat our numerical experiments with $n=200$, $p=60$
for i.i.d.~Gaussian design, and $4.8 \times 10^5$ p-values. The
empirical CDF for these p-values are estimated and Figure \ref{fig:
  ecdf2} shows that the adjusted p-values are nearly uniformly
distributed even for moderate sample sizes such as $n=200$.

 \begin{figure}
\centering
  \includegraphics[scale=0.4,keepaspectratio]{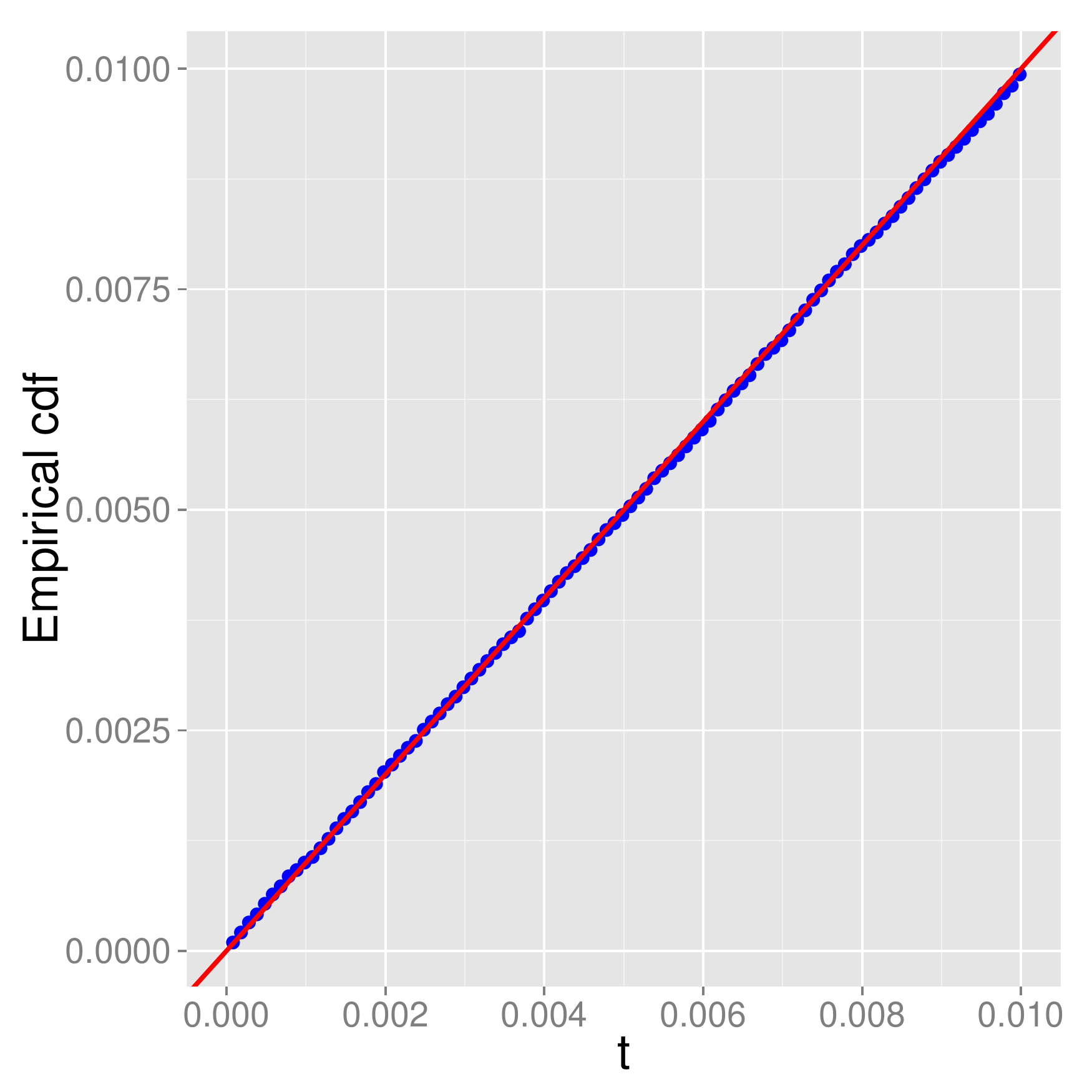}
	\caption{Empirical CDF of adjusted pvalues for logistic regression when $\bbeta=\bzero$, $n=4000$, $p=1200$. Here, the blue points represent the empirical CDF ($t$ vs.~the fraction of p-values below $t$), and the red line is the diagonal.  }
\label{fig: ecdf1}
\end{figure}

\begin{figure}
\centering
  \includegraphics[scale=0.4,keepaspectratio]{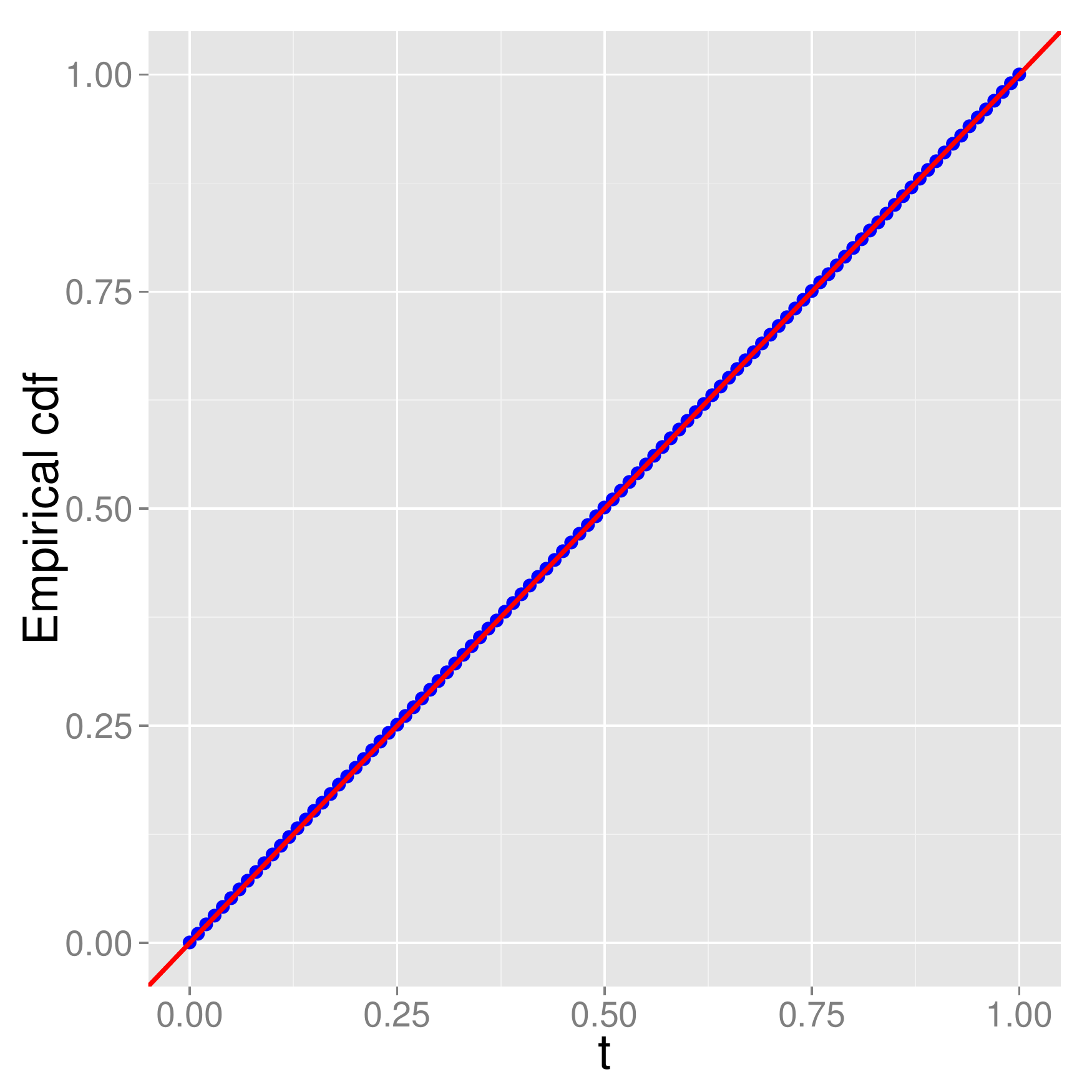}
	\caption{Empirical CDF of adjusted pvalues for logistic regression when $\bbeta=\bzero$, $n=200$, $p=60$. Here, the blue points represent the empirical CDF ($t$ vs.~the fraction of p-values below $t$), and the red line is the diagonal.  }
\label{fig: ecdf2}
\end{figure}

%% file: preliminaries.tex
\section{Preliminaries} \label{sec: prelims}

This section gathers a few preliminary results that will be useful throughout the paper. 
%
%
%
We start by collecting some facts regarding i.i.d.~Gaussian random matrices. 

\begin{lemma}
	\label{lemma: singval}
	Let $\bX = [\bX_1, \bX_2, \hdots \bX_n]^{\top}$ be an $n
        \times p$ matrix with i.i.d.~standard Gaussian entries. Then
\begin{equation}
	\label{eq: singval}
	\opP \left( \| \bX^{\tp} \bX \| \leq 9n \right) \geq 1-2\exp(-n/2); 
\end{equation}
%
\begin{align}
	\label{eq: xnorm}
	\opP\left(\sup\nolimits_{1\leq i\leq n} \|\bX_i \| \leq 2\sqrt{p} \right) & \geq 1- 2n \exp (- (\sqrt{p}-1)^2/2).
\end{align}
\end{lemma}

\begin{proof}
  This is a straighforward application of \cite[Corollary
  5.35]{vershynin2010introduction} and the union bound.
\end{proof}

\begin{lemma}
	\label{lem:eigen-min-S}
	Suppose $\bX$ is an $n \times p$ matrix with entries i.i.d $\dnorm(0,1)$, then there exists a constant $\epsilon_0$ such that whenever $0 \leq \epsilon \leq \epsilon_0$  and $0 \leq t \leq \sqrt{1-\epsilon} - \sqrt{p/n}$, 
\begin{equation}
\lambda_{\min}\left(\frac{1}{n}\sum_{i\in S}\bm{X}_{i}\bm{X}_{i}^{\top}\right)\geq\left(\sqrt{1-\epsilon}-\sqrt{\frac{p}{n}}-t\right)^{2},\quad\forall S\subseteq[n]\text{ with }|S|=(1-\epsilon)n
\end{equation}
with probability exceeding $1-2\exp\left(-\left(\frac{\left(1-\epsilon\right)t^{2}}{2} - H\left(\epsilon\right)\right)n\right)$. Here, 
	$H(\epsilon)=-\epsilon\log\epsilon -  (1-\epsilon)\log(1-\epsilon)$. 
\end{lemma}
\begin{proof}
	See Appendix \ref{sec:Proof-eigen-min-S}.
\end{proof}

The above facts are useful in establishing an eigenvalue lower bound on the Hessian of the log-likelihood function. Specifically, recall that
\begin{equation}
	\nabla^2 \ell(\bm{\beta}) = \sum\nolimits_{i=1}^{n}\rho''\left(\bm{X}_{i}^{\top}\bm{\beta}\right)\bm{X}_{i}\bm{X}_{i}^{\top},
\end{equation}
and the result is this:

\begin{lemma}[Likelihood Curvature Condition]
\label{lem:min-eigenvalue-bound}
Suppose that $p/n < 1$ 
and that $\rho''(\cdot) \geq 0$. Then there exists a constant $\epsilon_0$ such that whenever $0 \leq \epsilon \leq \epsilon_0$, with probability at least
$1-2\exp\left(-n{H}\left(\epsilon\right)\right) 
-2\exp\left(-{n}/{2}\right)$, 
the matrix inequality 
\begin{equation}
	\frac{1}{n} \nabla^2 \ell(\bm{\beta})
	~\succeq~  \left( \inf_{z:|z|\le\frac{3\|\bbeta \| }{\sqrt{\epsilon}}}  \rho''\left(z\right) \right)    
	\left(\sqrt{1-\epsilon}-\sqrt{\frac{p}{n}}-2\sqrt{\frac{H(\epsilon)}{1-\epsilon}}\right)^{2}
	\bm{I}
\end{equation}
holds simultaneously for all $\bm{\beta}\in\mathbb{R}^{p}$.
\end{lemma}
\begin{proof} See Appendix \ref{sec:proof-min-eigenvalue-bound}. 
\end{proof}

The message of Lemma \ref{lem:min-eigenvalue-bound} is this: take 
$\epsilon>0$ to be  a sufficiently small constant. Then  
  \[
    \frac{1}{n} \nabla^2 \ell(\bm{\beta}) ~\succeq~ \omega(\|\bm{\beta}\|) \hspace{0.2em} \bm{I}
  \]
  for some non-increasing and positive function $\omega(\cdot)$ independent of $n$. This is a generalization of the strong convexity condition. 


%% file: stat-dim.tex

\section{When is the MLE bounded?} \label{sec: MLEbounded}

\subsection{Phase transition\label{subsec:PT-norm-MLE}}

In Section \ref{sec:separation}, we argued that the MLE is at infinity
if we have less than two observations per dimension or $\kappa >
1/2$. In fact, a stronger version of the phase transition phenemonon
occurs in the sense that $$\|\hat{\bm{\beta}}\|=O(1)$$ as soon as
$\kappa<1/2$. This is formalized in the following theorem.

\begin{theorem}[Norm Bound Condition]\label{thm: normbound-MLE} 
  Fix any small constant $\epsilon>0$, and let $\hat{\bm{\beta}}$ be
  the MLE for a model with effective link satisfying the conditions
  from Section \ref{sec:rho}.

\begin{itemize}

\item[(i)]If $p/n\geq1/2+\epsilon$, then 
\[
\|\hat{\bm{\beta}}\|=\infty
\]
with probability exceeding $1-4\exp\left(-\epsilon^{2}n/8\right)$.

\item[(ii)] There exist universal constants $c_{1},c_{2},C_{2}>0$
	such that if $p/n<1/2-c_{1}\epsilon^{3/4}$, then\footnote{When $\bm{X}_i\sim \mathcal{N}(\bm{0},\bSigma)$ for a general $\bm{\Sigma}\succ \bm{0}$, one has $\|\bm{\Sigma}^{1/2}\hat{\bm{\beta}}\| \lesssim {1}/{\epsilon^{2}}$ with high probability.}
\[
\|\hat{\bm{\beta}}\|<\frac{4\log2}{\epsilon^{2}}
\]
with probability at least $1-C_{2}\exp({-c_{2}\epsilon^{2}n})$.
\end{itemize}
These conclusions clearly continue to hold if $\hat{\bm{\beta}}$ is
replaced by $\tilde{\bm{\beta}}$ (the MLE under the restricted model
obtained on dropping the first predictor).
\end{theorem}

The rest of this section is devoted to proving this theorem.  As we
will see later, the fact that $\|\hat{\bm{\beta}}\|=O(1)$ is crucial
for utilizing the AMP machinery in the absence of strong convexity.

\subsection{Proof of Theorem \ref{thm: normbound-MLE}}

As in Section \ref{sub:AMP-step}, we assume $\tilde{y}_i \equiv 1$  throughout this section, and hence the MLE reduces to
\begin{equation}
	\text{minimize}_{\bm{\beta}\in \mathbb{R}^p} \quad \ell_0\left(\bm{\beta}\right):=  \sum\nolimits_{i=1}^{n}  \rho( - \bm{X}_i^{\top} \bm{\beta}  ).
	 \label{eq:LL-alternative-normbound}
\end{equation}

\subsubsection{Proof of Part (i)}

Invoking \cite[Theorem I]{amelunxen2014living} yields that if 
\[
\delta\left(\left\{ \bm{X}\bm{\beta}\mid\bm{\beta}\in\mathbb{R}^{p}\right\} \right)+\delta\left(\mathbb{R}_{+}^{n}\right)\geq\left(1+\epsilon\right)n,
\]
or equivalently, if $p/n\geq1/2+\epsilon,$ then
\[
\mathbb{P}\left\{ \left\{ \bm{X}\bm{\beta}\mid\bm{\beta}\in\mathbb{R}^{p}\right\} \cap\mathbb{R}_{+}^{n}\neq\left\{ \bm{0}\right\} \right\} \geq1-4\exp\left(-\epsilon^{2}n/8\right).
\]
As is seen in Section \ref{sec:separation}, $\|\hat{\bm{\beta}}\| = \infty$
when $\left\{ \bm{X}\bm{\beta}\mid\bm{\beta}\in\mathbb{R}^{p}\right\} \cap\mathbb{R}_{+}^{n}\neq\left\{ \bm{0}\right\} $,
establishing Part (i) of Theorem \ref{thm: normbound-MLE}.

\subsubsection{Proof of Part (ii)\label{subsec:Proof-of-Part-II-PT}}

We now turn to the regime in which $p/n\leq1/2-O(\epsilon^{3/4})$,
where $0<\epsilon<1$ is any fixed constant. Begin by observing that 
the least singular value of $\bm{X}$ obeys
\begin{equation}
\sigma_{\min}\left(\bm{X}\right)\geq\sqrt{n}/4\label{eq:sigma-min-X}
\end{equation}
 with probability at least $1-2\exp\big(-\frac{1}{2}\big(\frac{3}{4}-\frac{1}{\sqrt{2}}\big)^{2}n\big)$
 (this follows from Lemma \ref{lem:eigen-min-S} using $\epsilon =0$). Then 
 for any $\bm{\beta}\in\mathbb{R}^{p}$
obeying
\begin{equation}
\ell_{0}(\bm{\beta})=\sum\nolimits _{j=1}^{n}\rho\left(-\bm{X}_{j}^{\top}\bm{\beta}\right)\leq n\log2=\ell_{0}(\bm{0}) \label{eq:L-beta-UB}
\end{equation}
\begin{equation}
\text{and}\qquad\|\bm{\beta}\|\geq\frac{4\log2}{\epsilon^{2}},\label{eq:betanormbound}
\end{equation}
we must have 
\[
\sum_{j=1}^{n}\max\left\{ -\bm{X}_{j}^{\top}\bm{\beta},\text{
  }0\right\}   =  \sum_{j:\text{
  }\bm{X}_{j}^{\top}\bm{\beta}<0}\left(-\bm{X}_{j}^{\top}\bm{\beta}\right)\overset{(\text{a})}{\leq}\sum_{j:\text{
  }\bm{X}_{j}^{\top}\bm{\beta}<0}\rho\left(-\bm{X}_{j}^{\top}\bm{\beta}\right)\overset{(\text{b})}{\leq}n\log2; 
\]
(a) follows since $t\leq\rho(t)$ and (b) is a
consequence of (\ref{eq:L-beta-UB}). Continuing,  
(\ref{eq:sigma-min-X}) and (\ref{eq:betanormbound}) give 
\[
n\log2 \le
4\sqrt{n}\frac{\|\bm{X}\bm{\beta}\|}{\|\bm{\beta}\|}\log2\leq\epsilon^{2}\sqrt{n}\|\bm{X}\bm{\beta}\|. 
\]
This implies the following proposition: if the solution
$\hat{\bm{\beta}}$---which necessarily satisfies
$\ell_{0}(\hat{\bm{\beta}})\leq\ell_{0}(\bm{0})$---has norm exceeding
$\|\hat{\bm{\beta}}\|\geq\frac{4\log2}{\epsilon^{2}}$, then
$\bm{X}\hat{\bm{\beta}}$ must fall within the cone
\begin{equation}
\mathcal{A}:=\left\{ \bm{u}\in\mathbb{R}^{n}\left|\text{ }\sum\nolimits _{j=1}^{n}\max\left\{ -u_{j},0\right\} \leq\epsilon^{2}\sqrt{n}\|\bm{u}\|\right.\right\} .\label{eq:defn-event-A}
\end{equation}
Therefore, if one wishes to rule out the possibility of having $\|\hat{\bm{\beta}}\|\geq\frac{4\log2}{\epsilon^{2}}$,
it suffices to show that with high probability, 
\begin{equation}
\left\{ \bm{X}\bm{\beta}\mid\bm{\beta}\in\mathbb{R}^{p}\right\} \cap\mathcal{A}=\left\{ \bm{0}\right\}.\label{eq:Xbeta-A-intersection}
\end{equation}
This is the content of the remaining proof. 

We would like to utilize tools from conic geometry
\cite{amelunxen2014living} to analyze the probability of the
event (\ref{eq:Xbeta-A-intersection}).  Note, however,
that $\mathcal{A}$ is not convex, while the theory developed in
\cite{amelunxen2014living} applies only to convex cones.  To bypass
the non-convexity issue, we proceed in the following three steps:
\begin{enumerate}
\item Generate a set of $N=\exp\left(2\epsilon^{2}p\right)$ closed \emph{convex}
cones $\left\{ \mathcal{B}_{i}\mid1\leq i\leq N\right\} $ such that
		it forms a cover of $\mathcal{A}$ with probability exceeding $1- \exp\left(-\Omega(\epsilon^{2}p)\right)$. 
\item Show that if $p<\left(\frac{1}{2}-2\sqrt{2}\epsilon^{\frac{3}{4}}-2H(2\sqrt{\epsilon})\right)n$
and if $n$ is sufficiently large, then 
\[
\mathbb{P}\left\{ \left\{ \bm{X}\bm{\beta}\mid\bm{\beta}\in\mathbb{R}^{p}\right\} \cap\mathcal{B}_{i}\neq\left\{ \bm{0}\right\} \right\} \leq4\exp\left\{ -\frac{1}{8}\left(\frac{1}{2}-2\sqrt{2}\epsilon^{\frac{3}{4}}-10H(2\sqrt{\epsilon})-\frac{p}{n}\right)^{2}n\right\} 
\]
for each $1\leq i\leq N$.
\item Invoke the union bound to reach 
\begin{eqnarray*}
\mathbb{P}\left\{ \left\{ \bm{X}\bm{\beta}\mid\bm{\beta}\in\mathbb{R}^{p}\right\} \cap\mathcal{A}\neq\left\{ \bm{0}\right\} \right\}  & \leq & \mathbb{P}\left\{ \left\{ \mathcal{B}_{i}\mid1\leq i\leq N\right\} \text{ does not form a cover of }\mathcal{A}\right\} \\
 &  & \quad+\sum_{i=1}^{N}\mathbb{P}\left\{ \left\{ \bm{X}\bm{\beta}\mid\bm{\beta}\in\mathbb{R}^{p}\right\} \cap\mathcal{B}_{i}\neq\left\{ \bm{0}\right\} \right\} \\
	& \leq & \exp\left(- \Omega(\epsilon^{2}p)\right),
\end{eqnarray*}
where we have used the fact that 
\begin{eqnarray*}
\sum_{i=1}^{N}\mathbb{P}\left\{ \left\{ \bm{X}\bm{\beta}\mid\bm{\beta}\in\mathbb{R}^{p}\right\} \cap\mathcal{B}_{i}\neq\left\{ \bm{0}\right\} \right\}  & \leq & 4N\exp\left\{ -\frac{1}{8}\left(\frac{1}{2}-2\sqrt{2}\epsilon^{\frac{3}{4}}-10H(2\sqrt{\epsilon})-\frac{p}{n}\right)^{2}n\right\} \\
 & < & 4\exp\left\{ -\left(\frac{1}{8}\left(\frac{1}{2}-2\sqrt{2}\epsilon^{\frac{3}{4}}-10H(2\sqrt{\epsilon})-\frac{p}{n}\right)^{2}-2\epsilon^{2}\right)n\right\} \\
 & < & 4\exp\left\{ -\epsilon^{2}n\right\} .
\end{eqnarray*}
Here, the last inequality holds if $\left(\frac{1}{2}-2\sqrt{2}\epsilon^{\frac{3}{4}}-10H(2\sqrt{\epsilon})-\frac{p}{n}\right)^{2}>24\epsilon^{2}$,
or equivalently, $\frac{p}{n}<\frac{1}{2}-2\sqrt{2}\epsilon^{\frac{3}{4}}-10H(2\sqrt{\epsilon})-\sqrt{24}\epsilon$. 
\end{enumerate}
Taken collectively, these steps establish the following claim: if
$\frac{p}{n}<\frac{1}{2}-2\sqrt{2}\epsilon^{\frac{3}{4}}-10H(2\sqrt{\epsilon})-\sqrt{24}\epsilon$,
then
\begin{eqnarray*}
	\mathbb{P}\left\{ \|\hbbeta\|>\frac{4\log2}{\epsilon^{2}}\right\}  & < & \exp\left\{ - \Omega(\epsilon^{2}n) \right\} ,
\end{eqnarray*}
thus establishing Part (ii) of Theorem \ref{thm: normbound-MLE}.
We defer the complete details of the preceding steps to Appendix \ref{sub:Proof-Theorem-normbound}.

%% file: AMP.tex
\section{Asymptotic $\ell_2$ error of the
MLE}\label{sec: AMP}

%

This section aims to establish Theorem \ref{thm: MLE-norm}, which characterizes precisely the asymptotic squared error  of the MLE $\hat{\bm{\beta}}$ under the global null $\bm{\beta}=\bm{0}$. As described in Section \ref{sub:AMP-step}, it suffices to assume that $\hat{\bm{\beta}}$ is the solution to the following problem
\begin{equation}
  \label{eq:MLE-simpler}
  \text{minimize}_{\bm{\beta}\in \mathbb{R}^p} \quad \sum\nolimits_{i=1}^n \rho(-\bm{X}_i^{\top}\bm{\beta}).
\end{equation}
In what follows, we derive the asymptotic convergence of
$\|\hat{\bm{\beta}}\|$ under the assumptions from our main theorem. 

\begin{theorem}
\label{theorem:AMP}
Under the assumptions of Theorem \ref{thm: thm1}, 
the solution $\hat{\bm{\beta}}$ to (\ref{eq:MLE-simpler}) obeys
\begin{equation}
  \lim_{n\rightarrow\infty} \|\hat{\bm{\beta}}\|^{2}=_{\mathrm{a.s.}} {\tau_{*}^{2}}.
\end{equation}
 \end{theorem}


 Theorem \ref{theorem:AMP} is derived by invoking the AMP machinery
 \cite{bayati2011dynamics,bayati2012lasso,javanmard2013state}.  The
 high-level idea is the following: in order to study $\hbbeta$, one
 introduces an iterative algorithm (called AMP) where a sequence of
 iterates $\hbbeta^t$ is formed at each time $t$. The algorithm is
 constructed so that the iterates asymptotically converge to the MLE
 in the sense that
\begin{equation}\label{eq:convres}
  \lim_{t \rightarrow \infty}\lim_{n\rightarrow \infty} 
  \|\hbbeta^t -\hbbeta \|^2=_{\text{a.s.}}0. 
  \end{equation}
  On the other hand, the asymptotic behavior (asymptotic in $n$) of
  $\hbbeta^t$ for each $t$ can be described accurately by a scalar
  sequence $\{\tau_t\}$---called {\em state evolution}
  (SE)---following certain update equations
  \cite{bayati2011dynamics}. This, in turn, provides a
  characterization of the $\ell_2$ loss of $\hat{\bm{\beta}}$.
%
%

Further, in order to prove Theorem \ref{thm: MLE-norm}, one still needs to justify 
\begin{itemize} 
  \item[(a)] the existence of a solution to the system of equations  \eqref{eq: sysofeq-tau} and \eqref{eq: sysofeq-b}, 
  \item[(b)] and the existence of a fixed point for the iterative map governing the SE sequence updates.
\end{itemize}
We will elaborate on these steps in the rest of this section.

%



\subsection{State evolution}\label{subsec: SE}

We begin with the SE sequence $\{\tau_t\}$ introduced in \cite{donoho2013high}. Starting from some initial point $\tau_0$, we produce two sequences $\{b_t\}$ and $\{\tau_{t}\}$ following a two-step procedure. 
\begin{itemize}
	\item For  $t=0,1, \ldots$: 
	\begin{itemize}
		\item Set $b_t$ to be the solution in $b$ to 
		\begin{align}
			\kappa & = \E\big[\Psi'(\tau_tZ; b)\big]; \label{eq: SE2b}
		\end{align}
		\item Set $\tau_{t+1}$ to be
		\begin{align}
			\tau_{t+1}^2 & = \frac{1}{\kappa} \E\big[(\Psi^2(\tau_tZ; b_t))\big]. \label{eq: SE2a}
		\end{align}
	\end{itemize}
\end{itemize}
Suppose that for given any $\tau>0$, the solution in $b$ to (\ref{eq:
  SE2b}) with $\tau_t=\tau$ exists and is unique, then one can denote
the solution as $b(\tau)$, which in turn allows one to write the
sequence $\{\tau_t\}$ as
\[
  \tau_{t+1}^2 = \Nu(\tau_t^2)
\]
with the variance map
\begin{equation}
	\label{eq: varmapone}
	\Nu(\tau^2) = \frac{1}{\kappa} \E \left[\Psi^2(\tau Z ;b(\tau))\right] .
\end{equation}
As a result, if there exists a fixed point $\taus$ obeying $\Nu(\taus^2)=\taus^2$ and if we start with $\tau_0 = \taus$, then by induction,
\[
\tau_t \equiv \taus \quad \text{and} \quad b_t \equiv \bs:=b(\taus), \qquad t = 0,1,\ldots
\]
Notably, $(\taus,\bs)$ solves the system of equations \eqref{eq: sysofeq-tau} and \eqref{eq: sysofeq-b}. We shall work with this choice of initial condition throughout our proof. 

The preceding arguments hold under two conditions: (i) the solution to
\eqref{eq: sysofeq-b} exists and is unique for any $\tau_t>0$; (ii)
the variance map (\ref{eq: varmapone}) admits a fixed point. To verify
these two conditions, we make two observations.
\begin{itemize}
	\item Condition (i) holds if one can show that the function
\begin{equation}
G(b):=\mathbb{E}\left[\Psi'(\tau Z;b)\right], \qquad b>0 
	\label{eq:defn-Gb}
\end{equation}
is strictly monotone for any given $\tau>0$, and that $\lim_{b \rightarrow 0}G(b) < \kappa < \lim_{b \rightarrow \infty}G(b)$. 

%
%
\item Since $\Nu(\cdot)$ is a continuous function, Condition (ii)
  becomes self-evident once we show that $\Nu(0)>0$ and that there
  exists $\tau>0$ obeying $\Nu(\tau^2) < \tau^2$. The behavior of the
  variance map is illustrated in Figure \ref{fig: varmap} for the logistic
  and probit regression when $\kappa =0.3$. One can in fact observe
  that the fixed point is unique. For other values of $\kappa$, the
  variance map shows the same behavior.
\end{itemize}

\begin{figure}
\centering
\begin{subfigure}{.5\textwidth}
  \centering
  \includegraphics[scale=0.4,keepaspectratio]{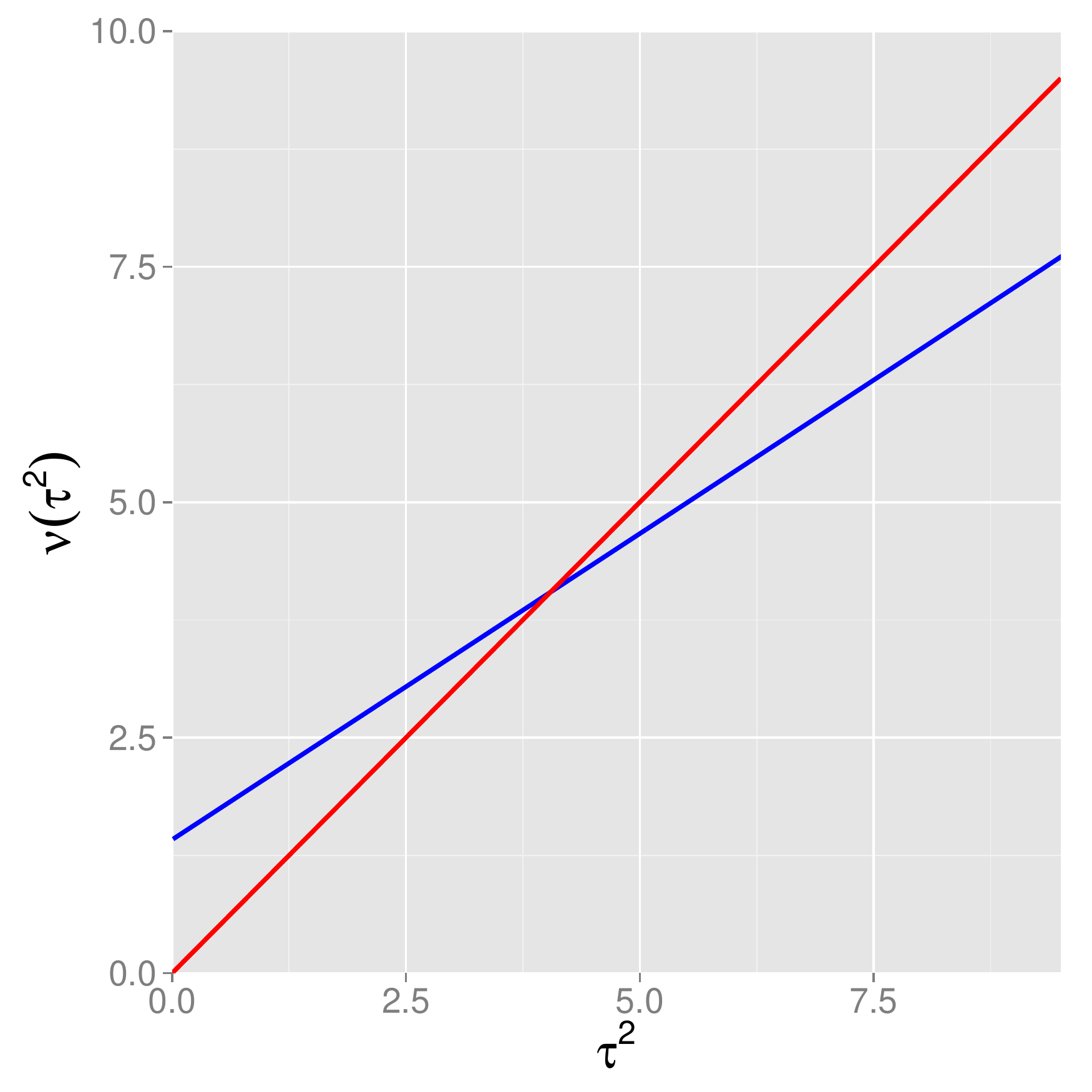}
  \caption{logistic regression}
\end{subfigure}%
\begin{subfigure}{.5\textwidth}
  \centering
  \includegraphics[scale=0.4,keepaspectratio]{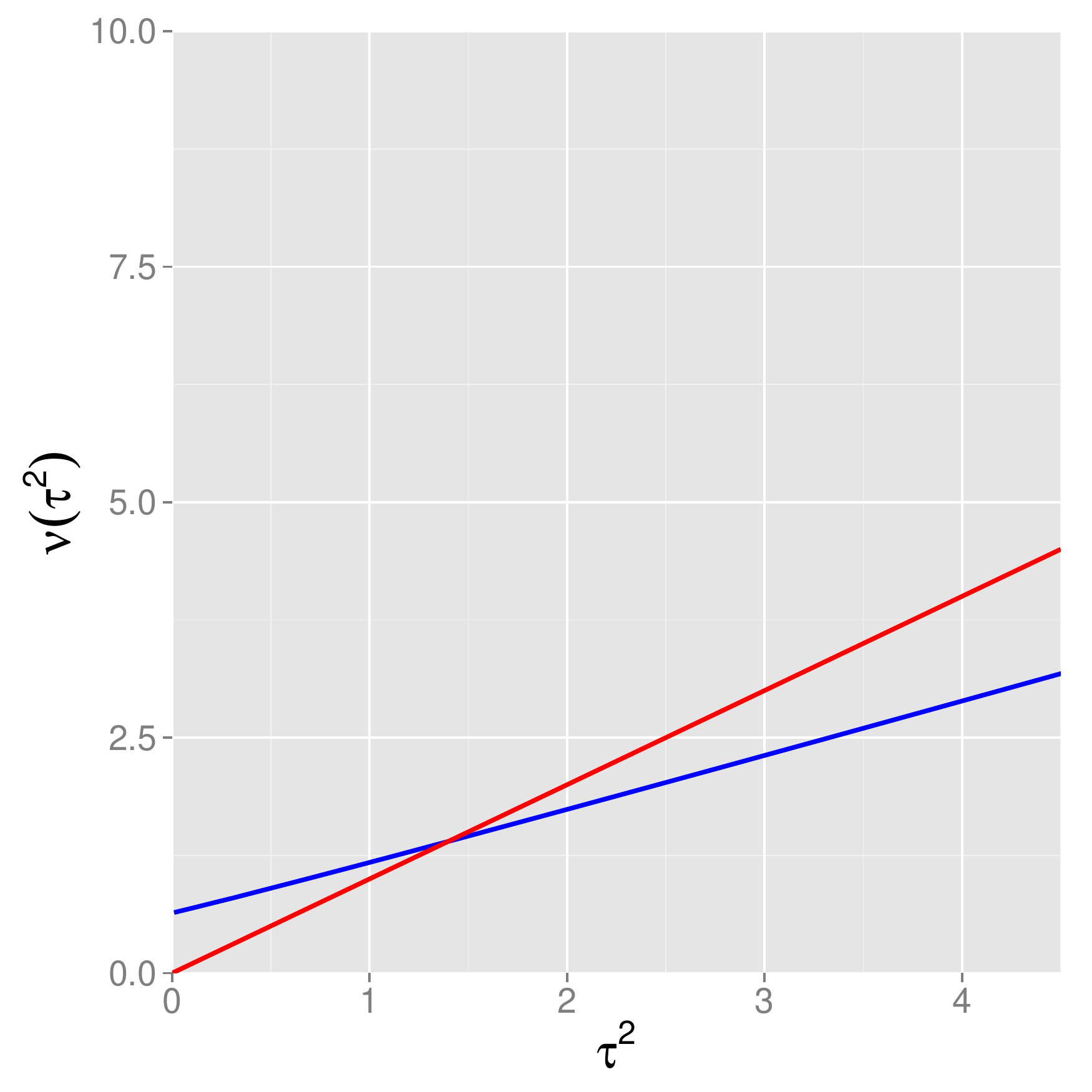}
  \caption{probit regression}
\end{subfigure}
	\caption{The variance map for both the logistic and the probit models when $\kappa = 0.3$:  (blue line) variance map $\Nu(\tau^2)$ as a function of $\tau^2$; (red line) diagonal.  }
\label{fig: varmap}
\end{figure}

In fact, the aforementioned properties can be proved for a certain class of effective links, as summarized in the following lemmas. In particular, they can be shown for the logistic and the probit models.

\begin{lemma}
\label{lemma: monotonicity}
Suppose the effective link $\rho$ satisfies the following two properties: 
\begin{enumerate}[label=(\alph*)]
\item $\rho'$ is log-concave.
\item For any fixed $\tau >0$ and any fixed $z$, $b \rho''(\prox_{b \rho}(\tau z)) \rightarrow \infty$ when $b \rightarrow \infty$.
\end{enumerate}
Then for any $\tau>0$, the function $G(b)$ defined in (\ref{eq:defn-Gb})
is an increasing function in $b$ ($b>0$), and the equation $$G(b)=\kappa$$
has a unique positive solution. 
\end{lemma}

\begin{proof} See Appendix \ref{sub:Proof-Lemma-monotonicity}.
\end{proof}


\begin{lemma}
\label{lemma: varmap} 
Suppose that $0<\kappa<1/2$ and that $\rho=\log(1+e^{t})$ or $\rho=-\log \Phi(-t)$. Then 
\begin{itemize}
\item[(i)] $\mathcal{V}(0)>0$;
\item[(ii)]  $\mathcal{V}(\tau^{2})<\tau^{2}$ for some  sufficiently
large $\tau^2$.
\end{itemize}
\end{lemma} 
\begin{proof} See Appendix \ref{sub:Proof-Lemma-var-map} and the
  supplemental material \cite{LRTsupp2017}.  
\end{proof}
\begin{remark} A byproduct of the proof is that the following relations hold for any constant $0 < \kappa < 1/2$: 
\begin{itemize}
\item In the logistic case,
\[
\begin{cases}
\lim_{\tau\rightarrow\infty}\frac{\mathcal{V}\left(\tau^{2}\right)}{\tau^{2}} & =\left.\frac{x^{2}\mathbb{P}\left\{ Z>x\right\} +\mathbb{E}\left[Z^{2}\bm{1}_{\left\{ 0<Z<x\right\} }\right]}{\mathbb{P}\left\{ 0<Z<x\right\} }\right|_{x=\Phi^{-1}(\kappa+0.5)};\\
	\lim_{\tau\rightarrow\infty}\frac{b({\tau})}{\tau} & =\Phi^{-1}(\kappa+0.5).
\end{cases}
\]

\item In the probit case,
\begin{align}
	\lim_{\tau \rightarrow \infty} b(\tau) = \frac{2\kappa}{1-2 \kappa}   \qquad \text{and} \qquad \lim_{\tau \rightarrow \infty} \frac{\Nu(\tau^2)}{\tau^2}  = 2 \kappa. 
\end{align}

\end{itemize}
\end{remark}
\begin{remark}
	Lemma \ref{lemma: varmap} is proved for the two special effective link functions, the logistic and the probit cases. However, the proof sheds light on general conditions on the effective link that suffice for the lemma to hold. Such general sufficient conditions are also discussed in the supplemental material \cite{LRTsupp2017}.
\end{remark}

\subsection{AMP recursion}
In this section, we construct the AMP trajectory tracked by two sequences $\{\hbbeta^t(n)\in \R^p\}$ and $\{\etab^t (n) \in \R^n\}$  for $t\geq 0$. Going forward we suppress the dependence on $n$ to simplify presentation. Picking $\hbbeta^0$ such that
\[
	\lim_{n\rightarrow \infty} \| \hbbeta^0 \|^2 = \tau_0^2 = \taus^2 
\]
and taking $\etab^{-1} = \bzero$ and $b_{-1} = 0$, the AMP path is obtained via Algorithm \ref{alg:AMP}, which is adapted from the algorithm in \cite[Section 2.2]{donoho2013high}.
\begin{algorithm}[H]
\caption{Approximate message passing.\label{alg:AMP}}
\begin{tabular}{>{\raggedright}p{1\textwidth}}
\textbf{For $t=0,1,\cdots$:}
\begin{enumerate}
\item Set 
\begin{equation}
\bm{\eta}^{t}=\bm{X}\hat{\bm{\beta}}^{t}+\Psi\left(\bm{\eta}^{t-1};b_{t-1}\right);\label{eq:AMP-eta}
\end{equation}
\vspace{-2em}
\item Let $b_t$ be the solution to
\begin{equation}
\kappa = \E \left[\Psi'(\tau_t Z; b) \right],
\end{equation}
where $\tau_t$ is the SE sequence value at that time.
\item Set 
\begin{equation}
\bm{\hat{\beta}}^{t+1}=\hat{\bm{\beta}}^{t}-\frac{1}{p}\bm{X}^{\top} \Psi\left(\bm{\eta}^{t};b_{t}\right).\label{eq:AMP-beta}
\end{equation}
\end{enumerate}
Here, $\Psi(\cdot)$  is applied in an entrywise manner, and $\Psi'(.,.)$ denotes derivative w.r.t the first variable.
\tabularnewline
\end{tabular}
\end{algorithm}

As asserted by  \cite{donoho2013high}, the SE sequence $\{\tau_t\}$ introduced in  Section \ref{subsec: SE} proves useful as it 
 offers a formal procedure for predicting operating characteristics of the AMP iterates at any fixed iteration. In particular it assigns predictions to two types of observables: observables which are functions of the $\hbbeta^t$ sequence and those which are functions of $\bm{\eta}^t$. 
 Repeating identical argument as in \cite[Theorem 3.4]{donoho2013high}, we obtain
\begin{equation}
  \lim_{n\rightarrow \infty} \| \hbbeta^t \|^2 =_{\text{a.s.}} {\tau_t^2} \equiv {\taus^2}, \qquad t = 0,1,\ldots.
  \label{eq: AMSE}
\end{equation}

\subsection{AMP converges to the MLE}\label{subsec: convergence}
We are now in position to show that the AMP iterates $\{\hbbeta^t\}$
converge to the MLE in the large $n$ and $t$ limit. Before continuing,
we state below two properties that are satisfied under our assumptions. 
\begin{itemize}
  \item The MLE $\hat{\bm{\beta}}$ obeys
  \begin{equation} \label{eq:beta-norm-bound}
	  \lim_{n \rightarrow \infty}\| \hat{\bm{\beta}} \|< \infty 
  \end{equation}
  almost surely. 
\item And there exists some non-increasing continuous function
  $0<\omega\left(\cdot\right)<1$ independent of $n$ such that
\begin{equation} 
  \label{eq: L-min-evalue}
  \mathbb{P}\left\{ 
  \frac{1}{n} \nabla^{2}\Lcal\left(\bm{\beta}\right)\succeq\omega\left(\|\bm{\beta}\|\right)\cdot\bm{I},\text{ }\forall\bm{\beta}\right\} \geq1-c_1 e^{-c_2 n}.
\end{equation} 
\end{itemize}
In fact, the norm bound \eqref{eq:beta-norm-bound} follows from
Theorem \ref{thm: normbound-MLE} together with Borel-Cantelli, while
the likelihood curvature condition (\ref{eq: L-min-evalue}) is an
immediate consequence of Lemma \ref{lem:min-eigenvalue-bound}. With
this in place, we have:
\begin{theorem}\label{thm: conv}
	Suppose (\ref{eq:beta-norm-bound}) and (\ref{eq: L-min-evalue}) hold. Let $(\taus,\bs)$ be a solution to the system \eqref{eq: sysofeq-tau} and \eqref{eq: sysofeq-b}, and assume that $\lim_{n \rightarrow \infty} \| \hbbeta^0\|^2 =\taus^2$. Then the AMP trajectory as defined in Algorithm \ref{alg:AMP} obeys
\[
	 \lim_{t \rightarrow \infty} \lim_{n \rightarrow \infty} \| \hbbeta^t - \hbbeta\| =_{\rm{a.s.}} 0. 
\]
\end{theorem}

Taken collectively, Theorem \ref{thm: conv} and Eqn.~(\ref{eq: AMSE}) imply that 
\begin{equation}
	  \lim_{n \rightarrow \infty} \| \hbbeta\| =_{\rm{a.s.}}  \lim_{t \rightarrow \infty} \lim_{n \rightarrow \infty} \| \hbbeta^t\| =_{\rm{a.s.}} \taus, 
\end{equation}
thus establishing Theorem \ref{theorem:AMP}. In addition, an upshot of these theorems is a uniqueness result:  
\begin{corollary}
	The solution to the system of equations \eqref{eq: sysofeq-tau} and \eqref{eq: sysofeq-b} is unique.
\end{corollary}
\begin{proof}
	When the AMP trajectory $\hbbeta^t$ is started with the initial condition from Theorem \ref{thm: conv}, $\lim_{n\rightarrow \infty} \| \hbbeta \|^2 =_{\text{a.s.}}\taus^2$. This holds for any $\taus$ such that $(\taus,\bs)$ is a solution to \eqref{eq: sysofeq-tau} and \eqref{eq: sysofeq-b}. However, since the MLE problem is strongly convex and hence admits a unique solution 
	$\hbbeta$, this implies that $\taus$ must be unique, which together with the monotonicity of $G(\cdot)$ (cf.~(\ref{eq:defn-Gb})) implies that $\bs$ is unique as well. 
\end{proof}

\begin{proof}[Proof of Theorem \ref{thm: conv}]

To begin with, repeating the arguments in \cite[Lemma 6.9]{donoho2013high} we reach
\begin{align}
\lim_{t \rightarrow \infty} \lim_{n \rightarrow \infty}  \| \hbbeta^{t+1} - \hbbeta^t\|^2 & = _{\rm{a.s.}} 0; \label{eq: beta-eta-diff1} \\
\lim_{t \rightarrow \infty} \lim_{n \rightarrow \infty}  \frac{1}{n}\| \etab^{t+1} - \etab^t\|^2 & = _{\rm{a.s.}} 0.\label{eq: beta-eta-diff}
\end{align}

To show that the AMP iterates converge to the MLE, we shall analyze the log-likelihood function. Recall from Taylor's theorem that 
\[
 \Lcal (\hbbeta) 
  = \Lcal (\hbbeta^t)+\left\langle \nabla\Lcal (\hbbeta ^t),\hbbeta -\hbbeta^t \right\rangle +\frac{1}{2} \left(\hbbeta -\hbbeta^{t} \right)^{\tp} \nabla^{2} \Lcal \left(\hbbeta^{t}+\lambda(\hbbeta-\hbbeta^{t})\right)\left(\hbbeta-\hbbeta^{t}\right)
\]
holds for some $0<\lambda<1$. To deal with the quadratic term, we would like to control the Hessian of the likelihood at a point between $\hbbeta$ and $\hbbeta^t$. 
Invoking the likelihood curvature condition (\ref{eq: L-min-evalue}), one has
\begin{equation}
	\Lcal(\hbbeta^t) \geq \Lcal(\hbbeta) \geq \Lcal (\hbbeta^t)+\left\langle \nabla\Lcal (\hbbeta^t), \hbbeta- \hbbeta^t \right\rangle +\frac{1}{2}{n\omega\Big(\max\left\{ \|\hat{\bm{\beta}}\|,\|\hat{\bm{\beta}}^{t}\|\right\} \Big)}\Vert \hbbeta-\hbbeta^t \Vert ^{2}
\label{eq:strong-convexity-1}
\end{equation}
with high probability. Apply Cauchy-Schwarz to yield that with exponentially high probability,
\[
 \|\hbbeta-\hbbeta^{t}\|
 \leq\frac{2}{\omega\big(\max\left\{ \|\hat{\bm{\beta}}\|,\|\hat{\bm{\beta}}^{t}\|\right\} \big)}\Big\Vert \frac{1}{n} \nabla\Lcal (\hbbeta^{t})\Big\Vert 
 \leq\frac{2}{\omega\big( \|\hat{\bm{\beta}}\| \big)\omega\big( \|\hat{\bm{\beta}}^{t}\| \big)}\Big\Vert
 \frac{1}{n} \nabla\Lcal (\hbbeta^{t})\Big\Vert ,
\]
where the last inequality follows since $0<\omega(\cdot)<1$ and $\omega(\cdot)$ is non-decreasing. 

It remains to control $\Vert \nabla\Lcal (\hbbeta^{t}) \Vert $. The
identity $\Psi(z;b_{*})=z-\mathsf{prox}_{b_{*}\rho}(z)$ and
\eqref{eq:AMP-eta} give
\begin{equation}
  \prox_{\bs\rho}\left(\etab^{t-1}\right) = \bX \hbbeta^{t}+\etab^{t-1}-\etab^{t}.
\end{equation}
In addition, substituting $\Psi\left(z;b\right)=b\rho'(\prox_{\rho b}(z))$
into \eqref{eq:AMP-beta} yields
\begin{eqnarray*}
  \frac{p}{\bs }(\hbbeta^{t}-\hbbeta^{t-1}) 
    =  -\bX^{\tp}\rho'\left(\prox_{\bs \rho}(\etab^{t-1})\right) 
    =  -\bX^{\tp}\rho'\left(\bX\hbbeta^{t}+\etab^{t-1}-\etab^{t}\right).
\end{eqnarray*}
We are now ready to bound $\|\nabla\Lcal(\hat{\bm{\beta}}^{t})\|$.
Recalling that 
\[
	\nabla\Lcal (\hbbeta^{t})=\bX^{\tp}\rho'(\bX^{\tp}\hbbeta^{t}) = \bX^{\tp}\rho'\left(\bX \hbbeta^{t}+\etab^{t-1}-\etab^{t}\right) + \bX^{\tp} \left( \rho'(\bX^{\tp}\hbbeta^{t}) - \rho' \left(\bX \hbbeta^{t}+\etab^{t-1}-\etab^{t}\right) \right)
\]
and that $\sup_{z}\rho''(z)<\infty$, we have 
\begin{eqnarray*}
  \big\Vert \nabla\Lcal(\hbbeta^{t})\big\Vert  & \leq & \left\Vert -\bX^{\tp}\rho'\left(\bX \hbbeta^{t}+\etab^{t-1}-\etab^{t}\right)\right\Vert +\|\bX\|\left|\rho'\left(\bX\hbbeta^{t}+\etab^{t-1}-\etab^{t}\right)-\rho' (\bX\hbbeta^{t} )\right|\\
 & \leq & \frac{p}{\bs  }\|\hbbeta^{t}-\hbbeta^{t-1}\|+\|\bX\|\left(\sup_{z}\rho''(z)\right)\|\etab^{t-1}-\etab^{t}\|.
\end{eqnarray*}
%
This establishes that with probability at least $1-c_1e^{-c_2n}$,
\begin{equation}
  \|\hbbeta-\hbbeta^{t}\|\leq\frac{2}{\omega\left( \|\hat{\bm{\beta}}\| \right)\omega\left( \|\hat{\bm{\beta}}^{t}\| \right)}\left\{ \frac{p}{\bs n} \|\hbbeta^{t}-\hbbeta^{t-1}\|+ \frac{1}{n}\left(\sup_{z}\rho''(z)\right)\|\bX\|\|\etab^{t-1}-\etab^{t}\|\right\} .
  \label{eq:theta-error-UB-1}
\end{equation}
%

Using \eqref{eq: xnorm} together with Borel-Cantelli yields $\lim_{n\rightarrow \infty} \| \bm{X} \|/\sqrt{n}<\infty$ almost surely. Further,  it follows from (\ref{eq: AMSE}) that $\lim_{n \rightarrow \infty} \|\hbbeta^t \| $ is finite almost surely as $\taus < \infty.$ These taken together with  (\ref{eq:beta-norm-bound}), 
(\ref{eq: beta-eta-diff1}) and (\ref{eq: beta-eta-diff}) yield 
\begin{equation}
  \lim_{t\rightarrow \infty}\lim_{n \rightarrow \infty}\|\hbbeta-\hbbeta^{t}\| =_{\text{a.s.}} 0
\end{equation}
as claimed.  
\end{proof}

%% file: LRT.tex
\section{Likelihood ratio analysis}
\label{sec: LRT}

This section presents the analytical details for Section \ref{sub:step2-LRT}, which relates the log-likelihood ratio statistic $\Lambda_i$ with $\hat{\beta}_i$.  
Recall from \eqref{eq:LRT-simplified} that the LLR statistic for testing $\beta_1=0$ vs.~$\beta_1 \neq 0$ is given by 
\begin{align}
\label{eq:LLR-2-terms}
 \Lambda_1 =  \frac{1}{2} \left(\tbX\tbbeta - \bX \hbbeta\right)^  \tp \bD_{ \hbbeta} \left(\tbX \tbbeta - \bX \hbbeta\right) + \frac{1}{6} \sum_{i=1}^n \rho'''(\gamma_i) \left(\tbX_i ^ \tp \tbbeta - \bX_i ^ \tp \hbbeta \right)^3,
 \end{align}
 where 
 \begin{equation}
\label{defn:Dbeta}
\bD_{\hbbeta}
:=
\left[\begin{array}{ccc}
\rho''\left(\bX_{1}^{\top}{\hbbeta}\right)\\
 & \ddots\\
 &  & \rho''\left(\bX_{n}^{\top}{{\hbbeta}}\right)
\end{array}\right]
\end{equation}
and $\gamma_i$ lies between $\bX_i^{\tp}\hbbeta$ and $\tbX_i^{\tp}\tbbeta$. 
The asymptotic  distribution of $\Lambda_1$ claimed in Theorem \ref{thm: Lambda-marginal} immediately follows from the result below, whose proof is the subject of the rest of this section.
\begin{theorem}
\label{thm:mainthm}
Let $(\taus,\bs)$ be the unique solution to the system of equations \eqref{eq: sysofeq-tau} and \eqref{eq: sysofeq-b}, and define
\begin{equation}
	\label{eq:defn-talpha}
	\tbG = \frac{1}{n} \tbX^{\tp} \Dtbeta \tbX  \qquad \text{and} \qquad \talpha= \frac{1}{n}\tr(\tbG^{-1}) .
\end{equation}
Suppose $p/n \rightarrow \kappa \in (0,1/2)$ 
. Then
\begin{itemize}
\item[(a)] the log-likelihood ratio statistic obeys
\begin{equation}
\label{eq:lL-beta1}
2 \Lambda_1 - p \hat{\beta}_1^2/\talpha ~\convP~ 0;
\end{equation}
\item[(b)] and the scalar $\talpha$ converges,
\begin{equation}
\label{eq:alpha-convergence}
\talpha ~\convP~ \bs. 
\end{equation}
\end{itemize}
	
\end{theorem}


\subsection{More notations and preliminaries} 
 
 Before proceeding, we introduce some notations that will be used throughout.  
 For any matrix $\bX$, denote by $X_{ij}$ and $\bX_{\cdot j}$ its $(i,j)$-th entry and $j$th column, respectively. 
 We denote an analogue $\br = \{r_i\}_{1\leq i\leq n}$
 (resp.~$\brt = \{\tilde{r}_i\}_{1\leq i\leq n} $) of residuals in the
 full (resp.~reduced) model by 
\begin{align}
\label{defn:residual}
  r_i := - \rho'\big(\bX_i^{\tp}\hbbeta\big) 
  \qquad \text{and} \qquad \tilde{r}_i :=  - \rho'\big(\tbX_i^{\tp}\tbbeta\big). 
\end{align}
As in (\ref{defn:Dbeta}),  set
\begin{equation}
\label{defn:Dbeta-tilde}
\bD_{\tbbeta}
:=
\left[\begin{array}{ccc}
\rho''\big(\tbX_{1}^{\top}{\tbbeta}\big)\\
 & \ddots\\
 &  & \rho''\big(\tbX_{n}^{\top}{{\tbbeta}}\big)
\end{array}\right] 
\quad \text{and} \quad
\Dinterm:= \left[\begin{array}{ccc}
\rho''\left(\gamma_1^*\right)\\
 & \ddots\\
 &  & \rho''(\gamma_n^*),
\end{array}\right] ,
\end{equation}
where $\gamma_i^*$ is between $\bX_i^{\tp}\hbbeta$ and $\bX_i^{\tp} \tbb$, and $\tbb$ is to be defined later in Section \ref{sub:defn-tbb}. 
Further, as in \eqref{eq:defn-talpha}, introduce the Gram matrices 
\begin{align}
\bG := \frac{1}{n} \bX^{\tp}\Dhbeta \bX
\quad
\text{and} \quad
\Ginterm = \frac{1}{n} \bX^{\tp} \Dinterm \bX.
\label{eq: defGtilde}
\end{align}
Let $\tbGi$ denote the version of $\tbG$ without the term corresponding to the $i{\text{th}}$ observation, that is, 
\begin{equation}
	\label{eq: defGi}
	\tbGi= \frac{1}{n} \sum\nolimits_{j: j \neq i} \rho''(\tbX_j^{\tp} \tbbeta) \tbX_j \tbX_j^{\tp}. 
 \end{equation}
%
 Additionally, let $\hbbetaimin$ be the MLE when the $i{\text{th}}$ observation is dropped and let $\bGimin$ be the corresponding Gram matrix,  
 \begin{align}\label{eq: bGimin} 
 \bGimin & = \frac{1}{n} \sum\nolimits_{j: j \neq i} \rho''(\bX_j^{\tp}\hbbetaimin)\bX_j \bX_j^{\tp}.
 \end{align}
Further, let  $\tbbetaimin$ be the MLE when the first predictor and $i$th  observation  are  removed, i.e. 
\[
	\tbbetaimin := \arg \min_{\bm{\beta}\in \mathbb{R}^{p-1}}~ \sum\nolimits_{j:j\neq i} \rho(   \tilde{\bm{X}}_j^{\top} \bm{\beta} ).
\]
Below $\tbGimin$ is the corresponding version of $\tbG$, 
\begin{equation}
	\label{eq: tbGimin}
	\tbGimin = \frac{1}{n} \sum\nolimits_{j: j \neq i} \rho''(\tbX_j^{\tp}\tbbetaimin)\tbX_j \tbX_j^{\tp}.  
\end{equation}
For these different versions of $\bG$, their least eigenvalues are all bounded away from 0, as asserted by the following lemma. 
\begin{lemma}
\label{lemma: eigenvalue}
There exist some absolute constants $\lambda_{\mathrm{lb}}, C, c >0 $ such that
	\[ 
		\opP(\lambda_{\min} (\bG) > \lambda_{\mathrm{lb}}) \geq 1-Ce^{-cn}. 
	\]
Moreover, the same result holds for $\tbG$, $\Ginterm$, $\tbGi$, $\bGimin$ and $\tbGimin$ for all $i \in [n]$.  
\end{lemma}
\begin{proof}This result follows directly from Lemma \ref{lemma: singval}, Lemma \ref{lem:min-eigenvalue-bound}, and Theorem \ref{thm: normbound-MLE}. 
\end{proof}

Throughout the rest of this section, we restrict ourselves (for any
given $n$) to the following event: 
\begin{align}
	\A_n & := \{ \lambda _{\min} (\tbG) > \lambda_{\mathrm{lb}} \} 
   ~\cap~ \{ \lambda_{\min}(\bG)> \lambda_{\mathrm{lb}}\} 
   ~\cap~ \{ \lambda_{\min}(\Ginterm) > \lambda_{\mathrm{lb}} \} \nonumber \\
   & ~\cap~ \{\cap_{i=1}^n \lambda_{\min}(\tbGi) > \lambda_{\mathrm{lb}} \} ~\cap~ \{ \cap_{i=1}^n \lambda_{\min}(\tbGimin) > \lambda_{\mathrm{lb}} \} ~\cap~ \{\cap_{i=1}^n \lambda_{\min}(\bGimin) > \lambda_{\mathrm{lb}} \}.
	\label{eq:defnAn}
 \end{align}
By Lemma \ref{lemma: eigenvalue}, $\A_n $ arises with exponentially high probability, i.e. 
\begin{equation}
  \label{eq: Anprob}
  \opP(\A_n ) \geq 1- \exp(-\Omega(n)).
\end{equation}

\subsection{A surrogate for the MLE}
\label{sub:defn-tbb}
In view of (\ref{eq:LLR-2-terms}), the main step in controlling $\Lambda_1$ consists of  characterizing the differences $\bX \hbbeta - \tbX \tbbeta $ or $ \hbbeta - \left[\begin{array}{c}
0\\
\tbbeta\\
\end{array}\right] $. 
Since the definition of $\hbbeta$ is implicit and not amenable to
direct analysis, we approximate $\hbbeta$ by a more amenable surrogate
$\tbb$, an idea introduced in
\cite{el2013robust,el2015impact,karoui2013asymptotic}.  We collect
some properties of the surrogate which will prove valuable in the
subsequent analysis.
%
%

To begin with, our surrogate is 
\begin{equation}
\label{eq: defbtilde}
\tbb = 
\begin{bmatrix}
0 \\ \tbbeta \end{bmatrix} 
+ \tb_1 \begin{bmatrix} 1 \\ -\tbG^{-1}  \bw \end{bmatrix},
\end{equation}
where $\tbG$ is defined in (\ref{eq: defGtilde}),
\begin{align}  
  \bw & := \frac{1}{n} \sum\nolimits_{i=1}^n \rho''(\tbX_i^{\tp} \tbbeta)X_{i1} \tbX_i = \frac{1}{n} \tbX^{\tp}\Dtbeta \Xcone , \label{eq: defw}
 \end{align}
 and $\tilde{b}_1$ is a scalar to be specified later. This vector is constructed in the hope that
 \begin{equation}
	 \hbbeta  ~\approx~ \tbb,   \qquad \text{or equivalently,} \qquad \begin{cases} \hat{\beta}_1  ~\approx~ \tilde{b}_1 ,  \\  \hbbeta_{2:p} - \tbbeta ~\approx~ - \tilde{b}_1  \tbG^{-1}  \bw,  \end{cases} 
  \label{eq:approx-tb-hbbeta}
 \end{equation}
where $\hbbeta_{2:p}$ contains the $2{\text{nd}}$ through $p{\text{th}}$ components of $\hbbeta$. 

Before specifying $\tb_1$, 
we shall first shed some insights into the remaining terms in $\tbb$.   By definition,  
%
%
%
\[
\nabla^2 \Lcal \left( 
\begin{bmatrix}
	0 \\ 
	\tbbeta  
\end{bmatrix}
 \right)
 = \bX^{\tp}\bD_{\tbbeta} \bX = \begin{bmatrix}
 \Xcone ^{\tp} \Dtbeta \Xcone & \Xcone^{\tp} \Dtbeta \tbX \\ 
 \tbX^{\tp} \Dtbeta \Xcone & \tbX ^ {\tp} \Dtbeta \tbX 
 \end{bmatrix}
= \begin{bmatrix}
	\Xcone ^{\tp} \Dtbeta \Xcone & n\bm{w}^{\top} \\ 
	n\bm{w} & n\tilde{\bm{G}} 
\end{bmatrix}.
\]
Employing the first-order approximation of $\nabla \Lcal(\cdot)$ gives
\begin{equation}
   \nabla^2 \Lcal \left( 
\begin{bmatrix}
	0 \\ 
	\tbbeta  
\end{bmatrix}
 \right)  
 \left( \hat{\bm{\beta}}  - 
 \begin{bmatrix}
	0 \\ 
	\tbbeta  
\end{bmatrix}
  \right)
  \approx \nabla \Lcal(\hbbeta) - \nabla \Lcal\left(\begin{bmatrix}
	0 \\ 
	\tbbeta  
\end{bmatrix}\right). 
	\label{eq:approx-L2}
\end{equation}
Suppose $\hbbeta_{2:p}$ is  well approximated by $\tbbeta$. Then all but the first coordinates of  $\nabla \Lcal(\tbbeta)$ and $\nabla \Lcal\left(\begin{bmatrix}
	0 \\ 
	\hbbeta  
\end{bmatrix}\right)$ are also very close to each other. Therefore, taking the $2{\text{nd}}$ through $p{\text{th}}$ components of (\ref{eq:approx-L2}) and approximating them by zero give
\[
  \left[\bm{w}, \tilde{\bm{G}} \right] 
 \left( \hat{\bm{\beta}}  - 
 \begin{bmatrix}
	0 \\ 
	\tbbeta  
\end{bmatrix}
  \right)
  \approx \bm{0}. 
\]
This together with a little algebra yields
\[
	\hat{\bm{\beta}}_{2:p} - \tilde{\bm{\beta}} ~\approx~ - \hat{\beta}_1 \tilde{\bm{G}}^{-1} \bm{w} ~\approx~ - \tilde{b}_1 \tilde{\bm{G}}^{-1} \bm{w}, 
\]
which coincides with (\ref{eq:approx-tb-hbbeta}).  
In fact, for all but the $1{\text{st}}$ entries, $\tbb$ is constructed by moving $\tbbeta$ one-step in the direction   which takes it closest to $\hbbeta$.

%
 
Next, we come to discussing the scalar $\tb_1$. Introduce the projection
matrix
\begin{equation}
\label{eq: defH}
\bH := \Id - \frac{1}{n} \Dtbeta^{1/2} \tbX \tbG^{-1} \tbX^{\tp} \Dtbeta ^{1/2}  ,
\end{equation}
and define  $\tb_1$ as 
\begin{equation}
\label{eq: deftb1}
\tb_1 := \frac{\Xcone^{\tp} \brt}{\Xcone^{\tp} \Dtbeta^{1/2} \bH \Dtbeta^{1/2} \Xcone},
\end{equation}
where $\brt$ comes from (\ref{defn:residual}). In fact, the expression
$\tb_1$ is obtained through similar (but slightly more complicated)
first-order approximation as for $\tbb_{2:p}$, in order to make sure that 
$b_1 \approx \hat{\beta}_1$; see \cite[Pages
14560-14561]{el2013robust} for a detailed description.


We now formally justify  that the surrogate $\tbb$ and the MLE $\hbbeta$  are  close to each other. 
\begin{theorem}\label{thm: hbbetabdiff}
The MLE $\hbbeta$ and the surrote $\tbb$  \eqref{eq: defbtilde} obey  
\begin{equation} 
\label{eq: hbbetabdiff}
	\|\hbbeta - \tbb \| ~\lesssim~ {n^{-1+o(1)}}, 
\end{equation}
\begin{equation}\label{eq: tb1anal}
	|\tb_1| ~\lesssim~ n^{-1/2+o(1)} ,
\end{equation}
and
\begin{equation}
	\sup_{1\leq i\leq n}|\bX_i^{\tp}\tbb - \tbX_i^{\tp}\tbbeta | ~\lesssim~ n^{-1/2+o(1)} 
\end{equation}
with probability tending to one as $n\rightarrow \infty$.  
\end{theorem}
\begin{proof} See Section \ref{sec:Proof-outline-Thm-hbbetadiff}. \end{proof}

%
The global accuracy \eqref{eq: hbbetabdiff} immediately leads to a coordinate-wise approximation result between $\hat{\beta}_1$ and $\tb_1$.

\begin{corollary}
\label{cor: one-coord}
With probability tending to one as $n\rightarrow \infty$, 
\begin{equation}
\label{eq: one-coord}
	\sqrt{n}|\tb_1 - \hat{\beta}_1| ~\lesssim~  n^{-1/2+o(1)}  .
\end{equation}
%
\end{corollary}


Another consequence from Theorem \ref{thm: hbbetabdiff} is that the value $\bX_i^{\tp}\hbbeta$ in the full model and its counterpart $\tbX_i^{\tp}\tbbeta$ in the reduced model are uniformly close.

\begin{corollary}\label{cor: fitdiff}
The values $\bX_i^{\tp}\hbbeta$ and $\tbX_i^{\tp}\tbbeta$ are uniformly close in the sense that
\begin{equation}
	\sup_{1\leq i\leq n} \big| \bX_i^{\tp}\hbbeta - \tbX_i^{\tp}\tbbeta \big| ~\lesssim~ n^{- 1/2 +o(1)} 
\end{equation}
holds with probability approaching one as $n\rightarrow \infty$. 
\end{corollary}
\begin{proof}
Note that 
\begin{align*}
	\sup_{1\leq i\leq n} \big|\bX_i^{\tp}\hbbeta - \tbX_i ^{\tp} \tbbeta \big|  & ~\leq~ \sup_{1\leq i\leq n} \big| \bX_i^{\tp} (\hbbeta -\tbb) \big| + \sup_{1\leq i\leq n}  \big| \bX_i^{\tp} \tbb - \tbX_i^{\tp}\tbbeta \big|.
\end{align*}
The second term in the right-hans side is upper bounded by $n^{- 1/2 +o(1)}$ with probability $1-o(1)$ according to Theorem \ref{thm: hbbetabdiff}. 
	Invoking Lemma \ref{lemma: singval} and Theorem \ref{thm: hbbetabdiff} and applying Cauchy-Schwarz inequality yield that the first term is $O(n^{-1/2+o(1)})$ with probability $1-o(1)$. This  establishes the claim. 
\end{proof}

\subsection{Analysis of the likelihood-ratio statistic}
\label{sub:analysis-LRT}

We are now positioned to use our surrogate $\tbb$ to analyze the likelihood-ratio statistic. In this subsection we establish Theorem \ref{thm:mainthm}(a). The proof for Theorem \ref{thm:mainthm}(b) is deferred to Appendix \ref{sec: appendix8}. 

	Recall from (\ref{eq:LRT-simplified}) that 
\begin{align*}
	2\Lambda_1 =   (\tbX\tbbeta - \bX \hbbeta)^  \tp \bD_{ \hbbeta} (\tbX \tbbeta - \bX \hbbeta) + \underset{:=I_3}{\underbrace{\frac{1}{3} \sum_{i=1}^n \rho'''(\gamma_i) (\tbX_i ^ \tp \tbbeta - \bX_i ^ \tp \hbbeta )^3}}.
 \end{align*}
To begin with, Corollary \ref{cor: fitdiff} together with the assumption $\sup_z\rho'''(z)<\infty$ implies that 
\[
	I_3 ~ \lesssim ~ n^{-1/2+o(1)}
\]
 with probability $1-o(1)$. Hence, $I_3$ converges to zero in probability.  
	
Reorganize the quadratic term as follows:
\begin{align}
(\tbX \tbbeta - \bX \hbbeta)^{\tp} \Dhbeta (\tbX\tbbeta - \bX \hbbeta)  & = \sum_i \rho''(\bX_i^{\tp}\hbbeta) \left(\bX_i ^{\tp}\hbbeta - \tbX_i^{\tp}\tbbeta \right)^2 \nonumber \\
& =\sum_i\rho''(\bX_i^{\tp}\hbbeta) \left[\bX_i^{\tp} (\hbbeta -\tbb) + (\bX_i^{\tp}\tbb -\tbX_i^{\tp}\tbbeta) \right]^2 \nonumber\\
& =  \sum_i\rho''(\bX_i^{\tp}\hbbeta)(\bX_i^{\tp} (\hbbeta -\tbb))^2
+2\sum_i\rho''(\bX_i^{\tp}\hbbeta)\bX_i^{\tp} (\hbbeta -\tbb)(\bX_i^{\tp}\tbb -\tbX_i^{\tp}\tbbeta) \nonumber\\
   & \qquad\qquad + \sum_i\rho''(\bX_i^{\tp}\hbbeta) \big(\bX_i^{\tp}\tbb -\tbX_i^{\tp}\tbbeta \big)^2 .
   \label{eq: Qform}
\end{align}
We control each of the three terms in the right-hand side of (\ref{eq: Qform}). 
\begin{itemize}
  \item Since $\sup_{z}\rho''(z) < \infty$, the first term in (\ref{eq: Qform}) is bounded by 
\begin{align*}
\sum\nolimits_i\rho''(\bX_i^{\tp}\hbbeta)(\bX_i^{\tp} (\hbbeta -\tbb))^2
~\lesssim ~
	||\tbbeta -\tbb ||^2 \left\|\sum\nolimits_i \bX_i\bX_i^{\tp} \right\| ~\lesssim~  {n^{-1+o(1)}}
\end{align*}
with probability $1-o(1)$, by an application of Theorem \ref{thm: hbbetabdiff} and Lemma \ref{lemma: singval}. 
  \item From the definition of $\tbb$, the second term can be upper bounded by
\begin{align*}
 2 \sum_i\rho''(\bX_i^{\tp}\hbbeta) (\hbbeta -\tbb)^{\tp} \bX_i \bX_i^{\tp}\tb_1  \begin{bmatrix} 1\\  -\tbG^{-1} \bw \end{bmatrix}  
	 & \leq  |\tb_1| \cdot\|\hbbeta -\tbb \| \cdot \left\|\sum\nolimits_i \bX_i\bX_i^{\tp} \right\| \cdot  \sqrt{1+\bw^{\tp}\tbG^{-2}\bw }  \\
	 & \lesssim  n^{-\frac{1}{2}+o(1)}
\end{align*}
with probability $1-o(1)$, where the last line follows from a combination of 
Theorem \ref{thm: hbbetabdiff},  Lemma \ref{lemma: singval} and the following lemma. 
%
\begin{lemma}
\label{lemma: svd}
Let $\tbG$ and $\bw$ be as defined in \eqref{eq: defGtilde} and \eqref{eq: defw}, respectively. Then  
\begin{equation}\label{eq: svd}
	\opP\left(\bw^{\tp} \tbG^{-2} \bw \lesssim 1 \right) ~\geq~ 1- \exp(-\Omega(n)).
\end{equation}
\end{lemma}
	\begin{proof} See Appendix \ref{sub:Proof-of-Lemmas-svd}.  \end{proof}

  \item The third term in \eqref{eq: Qform} can be decomposed as
\begin{align}
&\sum_i\rho''(\bX_i^{\tp}\hbbeta)(\bX_i^{\tp}\tbb -\tbX_i^{\tp}\tbbeta))^2  \nonumber\\
&\quad = \sum_i \left(\rho''(\bX_i^{\tp}\hbbeta)-\rho''(\tbX_i^{\tp}\tbbeta) \right) (\bX_i^{\tp}\tbb -\tbX_i^{\tp}\tbbeta))^2 + \sum_i\rho''(\tbX_i^{\tp}\tbbeta)(\bX_i^{\tp}\tbb -\tbX_i^{\tp}\tbbeta)^2 \nonumber\\
	& \quad = \sum_i \rho'''(\tilde{\gamma}_i)(\bX_i^{\tp}\hbbeta - \tbX_i^{\tp}\tbbeta) \left(\bX_i^{\tp}\tbb -\tbX_i^{\tp}\tbbeta \right)^2 + \sum_i\rho''(\tbX_i^{\tp}\tbbeta) \left(\bX_i^{\tp}\tbb -\tbX_i^{\tp}\tbbeta \right)^2  \label{eq:3rd-term-2-parts}
\end{align}
for some $\tilde{\gamma}_i$ between $\bX_i^{\tp}\hat{\bm{\beta}}$ and $\tbX_{i}^{\tp} \tbbeta$. From Theorem \ref{thm: hbbetabdiff} and Corollary \ref{cor: fitdiff}, 
the first term in \eqref{eq:3rd-term-2-parts} is $O(n^{-1/2+o(1)})$
with probability $1-o(1)$.  Hence, the only remaining term is the
second.
\end{itemize}
In summary, we have 
\begin{align}
\label{eq: LRTappinterm}
	2\Lambda_1 - \underset{= \bv^{\tp}\bX^{\tp}\Dtbeta \bX \bv}{\underbrace{\sum_i\rho''(\tbX_i^{\tp}\tbbeta) \left(\bX_i^{\tp}\tbb -\tbX_i^{\tp}\tbbeta\right)^2}} ~~\convP~ 0,
\end{align}
where $\bv := \tb_1 \begin{bmatrix}1 \\ -\tbG^{-1}\bw \end{bmatrix}$ according to (\ref{eq: defbtilde}). On simplification, the quadratic form reduces to 
\begin{align*}
\bm{v}^{\top}\bm{X}^{\top}\bm{D}_{\tilde{\bm{\beta}}}\bm{X}\bm{v} & =\tilde{b}_{1}^{2}\left(\bm{X}_{\cdot1}-\tilde{\bm{X}}\tilde{\bm{G}}^{-1}\bm{w}\right)^{\top}\bm{D}_{\tilde{\bm{\beta}}}\left(\bm{X}_{\cdot1}-\tilde{\bm{X}}\tilde{\bm{G}}^{-1}\bm{w}\right)\\
 & =\tilde{b}_{1}^{2}\left(\bm{X}_{\cdot1}^{\top}\bm{D}_{\tilde{\bm{\beta}}}\bm{X}_{\cdot1}-2\bm{X}_{\cdot1}^{\top}\bm{D}_{\tilde{\bm{\beta}}}\tilde{\bm{X}}\tilde{\bm{G}}^{-1}\bm{w}+\bm{w}^{\top}\tilde{\bm{G}}^{-1}\tilde{\bm{X}}^{\top}\bm{D}_{\tilde{\bm{\beta}}}\tilde{\bm{X}}\tilde{\bm{G}}^{-1}\bm{w}\right)\\
 & =\tilde{b}_{1}^{2}\left(\bm{X}_{\cdot1}^{\top}\bm{D}_{\tilde{\bm{\beta}}}\bm{X}_{\cdot1}-n\bm{w}^{\top}\tilde{\bm{G}}^{-1}\bm{w}\right)\\
	& =n\tilde{b}_{1}^{2}\Bigg( \underset{:=\xi}{\underbrace{ \frac{1}{n}\bm{X}_{\cdot1}^{\top}\bm{D}_{\tilde{\bm{\beta}}}^{1/2}\bm{H}\bm{D}_{\tilde{\bm{\beta}}}^{1/2}\bm{X}_{\cdot1} }} \Bigg),
\end{align*}
recalling the definitions \eqref{eq: defGtilde}, \eqref{eq: defw}, and \eqref{eq: defH}. 
Hence, the log-likelihood ratio $2\Lambda_1$ simplifies to $n \tb_1^2 \xi + o_P(1)$ on $\A_n$. 

Finally,  rewrite $\bm{v}^{\top}\bm{X}^{\top}\bm{D}_{\tilde{\bm{\beta}}}\bm{X}\bm{v}$ as $n(\tb_1^2 -\hat{\beta}_1^2)\xi+n\hat{\beta}_1^2 \xi$. To analyze the first term, note that 
 \begin{align}\label{eq: oneapprox}
 n|\tb_1^2-\hat{\beta}_1^2|& =n|\tb_1-\hat{\beta}_1| \cdot |\tb_1+\hat{\beta}_1| 
	 \leq  n|\tb_1-\hat{\beta_1}|^2 +2n|\tb_1| \cdot |\tb_1-\hat{\beta}_1| \lesssim n^{-\frac{1}{2}+o(1)}
 \end{align}
with probability $1-o(1)$ in view of Theorem \ref{thm: hbbetabdiff} and Corollary \ref{cor: one-coord}. It remains to analyze $\xi$. Recognize that   $\Xcone$ is independent of $\Dtbeta^{1/2}\bH \Dtbeta^{1/2}$. 
Applying the Hanson-Wright inequality \cite{hanson1971bound,rudelson2013hanson} and the Sherman-Morrison-Woodbury formula (e.g.~\cite{hager1989updating}) leads to the following lemma:
\begin{lemma}\label{lemma: xinapprox}
Let $\talpha = \frac{1}{n} \tr (\tbG^{-1})$, where $\tbG=\frac{1}{n}\tbX^{\tp} \Dtbeta \tbX$. Then one has 
\begin{equation}
  \left|\frac{p-1}{n} -  \talpha \frac{1}{n}  \Xcone^{\tp} \Dtbeta ^{1/2} \bH \Dtbeta^{1/2} \Xcone \right| \lesssim n^{-1/2+o(1)}
\end{equation}
with probability approaching one as $n\rightarrow \infty$. 
\end{lemma}
\begin{proof} See Appendix \ref{sec:Proof-ofLemma-xinapprox}. \end{proof}
%
 
\noindent In addition, if one can show that $\talpha$ is bounded away from zero with probability $1-o(1)$, then it is seen from Lemma \ref{lemma: xinapprox}  that
\begin{equation}\label{eq: xieq}
\xi - \frac{p}{n\tilde{\alpha}} ~\convP~ 0. 
\end{equation}
To justify the above claim, we observe that since $\rho''$ is bounded,  
$\lambda_{\max}(\tbG) \lesssim \lambda_{\max}(\tbX^{\tp}\tbX)/n \lesssim 1$
with exponentially high probability (Lemma \ref{lemma: singval}). This yields 
\[
	\talpha ~=~  \tr (\tbG^{-1}) / n ~\gtrsim~ p/n 
\]
with probability $1-o(1)$. On the other hand, on $\A_n$ one has
\[
	\talpha ~\leq~ p/(n \lambda_{\min}(\tbG)) ~\lesssim~ p/n.
\] 
Hence, it follows that $\xi = \Omega(1)$ with probability $1-o(1)$. Putting this together with \eqref{eq: oneapprox} gives the approximation 
\begin{equation}\label{eq: smallquadform}
\bv^{\tp} \bX^{\tp} \Dtbeta \bX \bv = n\hat{\beta}_1^2 \xi + o(1). 
\end{equation}
Taken collectively \eqref{eq: LRTappinterm}, \eqref{eq: xieq} and \eqref{eq: smallquadform} yields the desired result
 \[
   2\Lambda_1  - p\hat{\beta}_1^2/\talpha ~\convP~ 0.
 \]
 





\subsection{Proof of Theorem \ref{thm: hbbetabdiff}} 
\label{sec:Proof-outline-Thm-hbbetadiff}
This subsection  outlines the main steps for the proof of Theorem \ref{thm: hbbetabdiff}. 
To begin with, we shall express the difference $\hbbeta - \tbb$ in terms of the gradient of the negative log-likelihood function. Note that $\nabla \Lcal (\hbbeta) = \bm{0}$, and hence
\begin{align*}
\nabla \Lcal (\tbb) =  \nabla \Lcal (\tbb) - \nabla \Lcal (\hbbeta) & = \sum\nolimits_{i=1}^n \bX_i [\rho'(\bX_i ^{\tp} \tbb) - \rho'(\bX_i ' \hbbeta)] \\
& =  \sum\nolimits_{i=1}^n  \rho''(\gamma_i^{*}) \bX_i \bX_i^{\tp} (\tbb - \hbbeta), 
\end{align*}
where $\gamma_i^{*}$ is  between $\bX_i ^{\tp} \hbbeta$ and $\bX_i^{\tp} \tbb $. Recalling the notation introduced in \eqref{eq: defGtilde}, this can be rearranged as 
\begin{equation*}
\tbb - \hbbeta  =\frac{1}{n} \Ginterm^{-1} \nabla \Lcal(\tbb). 
\end{equation*}
Hence, on $\A_n$, this yields 
\begin{equation}
\label{eq: diffnabla}
\|\hbbeta-\tbb \| \leq  \frac{\|\nabla\Lcal(\tbb)\|}{\lambda_{\mathrm{lb}} n} . 
\end{equation}

The next step involves expressing $\nabla \Lcal (\tbb)$ in terms of the difference $\tbb - \begin{bmatrix}0 \\ \tbbeta \end{bmatrix}$.
%
\begin{lemma}
\label{lemma: nablal}
On the event $\A_n$ \eqref{eq:defnAn}, the negative log-likelihood evaluated at the surrogate $\tbb$  obeys 
\[\nabla \Lcal(\tbb) =  \sum_{i=1}^n \big[\rho''(\gamma_i^{*}) - \rho''(\tbX_i^{\tp} \tbbeta)\big] \bX_i \bX_i^{\tp} \left(\tbb - \left[\begin{array}{c}
0\\
\tilde{\bm{\beta}}
\end{array}\right]\right), \]
where $\gamma_i^{*}$ is some quantity between $\bX_i^{\tp} \tbb$ and $\tbX_i^{\tp} \tbbeta$.
\end{lemma}
\begin{proof} 
	The proof follows exactly the same argument as in the proof of \cite[Proposition 3.11]{el2015impact}, and is thus omitted.
\end{proof}
%
The point of expressing $\nabla \Lcal (\tbb)$ in this way is that the difference $\tbb - \left[\begin{array}{c}
0\\
\tilde{\bm{\beta}}
\end{array}\right]$ is known explicitly from the definition of $\tbb$. 
Invoking Lemma \ref{lemma: nablal} and the definition \eqref{eq: defbtilde}  allows one to further upper bound (\ref{eq: diffnabla}) as
\begin{align}
\|\hbbeta - \tbb \| 
	~\lesssim~ 
	\frac{1}{n}\left\Vert \nabla\ell(\tilde{\bm{b}})\right\Vert ~&\lesssim~ 
	\sup_{i}\left|\rho''(\gamma_{i}^{*})-\rho''(\tilde{\bm{X}}_{i}^{\top}\tilde{\bm{\beta}})\right|\left\Vert \frac{1}{n}\sum_{i=1}^{n}\bm{X}_{i}\bm{X}_{i}^{\top}\right\Vert \left\Vert \tilde{\bm{b}}-\left[\begin{array}{c}
0\\
\tilde{\bm{\beta}}
\end{array}\right]\right\Vert  \nonumber \\
~&\lesssim~  \sup_i \left|\bX_i ^{\tp} \tbb - \tbX_i^{\tp} \tbbeta \right|  | \rho'''|_{\infty} \left \|\frac{1}{n} \sum\nolimits_{i=1}^{n} \bX_i \bX_i ^{\tp} \right \| \cdot |\tb_1 |   \sqrt{1+ \bw ^{\tp} \tbG^{-2} \bw} \nonumber \\
	& \lesssim~ |\tb_1 |  \sup_i \left|\bX_i ^{\tp} \tbb - \tbX_i^{\tp} \tbbeta \right|  \label{eq: u.b.1}
\end{align}
with probability at least $1-\exp(-\Omega(n))$. The 
last inequality here comes from 
 our assumption that $\sup_z|\rho'''(z) | < \infty$ together with Lemmas \ref{lemma: singval} and \ref{lemma: svd}.

In order to bound \eqref{eq: u.b.1}, we first make use of the definition of $\tbb$ to reach 
\begin{equation}
  \label{eq: diff}
  \sup_i \left|\bX_i^{\tp} \tbb - \tbX_i^{\tp} \tbbeta \right| =| \tb_1| \sup_i|X_{i1} - \tbX_i^{\tp} \tbG^{-1} \bw|.
\end{equation}
The following lemma provides an upper bound on $\sup_i|X_{i1} - \tbX_i^{\tp} \tbG^{-1} \bw|$. 

\begin{lemma}
\label{lemma: approx}
With $\tbG$ and $\bw$ as defined in \eqref{eq: defGtilde} and \eqref{eq: defw}, 
\begin{equation}
	\label{eq: approx}
	\opP \left( \sup_{1\leq i\leq n} \big|X_{i1} - \tbX_i ^{\tp} \tbG^{-1} \bw \big| \leq n^{o(1)} \right) \geq 1- o(1). 
\end{equation}
\end{lemma}
\begin{proof} See Appendix \ref{sub:Proof-of-Lemmas-nablal-approx}.  
\end{proof}

In view of Lemma \ref{lemma: approx},  the second term in the right-hand side of (\ref{eq: diff}) is bounded above by $n^{o(1)}$ with high probability. Thus, in both the bounds \eqref{eq: u.b.1} and \eqref{eq: diff}, it only remains to analyze the term $\tb_1$. To this end, we control the numerator and the denominator of $\tb_1$ separately. 


\begin{itemize}
	\item
Recall from the definition (\ref{eq: deftb1}) that the numerator of $\tb_1$ is given by $\Xcone^{\tp} \brt$ and that $\brt$ is independent of $\Xcone$. Thus, conditional on $\tbX$, the quantity $\Xcone^{\tp}\brt$ is distributed as a Gaussian with mean zero and variance 
\[ 
	\sigma^2 = \sum\nolimits_{i=1}^n \big(\rho'(\tbX_i^{\tp}\tbbeta) \big)^2 .  
\]
Since $|\rho'(x)| = O(|x|)$, the variance is bounded by
\begin{equation}
	\sigma^2 ~\lesssim \tbbeta^{\tp} \left( \sum\nolimits_{i=1}^n \tbX_i \tbX_i^{\tp} \right) \tbbeta ~\lesssim~  n \| \tbbeta \|^2 \lesssim n
\end{equation}
with probability at least $1-\exp(-\Omega(n)))$, a consequence from
Theorem \ref{thm: normbound-MLE} and Lemma \ref{lemma: singval}.
Therefore, with probability  $1- o(1)$, we have 
\begin{equation}
\label{eq: btilnum}
	\frac{1}{\sqrt{n}} \Xcone^{\tp} \brt ~\lesssim~ n^{o(1)}  . 
\end{equation}


\item 
	We now move on to the denominator of $\tb_1$ in (\ref{eq: deftb1}). In the discussion following Lemma \ref{lemma: xinapprox} we showed $\frac{1}{n}  \Xcone^{\tp} \Dtbeta ^{1/2} \bH \Dtbeta^{1/2} \Xcone = \Omega (1)$ with probability $1-o(1)$.  
\end{itemize}
Putting the above bounds together, we conclude  
\begin{equation}\label{eq: btilde1}
	\opP \left( |\tb_1| \lesssim n^{-\frac{1}{2}+o(1)} \right) = 1-o(1) .
\end{equation}
Substitution into \eqref{eq: u.b.1} and \eqref{eq: diff} yields
\[
	\| \hat{\bm{\beta}} - \tilde{\bm{b}} \| ~\lesssim~  {n^{-1+o(1)}} \qquad
	\text{and} \qquad \sup_i \left|\bX_i^{\tp} \tbb - \tbX_i^{\tp} \tbbeta \right| ~\lesssim~ {n^{-1/2+o(1)} }
\]
%
with probability $1-o(1)$ as claimed.


%% file: discussion.tex
\section{Discussion}\label{sec: discussion}

In this paper, we derived the high-dimensional asymptotic distribution
of the LLR under our modelling assumptions. In particular, we showed
that the LLR is inflated vis a vis the classical Wilks' approximation
and that this inflation grows as the dimensionality $\kappa$
increases. This inflation is typical of high-dimensional problems, and
one immediate practical consequence is that it explains why
classically computed p-values are completely off since they tend to be
far too small under the null hypothesis. In contrast, we have shown in
our simulations that our new limiting distribution yields reasonably
accurate p-values in finite samples. Having said this, our work raises
a few important questions that we have not answered and we conclude
this paper with a couple of them.
\begin{itemize}
\item We expect that our results continue to hold when the covariates
  are not normally distributed, see Section \ref{sec: numerics} for
  some numerical evidence in this direction. To be more precise, we
  expect the same limiting distribution to hold when the variables are
  simply sub-Gaussian.  If this were true, then this would imply that our
  rescaled chi-square has a form of universal validity.

\item The major limitation of our work is arguably the fact that our
  limiting distribution holds under the global null; that is, under
  the assumption that all the regression coefficients vanish. It is
  unclear to us how the distribution would change in the case where
  the coefficients are not all zero. In particular, would the limiting
  distribution depend upon the unknown values of these coefficients?
  Are there assumptions under which it would not? Suppose for instance
  that we model the regression coefficients as i.i.d.~samples from the
  mixture model
  \[
(1-\epsilon) \delta_0 + \epsilon \Pi^\star,
\]
where $0 < \epsilon < 1$ is a mixture parameter, $\delta_0$ is a point
mass at zero and $\Pi^\star$ is a distribution with vanishing mass at
zero. Then what would we need to know about $\epsilon$ and $\Pi^\star$
to compute the asymptotic distribution of the LLR under the null?
\end{itemize}

%% file: appendix.tex
\appendix

\section{Proofs for Eigenvalue Bounds}\label{sec: appendix1}

\subsection{Proof of Lemma \ref{lem:eigen-min-S}}
\label{sec:Proof-eigen-min-S}


Fix $\epsilon \geq 0$ sufficiently small. For 
any given $S\subseteq[n]$ obeying $|S|=(1-\epsilon)n$ and $0 \leq  t \leq \sqrt{1-\epsilon} - \sqrt{p/n}$ it follows
from \cite[Corollary 5.35]{vershynin2010introduction} that
\[
\lambda_{\min}\left(\frac{1}{n}\sum_{i\in S}\bm{X}_{i}\bm{X}_{i}^{\top}\right)
<
\frac{1}{n}\left(\sqrt{|S|}-\sqrt{p}-t\sqrt{n}\right)^{2}=\left(\sqrt{1-\epsilon}-\sqrt{\frac{p}{n}}-t\right)^{2}
\]
holds with probability at most $2\exp\left(-\frac{t^{2}|S|}{2}\right)=2\exp\left(-\frac{\left(1-\epsilon\right)t^{2}n}{2}\right)$.
Taking the union bound over all possible subsets $S$ of size
$(1-\epsilon)n$ gives
\begin{align*}
\  & \mathbb{P}\left\{ \exists S\subseteq[n]\text{ with }|S|=(1-\epsilon)n\quad\text{s.t.}\quad \frac{1}{n}\lambda_{\min}\left(\sum_{i\in S}\bm{X}_{i}\bm{X}_{i}^{\top}\right)<\left(\sqrt{1-\epsilon}-\sqrt{\frac{p}{n}}-t\right)^{2}\right\} \\
 & \quad\leq~{n \choose (1-\epsilon)n}2\exp\left(-\frac{\left(1-\epsilon\right)t^{2}n}{2}\right)\\
 & \quad\leq~2\exp\left(n H\left(\epsilon\right) - \frac{\left(1-\epsilon\right)t^{2}}{2}n\right),
\end{align*}
where the last line is a consequence of the inequality ${n \choose (1-\epsilon)n}\leq e^{n H(\epsilon)}$ \cite[Example 11.1.3]{cover2012elements}.

\subsection{Proof of Lemma \ref{lem:min-eigenvalue-bound}}
\label{sec:proof-min-eigenvalue-bound}

Define
\[
S_{B}\left(\bm{\beta}\right):=\left\{ i:\text{ }|\bm{X}_{i}^{\top}\bm{\beta}|\leq B\|\bm{\beta}\|\right\} 
\]
for any $B>0$ and any $\bbeta$. Then 
\begin{eqnarray*}
\sum_{i=1}^{n}\rho''\left(\bm{X}_{i}^{\top}\bm{\beta}\right)\bm{X}_{i}\bm{X}_{i}^{\top} & \succeq & \sum_{i\in S_{B}\left(\bm{\beta}\right)}\rho''\left(\bm{X}_{i}^{\top}\bm{\beta}\right)\bm{X}_{i}\bm{X}_{i}^{\top}\succeq\inf_{z:|z|\le B\|\bbeta \|}\rho''\left(z\right)\sum_{i\in S_{B}\left(\bm{\beta}\right)}\bm{X}_{i}\bm{X}_{i}^{\top}.
\end{eqnarray*}
If one also has $|S_{B}\left(\bm{\beta}\right)|\geq(1-\epsilon)n$ (for $\epsilon \geq 0$ sufficiently small),
then this together with Lemma \ref{lem:eigen-min-S} implies that
\[
\frac{1}{n}\sum_{i=1}^{n}\rho''\left(\bm{X}_{i}^{\top}\bm{\beta}\right)\bm{X}_{i}\bm{X}_{i}^{\top}\succeq\inf_{z:|z|\le B\|\bbeta \|}\rho''\left(z\right)\left(\sqrt{1-\epsilon}-\sqrt{\frac{p}{n}}-t\right)^{2}\bm{I}
\]
with probability at least $1-2\exp\left(-\left(\frac{\left(1-\epsilon\right)t^{2}}{2} - H\left(\epsilon\right)\right)n\right)$. 

Thus if we can ensure that  with high probability, $|S_{B}\left(\bm{\beta}\right)|\geq(1-\epsilon)n$ holds simultaneously for all $\bbeta$, then we are done. From Lemma \ref{lemma: singval}
we see that
$
\frac{1}{n}\left\Vert \bm{X}^{\top}\bm{X}\right\Vert 
	\leq 9
$
with probability exceeding $1-2\exp\left(-n/2\right)$. On this event, 
\begin{equation}
	\left\Vert \bm{X}\bm{\beta}\right\Vert ^{2}\leq9n\|\bm{\beta}\|^{2}, \qquad \forall\bbeta. 
	\label{eq:XB-UB}
\end{equation}
On the other hand, the definition of $S_B(\bm{\beta})$ gives
\begin{equation}
\left\Vert \bm{X}\bm{\beta}\right\Vert ^{2}\geq\sum_{i\notin S_{B}(\bm{\beta})}\left|\bm{X}_{i}^{\top}\bm{\beta}\right|^{2}
\geq \big(n-\left|S_{B}(\bm{\beta})\right|\big)\left(B\|\bm{\beta}\|\right)^{2}=n\left(1-\frac{\left|S_{B}(\bm{\beta})\right|}{n}\right)B^{2}\|\bm{\beta}\|^2.
\label{eq:XB-LB}
\end{equation}
Taken together, (\ref{eq:XB-UB}) and (\ref{eq:XB-LB}) yield
\[
\left|S_{B}(\bm{\beta})\right|\geq\left(1-\frac{9}{B^{2}}\right)n, \qquad \forall\bbeta
\]
with probability at least $1-2\exp(-n/2)$. Therefore, with probability
$1-2\exp(-n/2)$,
$
\left|S_{3/\sqrt{\epsilon}}(\bm{\beta})\right|\geq\left(1-\epsilon\right)n
$
holds simultaneously for all $\bbeta$.  Putting the above results
together and setting $t = 2\sqrt{\frac{H(\epsilon)}{1-\epsilon}}$ give
\[
\sum_{i=1}^{n}\rho''\left(\bm{X}_{i}^{\top}\bm{\beta}\right)\bm{X}_{i}\bm{X}_{i}^{\top}\succeq\inf_{z:|z|\le\frac{3\|\bbeta \|}{\sqrt{\epsilon}}}\rho''\left(z\right)\left(\sqrt{1-\epsilon}-\sqrt{\frac{p}{n}}- 2\sqrt{\frac{H(\epsilon)}{1-\epsilon}}\right)^{2}\bm{I}
\]
simultaneously for all $\bbeta$ with probability at least $1-2\exp\left(-nH\left(\epsilon\right)\right)
-2\exp\left(-{n}/2\right)$.

\section{Proof of Lemma \ref{lemma: monotonicity}}
\label{sub:Proof-Lemma-monotonicity}

Applying an integration by parts leads to
\begin{eqnarray*}
\mathbb{E}\left[\Psi'(\tau Z;b)\right] & = & {\displaystyle \int}_{-\infty}^{\infty}\Psi'(\tau z;b)\phi(z)\mathrm{d}z=\frac{1}{\tau} \Psi(\tau z;b)\phi(z)\Big|_{-\infty}^{\infty}-\frac{1}{\tau}{\displaystyle \int}_{-\infty}^{\infty}\Psi(\tau z;b)\phi'(z)\mathrm{d}z\\
 & = & - \frac{1}{\tau} {\displaystyle \int}_{-\infty}^{\infty}\Psi(\tau z;b)\phi'(z)\mathrm{d}z
\end{eqnarray*}
with $\phi(z)=\frac{1}{\sqrt{2\pi}}\exp(-z^{2}/2)$. This reveals
that
\begin{eqnarray}
G'(b) & = & -\frac{1}{\tau}{\displaystyle \int}_{-\infty}^{\infty}\frac{\partial\Psi(\tau z;b)}{\partial b}\phi'(z)\mathrm{d}z=-\frac{1}{\tau}{\displaystyle \int}_{-\infty}^{\infty}\frac{\rho'\left(\mathsf{prox}_{b\rho}(\tau z)\right)}{1+b\rho''\left(\mathsf{prox}_{b\rho}(\tau z)\right)}\phi'(z)\mathrm{d}z\nonumber \\
 & = & \frac{1}{\tau}\int_{0}^{\infty}\left(\frac{\rho'\left(\mathrm{prox}_{b\rho}(-\tau z)\right)}{1+x\rho''\left(\mathrm{prox}_{b\rho}(-\tau z)\right)}-\frac{\rho'\left(\mathrm{prox}_{b\rho}(\tau z)\right)}{1+x\rho''\left(\mathrm{prox}_{b\rho}(\tau z)\right)}\right)\phi'(z)\mathrm{d}z,\label{eq:G-prime-integral}
\end{eqnarray}
where the second identity comes from \cite[Proposition
6.4]{donoho2013high}, and the last identity holds since
$\phi'(z)=-\phi'(-z)$.

Next, we claim that 
\begin{itemize}
\item[(a)] The function $h\left(z\right):=\frac{\rho'\left(z\right)}{1+b\rho''\left(z\right)}$
is increasing in $z$; 
\item[(b)] $\mathrm{prox}_{b\rho}(z)$ is increasing in $z$. 
\end{itemize}
These two claims imply that
\[
\frac{\rho'\left(\mathrm{prox}_{b\rho}(-\tau z)\right)}{1+b\rho''\left(\mathrm{prox}_{b\rho}(-\tau z)\right)}-\frac{\rho'\left(\mathrm{prox}_{b\rho}(\tau z)\right)}{1+b\rho''\left(\mathrm{prox}_{b\rho}(\tau z)\right)}<0,\quad\forall z>0,
\]
which combined with the fact $\phi'(z)<0$ for $z>0$ reveals 
\[
\mathrm{sign}\left(\left(\frac{\rho'\left(\mathrm{prox}_{b\rho}(-\tau z)\right)}{1+b\rho''\left(\mathrm{prox}_{b\rho}(-\tau z)\right)}-\frac{\rho'\left(\mathrm{prox}_{b\rho}(\tau z)\right)}{1+b\rho''\left(\mathrm{prox}_{b\rho}(\tau z)\right)}\right)\phi'(z)\right)=1,\quad\forall z>0.
\]
In other words, the integrand in (\ref{eq:G-prime-integral}) is positive,
which allows one to conclude that $G'(b)>0$. 

We then move on to justify (a) and (b). For the first, the derivative of
$h$ is given by
\[h'(z)= \frac{\rho''(z)+b(\rho''(z))^2 - b\rho'(z)\rho'''(z) }{\left(1+b\rho''(z)  \right)^2}. \]
Since $\rho'$ is log concave, this directly yields $(\rho'')^2 - \rho' \rho''' > 0$. As $\rho'' > 0$ and $b \geq 0$, the above implies $h'(z) > 0$ for all $z$. 
The second claim follows from $\frac{\partial\mathrm{prox}_{b\rho}(z)}{\partial z}\geq\frac{1}{1+b\|\rho''\|_{\infty}}>0$
(cf.~\cite[Equation (56)]{donoho2013high}). 

It remains to analyze the behavior of $G$ in the limits when $b \rightarrow 0$ and $b \rightarrow \infty$. From \cite[Proposition 6.4]{donoho2013high}, $G(b)$ can also be expressed as 
\[G(b) = 1 - \E\left[\frac{1}{1+b\rho''(\prox_{b \rho}(\tau Z))}  \right].\]
Since $\rho''$ is bounded and the integrand is at most $1$, the dominated convergence theorem gives  
\[\lim_{b \rightarrow 0}G(b) =0. \]
When $b \rightarrow \infty$, $b\rho''(\prox_{b \rho}(\tau z)) \rightarrow \infty$ for a fixed $z$. Again by applying the dominated convergence theorem, 
\[\lim_{b \rightarrow \infty}G(b) =1.\]
It follows that
$\lim_{b \rightarrow 0}G(b) < \kappa < \lim_{b \rightarrow
  \infty}G(b)$
and, therefore, $G(b) = \kappa$ has a unique positive solution.
\begin{remark}
Finally, we show that the logistic and the probit effective links obey the assumptions of Lemma \ref{lemma: monotonicity}. We work with a fixed $\tau >0$.
\begin{itemize}
\item A direct computation shows that $\rho'$ is log-concave for the
  logistic model. For the probit, it is well-known that the reciprocal
  of the hazard function (also known as Mills' ratio) is strictly
  log-convex \cite{baricz2008mills}.

\item To check the other condition, recall that the proximal 	mapping operator satisfies 
	\begin{equation}
		b\rho'(\prox_{b \rho}(\tau z))+\prox_{b \rho}(\tau z) = \tau z. \label{eq:b1}
	\end{equation}
		For a fixed $z $, we claim that if $b \rightarrow \infty$, $\prox_{b \rho}(\tau z) \rightarrow -\infty$. To prove this claim, we start by assuming that this is not true. Then either $\prox_{b \rho}(\tau z)$ is bounded or diverges to $\infty$. If it is bounded, the LHS above diverges to $\infty$ while the RHS is fixed, which is a contradiction. Similarly if $\prox_{b \rho}(\tau z)$ diverges to $\infty,$ the left-hand side of (\ref{eq:b1}) diverges to $\infty$ while the right-hand side is fixed, which cannot be true as well. Further, when $b \rightarrow \infty$, we must have $\prox_{b \rho}(\tau z) \rightarrow -\infty$, $b \rho'(\prox_{b \rho}(\tau z)) \rightarrow \infty$, such that the difference of these two is $\tau z$. Observe that for the logistic, $\rho''(x)=\rho'(x)(1-\rho'(x))$ and for the probit, $\rho''(x) = \rho'(x)(\rho'(x)-x)$ \cite{sampford1953some}. Hence, combining the asymptotic behavior of $\prox_{b \rho}(\tau z)$ and $b \rho'(\prox_{b \rho}(\tau z)) $, we obtain that $b\rho''(\prox_{b \rho}(\tau z))$ diverges to $\infty$ in both  models when $b \rightarrow \infty$. 
\end{itemize}
\end{remark}

\section{Proof of Lemma \ref{lemma: varmap}}
\label{sub:Proof-Lemma-var-map}



\subsection{Proof of Part (i)}

Recall from 
\cite[Proposition 6.4]{donoho2013high}
that 
\begin{eqnarray}
	\kappa & = & \mathbb{E}\left[\Psi'\left(\tau Z;\text{ }b({\tau}) \right)\right]=1-\mathbb{E}\left[\frac{1}{1+b({\tau})\rho''\big(\mathsf{prox}_{b({\tau}) \rho}\left(\tau Z\right)\big)}\right].\label{eq:SE-b-proof}
\end{eqnarray}
If we denote $c:= \prox_{b\rho}(0)$, then $b(0) $ is given by the following relation: 
\[
  1- \kappa = \frac{1}{1+b(0) \rho''(c)} \quad \implies \quad b(0) = \frac{\kappa}{\rho''(c)(1-\kappa)} > 0 
\]
as $\rho''(c)>0$ for any given $c>0$. In addition, since $\rho'(c) > 0$, we have 
\[\Nu(0) = \frac{\Psi(0,b(0))^2}{\kappa} ~\overset{(\text{a})}{=}~ \frac{b(0)^2\rho'(c)^2}{\kappa} > 0, \]
where (a) comes from (\ref{eq:defn-psi}). 


\subsection{Proof of Part (ii)}

We defer the proof of this part to the supplemental materials \cite{LRTsupp2017}.



\section{Proof of Part (ii) of Theorem \ref{thm: normbound-MLE}}
\label{sub:Proof-Theorem-normbound}

As discussed in Section \ref{subsec:Proof-of-Part-II-PT}, it suffices
to (1) construct a set $\left\{ \mathcal{B}_{i}\mid1\leq i\leq N\right\} $
that forms a cover of the cone $\mathcal{A}$ defined in (\ref{eq:defn-event-A}),
and (2) upper bound $\mathbb{P}\left\{ \left\{ \bm{X}\bm{\beta}\mid\bm{\beta}\in\mathbb{R}^{p}\right\} \cap\mathcal{B}_{i}\neq\left\{ \bm{0}\right\} \right\} $.
In what follows, we elaborate on these two steps.
\begin{itemize}
\item \textbf{Step 1.} Generate $N=\exp\left(2\epsilon^{2}p\right)$ i.i.d.~points
$\bm{z}^{(i)}\sim\mathcal{N}(\bm{0},\frac{1}{p}\bm{I}_{p})$, $1\leq i\leq N$,
and construct a collection of convex cones
\[
\mathcal{C}_{i}:=\left\{ \bm{u}\in\mathbb{R}^{p}\left|\left\langle \bm{u},\frac{\bm{z}^{(i)}}{\|\bm{z}^{(i)}\|}\right\rangle \geq\epsilon\|\bm{u}\|\right.\right\} ,\qquad1\leq i\leq N.
\]
In words, $\mathcal{C}_{i}$ consists of all directions that have
nontrivial positive correlation with $\bm{z}^{(i)}$. With high
probability, this collection
$\left\{ \mathcal{C}_{i}\mid1\leq i\leq N\right\} $ forms a cover of
$\mathbb{R}^{p}$, a fact which is an immediate consequence of the
following
lemma. \begin{lemma}\label{lem:uniform-inner-product}Consider any
  given constant $0<\epsilon<1$, and let
  $N=\exp\left(2\epsilon^{2}p\right)$.  Then there exist some positive
  universal constants $c_{5},C_{5}>0$ such that with probability
  exceeding $1-C_{5}\exp\left(-c_{5}\epsilon^{2}p\right)$,
\[
\sum_{i=1}^{N}\bm{1}_{\left\{ \left\langle \bm{x},\bm{z}^{(i)}\right\rangle \geq\epsilon\|\bm{x}\|\|\bm{z}^{(i)}\|\right\} }\geq1
\]
holds simultaneously for all $\bm{x} \in \R^p$.\end{lemma} With our
family $\left\{ \mathcal{C}_{i}\mid1\leq i\leq N\right\} $ we can
introduce
\begin{equation}
\mathcal{B}_{i}:=\mathcal{C}_{i}\cap\left\{ \bm{u}\in\mathbb{R}^{n}\mid\sum_{j=1}^{n}\max\left\{ -u_{j},0\right\} \leq\epsilon\sqrt{n}\left\langle \bm{u},\frac{\bm{z}^{(i)}}{\|\bm{z}^{(i)}\|}\right\rangle \right\} ,\quad1\leq i\leq N,\label{eq:defn-Bi}
\end{equation}
which in turn forms a cover of the nonconvex cone $\mathcal{A}$
defined in (\ref{eq:defn-event-A}). To justify this, note that for any
$\bm{u}\in\mathcal{A}$, one can find $i\in\{1,\cdots,N\}$ obeying
$\bm{u}\in\mathcal{C}_{i}$, or equivalently,
$\left\langle
  \bm{u},\frac{\bm{z}^{(i)}}{\|\bm{z}^{(i)}\|}\right\rangle
\geq\epsilon\|\bm{u}\|$,
with high probability.  Combined with the membership to $\mathcal{A}$
this gives 
\[
\sum_{j=1}^{n}\max\left\{ -u_{j},0\right\} \leq\epsilon^{2}\sqrt{n}\|\bm{u}\|\leq\epsilon\sqrt{n}\left\langle \bm{u},\frac{\bm{z}^{(i)}}{\|\bm{z}^{(i)}\|}\right\rangle ,
\]
indicating that $\bm{u}$ is contained within some $\mathcal{B}_{i}$. 
\item \textbf{Step 2.} We now move on to control $\mathbb{P}\left\{ \left\{ \bm{X}\bm{\beta}\mid\bm{\beta}\in\mathbb{R}^{p}\right\} \cap\mathcal{B}_{i}\neq\left\{ \bm{0}\right\} \right\} $.
If the statistical dimensions of the two cones obey $\delta\left(\mathcal{B}_{i}\right)<n-\delta\left(\left\{ \bm{X}\bm{\beta}\mid\bm{\beta}\in\mathbb{R}^{p}\right\} \right)=n-p$,
then an application of \cite[Theorem I]{amelunxen2014living} gives
\begin{eqnarray}
\mathbb{P}\left\{ \left\{ \bm{X}\bm{\beta}\mid\bm{\beta}\in\mathbb{R}^{p}\right\} \cap\mathcal{B}_{i}\neq\left\{ \bm{0}\right\} \right\}  & \leq & 4\exp\left\{ -\frac{1}{8}\left(\frac{n-\delta\left(\left\{ \bm{X}\bm{\beta}\mid\bm{\beta}\in\mathbb{R}^{p}\right\} \right)-\delta\left(\mathcal{B}_{i}\right)}{\sqrt{n}}\right)^{2}\right\} \nonumber \\
 & \leq & 4\exp\left\{ -\frac{\left(n-p-\delta(\mathcal{B}_{i})\right)^{2}}{8n}\right\} .\label{eq:intersection-Bi}
\end{eqnarray}
It then comes down to upper bounding $\delta(\mathcal{B}_{i})$, which
is the content of the following lemma. \begin{lemma}\label{lem:stat-dim-Bi}Fix
$\epsilon>0$. When $n$ is sufficiently large, the statistical dimension
of the convex cone $\mathcal{B}_{i}$ defined in (\ref{eq:defn-Bi})
obeys
\begin{eqnarray}
\delta(\mathcal{B}_{i}) & \leq & \left(\frac{1}{2}+2\sqrt{2}\epsilon^{\frac{3}{4}}+10H(2\sqrt{\epsilon})\right)n,\label{eq:stat-dim-Bi-UB}
\end{eqnarray}
where $H(x):=-x\log x-(1-x)\log(1-x)$.\end{lemma} Substitution into
(\ref{eq:intersection-Bi}) gives
\begin{eqnarray}
\mathbb{P}\left\{ \left\{ \bm{X}\bm{\beta}\mid\bm{\beta}\in\mathbb{R}^{p}\right\} \cap\mathcal{B}_{i}\neq\left\{ \bm{0}\right\} \right\}  & \leq & 4\exp\left\{ -\frac{\left(\left(\frac{1}{2}-2\sqrt{2}\epsilon^{\frac{3}{4}}-10H(2\sqrt{\epsilon})\right)n-p\right)^{2}}{8n}\right\} \nonumber \\
 & = & 4\exp\left\{ -\frac{1}{8}\left(\frac{1}{2}-2\sqrt{2}\epsilon^{\frac{3}{4}}-10H(2\sqrt{\epsilon})-\frac{p}{n}\right)^{2}n\right\} .
\end{eqnarray}

\end{itemize}
Finally, we prove Lemmas
\ref{lem:uniform-inner-product}-\ref{lem:stat-dim-Bi} in the next
subsections. These are the only remaining parts for the proof of
Theorem \ref{thm: normbound-MLE}.

\subsection{Proof of Lemma \ref{lem:uniform-inner-product}}

To begin with, it is seen that all $\|\bm{z}^{(i)}\|$ concentrates
around $1$. Specifically, apply \cite[Proposition 1]{hsu2012tail}
to get
\[
\mathbb{P}\left\{ \|\bm{z}^{(i)}\|^{2}>1+2\sqrt{\frac{t}{p}}+\frac{2t}{p}\right\} \leq e^{-t},
\]
and set $t=3\epsilon^{2}p$ to reach
\[
\mathbb{P}\left\{ \|\bm{z}^{(i)}\|^{2}>1+10\epsilon\right\} \text{ }\leq\text{ }\mathbb{P}\left\{ \|\bm{z}^{(i)}\|^{2}>1+2\sqrt{3}\epsilon+6\epsilon^{2}\right\} \text{ }\leq\text{ }e^{-3\epsilon^{2}p}.
\]
Taking the union bound we obtain
\begin{eqnarray}
\mathbb{P}\left\{ \exists1\leq i\leq N\text{ s.t. }\|\bm{z}^{(i)}\|^{2}>1+10\epsilon\right\}  & \leq & Ne^{-3\epsilon^{2}p}=e^{-\epsilon^{2}p}.\label{eq:zi-norm-bound}
\end{eqnarray}

Next, we note that it suffices to prove Lemma \ref{lem:uniform-inner-product}
for all unit vectors $\bm{x}$. The following lemma provides a bound
on $\left\langle \bm{z}^{(i)},\bm{x}\right\rangle $ for any fixed
unit vector $\bm{x}\in\mathbb{R}^{p}$. 

\begin{lemma}\label{lem:z-fixed-x}Consider any fixed unit vector
$\bm{x}\in\mathbb{R}^{p}$ and any given constant $0<\epsilon<1$,
and set $N=\exp\left(2\epsilon^{2}p\right)$. There exist positive
universal constants $c_{5},c_{6},C_{6}>0$ such that
\begin{align}
 & \mathbb{P}\left\{ \sum_{i=1}^{N}\bm{1}_{\left\{ \left\langle \bm{z}^{(i)},\bm{x}\right\rangle \geq\frac{1}{2}\epsilon\right\} }\leq\exp\left(\left(1-o\left(1\right)\right)\frac{7}{4}\epsilon^{2}p\right)\right\} \leq \exp\left\{ -2\exp\left(\left(1-o\left(1\right)\right)\frac{7}{4}\epsilon^{2}p\right)\right\} .\label{eq:large-delta1}
\end{align}

\end{lemma}

Recognizing that Lemma \ref{lem:uniform-inner-product} is a uniform
result, we need to extend Lemma \ref{lem:z-fixed-x} to all $\bm{x}$
simultaneously, which we achieve via the standard covering argument.
Specifically, one can find a set $\mathcal{C}:=\left\{ \bm{x}^{(j)}\in\mathbb{R}^{p}\mid1\leq j\leq K\right\} $
of unit vectors with cardinality $K=\left(1+2p^{2}\right)^{p}$ to
form a cover of the unit ball of resolution $p^{-2}$ \cite[Lemma 5.2]{vershynin2010introduction};
that is, for any unit vector $\bm{x}\in\mathbb{R}^{p}$, there exists
a $\bm{x}^{(j)}\in\mathcal{C}$ such that
\[
\|\bm{x}^{(j)}-\bm{x}\|\leq p^{-2}.
\]
Apply Lemma \ref{lem:z-fixed-x} and take the union bound to arrive
at
\begin{align}
\sum_{i=1}^{N}\bm{1}_{\left\{ \left\langle \bm{z}^{(i)},\bm{x}^{(j)}\right\rangle \geq\frac{1}{2}\epsilon\right\} }\geq\exp\left(\left(1-o(1)\right)\frac{7}{4}\epsilon^{2}p\right)>1, & \qquad1\leq j\leq K
\end{align}
with probability exceeding $1-K\exp\left\{ -2\exp\left(\left(1-o(1)\right)\frac{7}{4}\epsilon^{2}p\right)\right\} \geq1-\exp\left\{ -2\left(1-o\left(1\right)\right)\exp\left(\left(1-o(1)\right)\frac{7}{4}\epsilon^{2}p\right)\right\} $.
This guarantees that for each $\bm{x}^{(j)}$, one can find at least
one $\bm{z}^{(i)}$ obeying 
\[
\left\langle \bm{z}^{(i)},\bm{x}^{(j)}\right\rangle \geq\frac{1}{2}\epsilon.
\]
This result together with (\ref{eq:zi-norm-bound}) yields that with
probability exceeding $1-C\exp\left(-c\epsilon^2 p\right)$, for some universal constants $C,c>0$. 
\begin{eqnarray*}
\left\langle \bm{z}^{(i)},\bm{x}\right\rangle  \geq  \left\langle
  \bm{z}^{(i)},\bm{x}^{(j)}\right\rangle -\left\langle
  \bm{z}^{(i)},\bm{x}^{(j)}-\bm{x}\right\rangle & \geq & \left\langle \bm{z}^{(i)},\bm{x}^{(j)}\right\rangle -\|\bm{z}^{(i)}\|\cdot\|\bm{x}^{(j)}-\bm{x}\| \\
	& \geq & \frac{1}{2}\epsilon-\frac{1}{p^{2}}\|\bm{z}^{(i)}\|
          \geq   \frac{\frac{1}{2}\epsilon}{\sqrt{1+10\epsilon}}\|\bm{z}^{(i)}\|-\frac{1}{p^{2}}\|\bm{z}^{(i)}\| \\
	& \geq & \frac{1}{30}\epsilon\|\bm{z}^{(i)}\|
\end{eqnarray*}
holds simultaneously for all unit vectors $\bm{x}\in\mathbb{R}^{p}$.
Since $\epsilon>0$ can be an arbitrary constant, this concludes the
proof.  

\begin{proof}[\textbf{Proof of Lemma \ref{lem:z-fixed-x}}]Without
loss of generality, it suffices to consider $\bm{x}=\bm{e}_{1}=[1,0,\cdots,0]^{\top}$.
For any $t>0$ and any constant $\zeta>0$, it comes from \cite[Theorem A.1.4]{alon2008probabilistic}
that
\begin{align*}
 & \mathbb{P}\left\{ \frac{1}{N}\sum_{i=1}^{N}\bm{1}_{\left\{ \left\langle \bm{z}^{(i)},\bm{e}_{1}\right\rangle <\zeta\right\} }>\left(1+t\right)\Phi\left(\zeta\sqrt{p}\right)\right\} \leq \exp\left(-2t^{2}\Phi^2\left(\zeta\sqrt{p}\right)N\right).
\end{align*}
Setting $t=1-\Phi\left(\zeta\sqrt{p}\right)$ gives
\[
\mathbb{P}\left\{ \frac{1}{N}\sum_{i=1}^{N}\bm{1}_{\left\{ \left\langle \bm{z}^{(i)},\bm{e}_{1}\right\rangle <\zeta\right\} }>\left(2-\Phi\left(\zeta\sqrt{p}\right)\right)\Phi\left(\zeta\sqrt{p}\right)\right\} \leq \exp\left(-2\left(1-\Phi\left(\zeta\sqrt{p}\right)\right)^{2}\Phi^2\left(\zeta\sqrt{p}\right)N\right).
\]

Recall that for any $t>1$, one has
$(t^{-1} -t^{-3})\phi(t) \leq1-\Phi(t)\leq t^{-1} \phi(t)$ which
implies that
\[
1-\Phi\left(\zeta\sqrt{p}\right)=\exp\left(-\frac{\left(1+o\left(1\right)\right)\zeta^{2}p}{2}\right).
\]
Taking $\zeta=\frac{1}{2}\epsilon$, we arrive at
\begin{eqnarray*}
\left(2-\Phi\left(\zeta\sqrt{p}\right)\right)\Phi\left(\zeta\sqrt{p}\right) & = & 1-\exp\left(-\left(1+o\left(1\right)\right)\zeta^{2}p\right)=1-\exp\left(-\left(1+o\left(1\right)\right)\frac{1}{4}\epsilon^{2}p\right),\\
\left(1-\Phi\left(\zeta\sqrt{p}\right)\right)^{2}\Phi^2\left(\zeta\sqrt{p}\right) & = & \exp\left(-\left(1+o\left(1\right)\right)\zeta^{2}p\right)=\exp\left(-\left(1+o\left(1\right)\right)\frac{1}{4}\epsilon^{2}p\right)\gg\frac{1}{N}.
\end{eqnarray*}
This justifies that
\begin{align*}
  \mathbb{P}\left\{ \sum_{i=1}^{N}\bm{1}_{\left\{ \left\langle \bm{z}^{(i)},\bm{e}_{1}\right\rangle \geq\frac{1}{2}\epsilon\right\} }\leq N\exp\left(-\left(1+o\left(1\right)\right)\frac{1}{4}\epsilon^{2}p\right)\right\} 
 & =\mathbb{P}\left\{ \frac{1}{N}\sum_{i=1}^{N}\bm{1}_{\left\{ \left\langle \bm{z}^{(i)},\bm{e}_{1}\right\rangle <\zeta\right\} }>\left(2-\Phi\left(\zeta\sqrt{p}\right)\right)\Phi\left(\zeta\sqrt{p}\right)\right\} \\
 &  \leq \exp\left\{
   -2\exp\left(-\left(1+o\left(1\right)\right)\frac{1}{4}\epsilon^{2}p\right)N\right\}
  \\
& =\exp\left\{ -2\exp\left(\left(1-o\left(1\right)\right)\frac{7}{4}\epsilon^{2}p\right)\right\} 
\end{align*}
as claimed. 
\end{proof}

\subsection{Proof of Lemma \ref{lem:stat-dim-Bi}}

First of all, recall from the definition (\ref{eq:defn-statistical-dimension})
that
\begin{eqnarray*}
\delta(\mathcal{B}_{i}) & = & \mathbb{E}\left[\left\Vert \Pi{}_{\mathcal{B}_{i}}\left(\bm{g}\right)\right\Vert ^{2}\right]=\mathbb{E}\left[\left\Vert \bm{g}\right\Vert ^{2}-\min_{\bm{u}\in\mathcal{B}_{i}}\left\Vert \bm{g}-\bm{u}\right\Vert ^{2}\right]=n-\mathbb{E}\left[\min_{\bm{u}\in\mathcal{B}_{i}}\left\Vert \bm{g}-\bm{u}\right\Vert ^{2}\right]\\
 & \leq & n-\mathbb{E}\left[\min_{\bm{u}\in\mathcal{D}_{i}}\left\Vert \bm{g}-\bm{u}\right\Vert ^{2}\right],
\end{eqnarray*}
where $\bm{g}\sim\mathcal{N}\left(\bm{0},\bm{I}_{n}\right)$, and
$\mathcal{D}_{i}$ is a superset of $\mathcal{B}_{i}$ defined by
\begin{equation}
\mathcal{D}_{i}:=\left\{ \bm{u}\in\mathbb{R}^{n}\mid\sum\nolimits _{j=1}^{n}\max\left\{ -u_{j},0\right\} \leq\epsilon\sqrt{n}\|\bm{u}\|\right\} .\label{eq:defn-Di}
\end{equation}
Recall from the triangle inequality that
\begin{eqnarray*}
\left\Vert \bm{g}-\bm{u}\right\Vert  & \geq & \|\bm{u}\|-\|\bm{g}\|>\|\bm{g}\|=\|\bm{g}-\bm{0}\|,\qquad\forall\bm{u}:\text{ }\|\bm{u}\|>2\|\bm{g}\|.
\end{eqnarray*}
Since $\bm{0}\in\mathcal{D}_{i}$, this implies that
\[
\Big\|\arg\min_{\bm{u}\in\mathcal{D}_{i}}\|\bm{g}-\bm{u}\|\Big\|\leq2\|\bm{g}\|,
\]
revealing that
\[
\mathbb{E}\left[\min_{\bm{u}\in\mathcal{D}_{i}}\left\Vert \bm{g}-\bm{u}\right\Vert ^{2}\right]=\mathbb{E}\left[\min_{\bm{u}\in\mathcal{D}_{i},\|\bm{u}\|\leq2\|\bm{g}\|}\left\Vert \bm{g}-\bm{u}\right\Vert ^{2}\right].
\]

In what follows, it suffices to look at the set of $\bm{u}$'s within
$\mathcal{D}_{i}$ obeying $\|\bm{u}\|\leq2\|\bm{g}\|$, which verify 
\begin{equation}
\sum\nolimits _{j=1}^{n}\max\left\{ -u_{j},0\right\} \leq\epsilon\sqrt{n}\|\bm{u}\|\leq2\epsilon\sqrt{n}\|\bm{g}\|.\label{eq:u-constraint-g}
\end{equation}
It is seen that
\begin{eqnarray}
\|\bm{g}-\bm{u}\|^{2} & \geq & \sum_{i:g_{i}<0}\left(g_{i}-u_{i}\right)^{2}=\left\{ \sum_{i:g_{i}<0,u_{i}\geq0}+\sum_{i:g_{i}<0,\text{ }-\sqrt{\frac{\epsilon}{n}}\|\bm{g}\|<u_{i}<0}+\sum_{i:g_{i}<0,\text{ }u_{i}\leq-\sqrt{\frac{\epsilon}{n}}\|\bm{g}\|}\right\} \left(g_{i}-u_{i}\right)^{2}\nonumber \\
 & \geq & \sum_{i:g_{i}<0,u_{i}\geq0}g_{i}^{2}+\sum_{i:g_{i}<0,\text{ }-\sqrt{\frac{\epsilon}{n}}\|\bm{g}\|<u_{i}<0}\left(g_{i}-u_{i}\right)^{2}\nonumber \\
 & \geq & \sum_{i:g_{i}<0,u_{i}\geq0}g_{i}^{2}+\sum_{i:g_{i}<0,\text{ }-\sqrt{\frac{\epsilon}{n}}\|\bm{g}\|<u_{i}<0}\left(g_{i}^{2}-2u_{i}g_{i}\right)\nonumber \\
 & \geq & \sum_{i:g_{i}<0,\text{ }u_{i}>-\sqrt{\frac{\epsilon}{n}}\|\bm{g}\|}g_{i}^{2}-\sum_{i:g_{i}<0,\text{ }-\sqrt{\frac{\epsilon}{n}}\|\bm{g}\|<u_{i}<0}2u_{i}g_{i}.\label{eq:g-u-LB}
\end{eqnarray}

\begin{enumerate}
\item Regarding the first term of (\ref{eq:g-u-LB}), we first recognize
that
\[
\left\{ i\mid u_{i}\leq-\sqrt{\frac{\epsilon}{n}}\|\bm{g}\|\right\} \leq\frac{\sum_{i:\text{ }u_{i}<0}|u_{i}|}{\sqrt{\frac{\epsilon}{n}}\|\bm{g}\|}=\frac{\sum_{i=1}^{n}\max\left\{ -u_{i},0\right\} }{\sqrt{\frac{\epsilon}{n}}\|\bm{g}\|}\leq2\sqrt{\epsilon}n,
\]
where the last inequality follows from the constraint (\ref{eq:u-constraint-g}).
As a consequence, 
\begin{eqnarray*}
\sum_{i:g_{i}<0,\text{ }u_{i}>-\sqrt{\frac{\epsilon}{n}}\|\bm{g}\|}g_{i}^{2} & \geq & \sum_{i:g_{i}<0}g_{i}^{2}-\sum_{i:u_{i}\leq-\sqrt{\frac{\epsilon}{n}}\|\bm{g}\|}g_{i}^{2}\\
 & \geq & \sum_{i:g_{i}<0}g_{i}^{2}-\max_{S\subseteq[n]:\text{
          }|S|=2\sqrt{\epsilon}n}\sum_{i\in S}g_{i}^{2}. 
\end{eqnarray*}

\item Next, we turn to the second term of (\ref{eq:g-u-LB}), which can
be bounded by
\begin{eqnarray*}
\sum_{i:g_{i}<0,\text{ }-\sqrt{\frac{\epsilon}{n}}\|\bm{g}\|<u_{i}<0}u_{i}g_{i} & \leq & \sqrt{\left(\sum_{i:g_{i}<0,\text{ }-\sqrt{\frac{\epsilon}{n}}\|\bm{g}\|<u_{i}<0}u_{i}^{2}\right)\left(\sum_{i:g_{i}<0,\text{ }-\sqrt{\frac{\epsilon}{n}}\|\bm{g}\|<u_{i}<0}g_{i}^{2}\right)}\\
 & \leq & \sqrt{\left(\max_{i:-\sqrt{\frac{\epsilon}{n}}\|\bm{g}\|<u_{i}<0}|u_{i}|\right)\left(\sum_{i:u_{i}<0}|u_{i}|\right)\cdot\|\bm{g}\|^{2}}\\
 & \leq & \sqrt{\sqrt{\frac{\epsilon}{n}}\|\bm{g}\|\left(\sum_{i:u_{i}<0}|u_{i}|\right)\cdot\|\bm{g}\|^{2}}\leq\sqrt{2}\epsilon^{\frac{3}{4}}\|\bm{g}\|^{2},
\end{eqnarray*}
where the last inequality follows from the constraint (\ref{eq:u-constraint-g}).
\end{enumerate}
Putting the above results together, we have 
\[
\left\Vert \bm{g}-\bm{u}\right\Vert ^{2}\geq\sum_{i:g_{i}<0}g_{i}^{2}-\max_{S\subseteq[n]:\text{ }|S|=2\sqrt{\epsilon}n}\sum_{i\in S}g_{i}^{2}-2\sqrt{2}\epsilon^{\frac{3}{4}}\|\bm{g}\|^{2}
\]
for any $\bm{u}\in\mathcal{D}_{i}$ obeying $\|\bm{u}\|\leq2\|\bm{g}\|$,
whence
\begin{eqnarray}
\mathbb{E}\left[\min_{\bm{u}\in\mathcal{D}_{i}}\left\Vert \bm{g}-\bm{u}\right\Vert ^{2}\right] & \geq & \mathbb{E}\left[\sum_{i:g_{i}<0}g_{i}^{2}-\max_{S\subseteq[n]:\text{ }|S|=2\sqrt{\epsilon}n}\sum_{i\in S}g_{i}^{2}-2\sqrt{2}\epsilon^{\frac{3}{4}}\|\bm{g}\|^{2}\right]\nonumber \\
 & = & \left(\frac{1}{2}-2\sqrt{2}\epsilon^{\frac{3}{4}}\right)n-\mathbb{E}\left[\max_{S\subseteq[n]:\text{ }|S|=2\sqrt{\epsilon}n}\sum_{i\in S}g_{i}^{2}\right].\label{eq:g-u-LB2}
\end{eqnarray}

Finally, it follows from \cite[Proposition 1]{hsu2012tail} that for
any $t>2\sqrt{\epsilon}n$, 
\[
\mathbb{P}\left\{ \sum_{i\in S}g_{i}^{2}\geq5t\right\} \leq\mathbb{P}\left\{ \sum_{i\in S}g_{i}^{2}\geq |S|+2\sqrt{|S|t}+2t\right\} \leq e^{-t},
\]
which together with the union bound gives
\[
\mathbb{P}\left\{ \max_{S\subseteq[n]:\text{ }|S|=
    2\sqrt{\epsilon}n}\sum_{i\in S}g_{i}^{2}\geq5t\right\}
\leq\sum_{S\subseteq[n]:\text{
  }|S|=2\sqrt{\epsilon}n}\mathbb{P}\left\{ \sum_{i\in
    S}g_{i}^{2}\geq5t\right\} \leq\exp\left\{
  H\left(2\sqrt{\epsilon}\right)n-t\right\}. 
\]
This gives 
\begin{eqnarray*}
\mathbb{E}\left[\max_{S\subseteq[n]:\text{ }|S|=2\sqrt{\epsilon}n}\sum_{i\in S}g_{i}^{2}\right] & = & {\displaystyle \int}_{0}^{\infty}\mathbb{P}\left\{ \max_{S\subseteq[n]:\text{ }|S|= 2\sqrt{\epsilon}n}\sum_{i\in S}g_{i}^{2}\geq t\right\} \mathrm{d}t\\
 & \leq & 5H\left(2\sqrt{\epsilon}\right)n+{\displaystyle \int}_{5H\left(2\sqrt{\epsilon}\right)n}^{\infty}\exp\left\{ H\left(2\sqrt{\epsilon}\right)n-\frac{1}{5}t\right\} \mathrm{d}t\\
 & < & 10H\left(2\sqrt{\epsilon}\right)n,
\end{eqnarray*}
for any given $\epsilon >0$ with the proviso that $n$ is sufficiently
large. This combined with (\ref{eq:g-u-LB2}) yields
\begin{eqnarray}
\mathbb{E}\left[\min_{\bm{u}\in\mathcal{D}_{i}}\left\Vert \bm{g}-\bm{u}\right\Vert ^{2}\right] & \geq & \left(\frac{1}{2}-2\sqrt{2}\epsilon^{\frac{3}{4}}-10H(2\sqrt{\epsilon})\right)n
\end{eqnarray}
as claimed.

\section{Proof of Lemma \ref{lemma: svd}}
\label{sub:Proof-of-Lemmas-svd}

Throughout, we shall restrict ourselves on the event $\A_n$ as defined in \eqref{eq:defnAn}, on which $\tbG\succeq  \lambda_{\mathrm{lb}}\bm{I}$. Recalling the definitions of $\tbG$ and $\bw$ from  \eqref{eq: defGtilde} and  \eqref{eq: defw}, we see that
\begin{align}
	\bw^{\tp}\tbG^{-2}\bw &= ~\frac{1}{n^2} \Xcone^{\tp}\Dtbeta \tbX \left(\frac{1}{n} \tbX^{\tp}\Dtbeta \tbX\right)^{-2} \tbX^{\tp} \Dtbeta \Xcone  \nonumber \\ 
			      &\leq  ~ \frac{\big\| \Xcone^{\tp} \big\|^2}{n}  \left\| \frac{1}{n} \Dtbeta \tbX \left(\frac{1}{n} \tbX^{\tp}\Dtbeta \tbX\right)^{-2} \tbX^{\tp} \Dtbeta \right \|.  \label{eq:wGw-UB}
\end{align}
	If we let the singular value decomposition of $\frac{1}{\sqrt{n}}\Dtbeta^{1/2} \tbX $ be $\bU \bSigma \bV^{\tp}$, then a little algebra gives $\bSigma \succeq \sqrt{\lambda_{\mathrm{lb}} }\bm{I}$ and
	\begin{align*}
\frac{1}{n} \Dtbeta^{1/2} \tbX \left( \frac{1}{n} \tbX'\Dtbeta \tbX \right)^{-2} \tbX^{\tp} \Dtbeta^{1/2} 
	&  = \bU \bSigma^{-2} \bU^{\tp}  ~\preceq~ \lambda_{\mathrm{lb}}^{-1} \bm{I}. \ \  
\end{align*} 
Substituting this  into (\ref{eq:wGw-UB}) and using the fact $\| \Xcone \|^2 \lesssim n$ with high probability  (by Lemma \ref{lemma: singval}), we obtain
\begin{align*}
\bw^{\tp} \tbG^{-2} \bw  ~\lesssim~ \frac{1}{n\lambda_L}  \| \Xcone \|^2 \lesssim 1
\end{align*} 
	with probability at least $1-\exp(-\Omega(n))$.

\section{Proof of Lemma \ref{lemma: xinapprox}}
\label{sec:Proof-ofLemma-xinapprox}
Throughout this and the subsequent sections, we consider $H_n$ and $K_n$ to be two diverging sequences with the following properties:
 \begin{equation}\label{eq: HnKn}
  H_n=o\left(n^\epsilon\right), \ \ K_n=o\left(n^\epsilon\right), \ \ n^2\exp\left(-c_1 H_n^2\right) = o(1), \ \   n\exp\left(-c_2K_n^2\right) =o(1), 
 \end{equation}
 for any constants $c_i > 0$, $i=1,2$ and any $\epsilon>0$.  This
 lemma is an analogue of \cite[Proposition 3.18]{el2015impact}. We
 modify and adapt the proof ideas to establish the result in our
 setup.  Throughout we shall restrict ourselves to the event $\A_n$,
 on which $\tbG \succeq \lambda_{\mathrm{lb}}\bm{I}$.
	
	Due to independence between $\Xcone$ and $\{\Dtbeta, \bH\}$, one can invoke the Hanson-Wright inequality \cite[Theorem 1.1]{rudelson2013hanson} to yield 
\begin{multline*}
	 \opP\left( \left| \frac{1}{n}  \Xcone^{\tp} \Dtbeta ^{1/2} \bH \Dtbeta^{1/2} \Xcone  - \frac{1}{n} \tr \left(\Dtbeta^{1/2} \bH \Dtbeta^{1/2} \right) \right| > t ~\Bigg|~ \bH, \Dtbeta\right)  \\
		 \leq 2 \exp \left (-c \min \left \{\frac{t^2}{\frac{K^4}{n^2} \big\|\Dtbeta^{1/2}\bH \Dtbeta^{1/2} \big\|^2_{\mathrm{F}}}, \frac{t}{\frac{K^2}{n} \big\|\Dtbeta^{1/2}\bH \Dtbeta^{1/2} \big\|} \right \}   \right )  \\
	  \leq 2 \exp \left (-c \min \left \{\frac{t^2}{\frac{K^4}{n} \big\|\Dtbeta^{1/2}\bH \Dtbeta^{1/2} \big\|^2}, \frac{t}{\frac{K^2}{n} \big\|\Dtbeta^{1/2}\bH \Dtbeta^{1/2} \big\|} \right \}   \right ),
\end{multline*}
	where $\|.\|_{\mathrm{F}}$ denotes the Frobenius norm. 
Choose $t = C^2 \big\|\Dtbeta^{1/2}\bH \Dtbeta^{1/2} \big\| H_n/\sqrt{n}$ with $C>0$ a sufficiently large constant, and take $H_n$ to be as in \eqref{eq: HnKn}. Substitution into the above inequality and unconditioning give
\begin{align}\label{eq: xintraceapprox}
	& \opP\left( \left| \frac{1}{n}  \Xcone^{\tp} \Dtbeta ^{1/2} \bH \Dtbeta^{1/2} \Xcone  - \frac{1}{n} \tr \left(\Dtbeta^{1/2} \bH \Dtbeta^{1/2} \right) \right| > \frac{1}{\sqrt{n}} C^2 H_n \|\Dtbeta^{1/2}\bH \Dtbeta^{1/2} \|  \right) \nonumber \\
	& \qquad \leq ~2\exp \left(-c \min\left \{\frac{C^4 H_n^2}{K^4}, \frac{C^2 \sqrt{n}H_n}{K^2} \right \} \right) = C\exp\left(-c H_n^2  \right) = o(1),
\end{align}
for some universal constants $C, c>0$.

We are left to analyzing
$\tr\big( \Dtbeta^{1/2} \bH \Dtbeta^{1/2} \big)$. Recall from the
definition \eqref{eq: defH} of $\bH$ that
\[\Dtbeta^{1/2} \bH \Dtbeta^{1/2} = \Dtbeta - \frac{1}{n} \Dtbeta \tbX \tbG^{-1} \tbX^{\tp} \Dtbeta ,  \]
and, hence, 
\begin{equation} 
	\label{eq: trace1}
	\tr\left(\Dtbeta^{1/2}\bH \Dtbeta^{1/2}\right) = \sum_{i=1}^n \left(\rho''(\tbX_i^{\tp}\tbbeta) - \frac{\rho''(\tbX_i^{\tp}\tbbeta)^2}{n} \tbX_i^{\tp}\tbG^{-1} \tbX_i \right) .
\end{equation}
This requires us to analyze  $\tbG^{-1}$  carefully. To this end, recall that the matrix $\tbGi$  defined  in \eqref{eq: defGi} obeys
\[
	\tbGi = \tbG - \frac{1}{n} \rho''(\tbX^{\tp}\tbbeta) \tbX_i \tbX_i^{\tp } .
\]
Invoking Sherman-Morrison-Woodbury formula (e.g.~\cite{hager1989updating}), we have 
\begin{align}\label{eq: SMWapp}
\tbG^{-1} & = \tbGi^{-1} -\frac{\frac{\rho''(\tbX_i^{\tp}\tbbeta)}{n}\tbGi^{-1} \tbX_i \tbX_i^{\tp} \tbGi^{-1}}{1+\frac{\rho''(\tbX_i^{\tp}\tbbeta)}{n}\tbX_i^{\tp}\tbGi^{-1}\tbX_i} .
\end{align}
It follows that 
\[\tbX_i^{\tp}\tbG^{-1} \tbX_i  = \tbX_i^{\tp}\tbGi^{-1} \tbX_i - \frac{\frac{\rho''(\tbX_i^{\tp}\tbbeta)}{n}(\bX_i^{\tp}\tbGi^{-1} \tbX_i )^2}{1+\frac{\rho''(\tbX_i^{\tp}\tbbeta)}{n}\tbX_i^{\tp}\tbGi^{-1}\tbX_i},   \]
which implies that 
\begin{equation}\label{eq: trace2}
	\tbX_i^{\tp}\tbG^{-1} \tbX_i  = \frac{\tbX_i^{\tp}\tbGi^{-1}\tbX_i}{1+\frac{\rho''(\tbX_i^{\tp}\tbbeta)}{n}\tbX_i^{\tp}\tbGi^{-1}\tbX_i}.
\end{equation}
The relations \eqref{eq: trace1} and \eqref{eq: trace2} taken collectively reveal that 
\begin{equation}\label{eq: traceapprox1}
\frac{1}{n}\tr\left (\Dtbeta^{1/2} \bH \Dtbeta^{1/2} \right) = \frac{1}{n} \sum_{i=1}^n \frac{\rho''(\tbX_i \tbbeta)}{1+\frac{\rho''(\tbX_i^{\tp}\tbbeta)}{n}\tbX_i^{\tp}\tbGi^{-1}\tbX_i}. 
\end{equation}

We shall show that the trace above is close to  $\tr(\Id - \bH)$ up to
some factors. For this purpose we analyze the latter quantity in two
different ways. To begin with, observe that 
\begin{equation} \label{eq: traceapprox2}
\tr(\Id-\bH)= \tr\Bigg(\frac{\Dtbeta^{1/2}\tbX\tbG^{-1}\tbX^{\tp}\Dtbeta^{1/2}}{n} \Bigg) =\tr(\tbG\tbG^{-1})=p-1.
\end{equation}
On the other hand,  it directly follows  from the definition of $\bH$
and \eqref{eq: trace2} that the $i{\text{th}}$ diagonal entry of $\bH$ is given by
\[ H_{i,i} = \frac{1}{ 1+\frac{\rho''(\tbX_i^{\tp}\tbbeta)}{n}\tbX_i^{\tp}\tbGi^{-1}\tbX_i} . \]
Applying this relation, we can compute  $\tr(\Id-\bH)$  analytically
as follows:  
\begin{align}
 \tr(\Id-\bH) & =  \sum_{i} \frac{\frac{\rho''(\tbX_i^{\tp}\tbbeta)}{n}\tbX_i^{\tp}\tbGi^{-1}\tbX_i}{1+\frac{\rho''(\tbX_i^{\tp}\tbbeta)}{n}\tbX_i^{\tp}\tbGi^{-1}\tbX_i} \label{eq: triminh}\\
& =  \sum_{i}  \frac{\rho''(\tbX_i^{\tp}\tbbeta)\talpha + \frac{\rho''(\tbX_i^{\tp}\tbbeta)}{n}\tbX_i^{\tp}\tbGi^{-1}\tbX_i - \rho''(\tbX_i^{\tp}\tbbeta)\talpha}{1+\frac{\rho''(\tbX_i^{\tp}\tbbeta)}{n}\tbX_i^{\tp}\tbGi^{-1}\tbX_i} \nonumber\\
& = \sum_{i} \rho''(\tbX_i^{\tp}\tbbeta)\talpha H_{i,i}  +  \sum_{i}  \frac{\rho''(\tbX_i^{\tp}\tbbeta) \left( \frac{1}{n}\tbX_i^{\tp}\tbGi^{-1}\tbX_i -\talpha \right)}{1+\frac{\rho''(\tbX_i^{\tp}\tbbeta)}{n}\tbX_i^{\tp}\tbGi^{-1}\tbX_i} \label{eq: traceapprox3} ,
\end{align}
where  $\talpha := \frac{1}{n} \tr\left(\tbG^{-1}\right)$. 

Observe that the first quantity in the right-hand side above is simply $\talpha \tr\big(\Dtbeta^{1/2}\bH \Dtbeta^{1/2}\big)$. For simplicity, denote 
\begin{equation} \label{eq: defetai}
\eta_i = \frac{1}{n}\tbX_i^{\tp}\tbGi^{-1}\tbX_i -\talpha.
\end{equation}
Note that  $\tbGi\succ \bm{0}$ on $\A_n $ and that $\rho''>0$. Hence the denominator in the second term in \eqref{eq: traceapprox3} is greater than $1$ for all $i$. Comparing \eqref{eq: traceapprox2} and \eqref{eq: traceapprox3}, we deduce that
\begin{equation}
	\label{eq: traceapprox4}
\left|\frac{p-1}{n}  - \frac{1}{n}\tr\left(\Dtbeta^{1/2} \bH \Dtbeta^{1/2}\right) \talpha \right| ~\leq~ \sup_i|\eta_i| \cdot \frac{1}{n}\sum_i|\rho''(\tbX_i^{\tp}\tbbeta)| ~\lesssim~  \sup_i|\eta_i|
\end{equation}
on $\A_n$. It thus suffices to control $\sup_i|\eta_i|$. The above
bounds together with Lemma \eqref{eq: Anprob} and the proposition
below complete the proof.


\begin{proposition}
\label{prop: etai}
Let $\eta_i$ be as defined in \eqref{eq: defetai}. Then there exist universal constants $C_1,C_2, C_3 > 0$ such that 
\begin{align}
\opP \left(\sup_i |\eta_i| \leq \frac{C_1 K_n^2 H_n}{\sqrt{n}} \right) & \geq 1- C_2n^2\exp\left(-c_2 H_n^2\right)-C_3n\exp\left(-c_3K_n^2\right) \nonumber\\
&- \exp\left(-C_4n\left(1+o(1)\right)\right)=1-o(1)\nonumber,
\end{align}
where $K_n , H_n$ are diverging sequences as specified in \eqref{eq: HnKn} 
\end{proposition}

	\begin{proof}[Proof of Proposition \ref{prop: etai}]
          Fix any index $i$. Recall that $\tbbetaimin$ is the MLE when
          the $1{\text{st}}$ predictor and $i{\text{th}}$ observation
          are removed. Also recall the definition of $\tbGimin$ in
          \eqref{eq: tbGimin}.  The proof essentially follows three
          steps. First, note that $\tbX_i$ and $\tbGimin$ are
          independent. Hence, an application of the Hanson-Wright
          inequality \cite{rudelson2013hanson} yields that
 \begin{align*}
	  \opP\left( \left| \frac{1}{n}  \tbX_i^{\tp} \tbGimin^{-1} \tbX_i  - \frac{1}{n} \tr \left(\tbGimin^{-1} \right) \right| > t ~ \Bigg|~ \tbGimin\right)   
	 & \leq 2 \exp \left (-c \min \left \{\frac{t^2}{\frac{K^4}{n^2} \big\|\tbGimin^{-1} \big\|_{\mathrm{F}}^2}, \frac{t}{\frac{K^2}{n} \big\|\tbGimin^{-1} \big\|} \right \}   \right ) \\
	 & \leq 2 \exp \left (-c \min \left \{\frac{t^2}{\frac{K^4}{n} \big\|\tbGimin^{-1} \big\|^2}, \frac{t}{\frac{K^2}{n} \big\|\tbGimin^{-1} \big\|} \right \}   \right ). 
\end{align*}
We choose $t = C^2 \big\| \tbGimin^{-1} \big\| H_n/\sqrt{n}$, where $C>0$ is a sufficiently large constant.
Now marginalizing gives
\begin{align*}
	\opP\left( \left| \frac{1}{n}  \tbX_i^{\tp}\tbGimin^{-1} \tbX_i  - \frac{1}{n} \tr \left(\tbGimin^{-1} \right) \right| > C^2 \big\|\tbGimin^{-1} \big\| \frac{H_n}{\sqrt{n}} \right) 
	&\leq 2\exp \left(-c \min\left \{\frac{C^4 H_n^2}{K^4}, \frac{C^2 \sqrt{n }H_n}{K^2} \right \} \right) \\
	& \leq 2 \exp\left(-C' H_n^2 \right),
\end{align*}
where $C' > 0 $ is a sufficiently large constant. On $\A_n$, the spectral norm $\big\|\tbGi^{-1} \big\|$ is bounded above by $\lambda_{\mathrm{lb}}$ for all $i$. Invoking \eqref{eq: Anprob} we obtain that there exist universal constants $C_1, C_2, C_3>0 $ such that 
\begin{equation}
\label{eq: tbGiminapprox-11}
\opP\left(\sup_i \left| \frac{1}{n}  \tbX_i^{\tp}\tbGimin^{-1} \tbX_i  - \frac{1}{n} \tr \left(\tbGimin^{-1} \right) \right| > C_1 \frac{H_n}{\sqrt{n}} \right)  \leq  C_2 n \exp\left(-C_3 H_n^2 \right) .
\end{equation}
The next step consists of showing that $\tr \big(\tbGimin^{-1}\big) $ (resp.~$\tbX_i^{\tp} \tbGimin^{-1}\tbX_i$) and $\tr\big(\tbGi^{-1}\big) $ (resp.~$\tbX_i^{\tp} \tbGi^{-1} \tbX_i$) are uniformly close across all $i$. This is established in the following lemma. 

\begin{lemma}\label{lemma: quadtraceapprox}
	Let $\tbGi$ and $\tbGimin$  be defined as in \eqref{eq: defGi} and \eqref{eq: tbGimin}, respectively. 
Then there exist universal constants $C_1, C_2, C_3,C_4,c_2,c_3 > 0 $ such that 
 \begin{multline}
	 \opP \left(\sup_i \left|\frac{1}{n} \tbX_i^{\tp} \tbGi^{-1} \tbX_i - \frac{1}{n}\tbX_i^{\tp} \tbGimin^{-1} \tbX_i \right| \leq C_1 \frac{K_n^2H_n}{\sqrt{n}} \right)  \\
	 = 1-C_2n^2\exp\left(-c_2 H_n^2\right)-C_3n\exp\left(-c_3K_n^2\right) - \exp\left(-C_4n\left(1+o(1)\right)\right)=1-o(1) \label{eq: quadapprox1},
\end{multline}
\begin{multline}
	 \opP \left(\sup_i \left|\frac{1}{n} \tr\big(\tbGi^{-1}\big) - \frac{1}{n}\tr\big( \tbGimin^{-1} \big)  \right| \leq C_1\frac{K_n^2H_n}{\sqrt{n}}   \right) \\
 = 1-C_2n^2\exp\left(-c_2 H_n^2\right)-C_3n\exp\left(-c_3K_n^2\right)  - \exp\left(-C_4n\left(1+o(1)\right)\right)=1-o(1)  \label{eq: traceapprox5},
\end{multline}
where $K_n, H_n$ are diverging sequences as defined in \eqref{eq: HnKn}. 
\end{lemma}

\noindent This together with \eqref{eq: tbGiminapprox-11} yields that
\begin{multline}\label{lemma: tbGiminapprox}
	 \opP \left(\sup_i \left| \frac{1}{n} \tbX_i^{\tp} \tbGi^{-1} \tbX_i - \frac{1}{n}\tr(\tbGi^{-1}) \right|  > C_1 \frac{K_n^2 H_n}{\sqrt{n}}\right) \\
 \leq ~C_2n^2\exp\left(-c_2 H_n^2\right)+ C_3n\exp\left(-c_3K_n^2\right) +\exp\left(-C_4n\left(1+o(1)\right)\right) . 
\end{multline} 
The final ingredient is to establish that $\frac{1}{n}\tr\big(\tbGi^{-1}\big)$ and $\frac{1}{n} \tr\big(\tbG^{-1}\big)$ are uniformly close across $i$.
\begin{lemma}\label{lemma: tracediff}
	Let $\tbG$ and $\tbGi$ be as defined in (\ref{eq: defGtilde}) and (\ref{eq: defGi}), respectively.   
Then one has
\begin{equation}\label{eq: traceapprox6}
	\opP \left( \left|\tr\big( \tbGi^{-1} \big) -\tr\big( \tbG^{-1} \big)\right| \leq \frac{1}{\lambda_{\mathrm{lb}}} \right) \geq 1-  \exp\left(- \Omega( n) \right).
\end{equation}
\end{lemma}
This completes the proof. 
\end{proof}


\begin{proof}[Proof of Lemma \ref{lemma: quadtraceapprox}]
For two invertible matrices $\bA$ and $\bB$ of the same dimensions, the difference of their inverses can be written as $$\bA^{-1} - \bB^{-1} = \bA^{-1}(\bB - \bA) \bB^{-1}.$$ Applying this identity, we have 
\[\tbGi^{-1}-\tbGimin^{-1} = \tbGi^{-1} \left(\tbGimin-\tbGi \right) \tbGimin^{-1}.\]
From the definition of these matrices, it follows directly that 
\begin{equation} \label{eq:diff-G}
	\tbGimin-\tbGi  = \frac{1}{n} \sum_{j: j \neq i} \left( \rho''\big(\tbX_j^{\tp} \tbbetaimin\big)- \rho''\big(\tbX_j^{\tp}\tbbeta\big) \right) \tbX_j\tbX_j^{\tp} . 
\end{equation}

	As $\rho'''$ is bounded, by the mean-value theorem, it suffices to control the  differences $\bX_j^{\tp} \tbbetaimin - \tbX_j^{\tp} \tbbeta$ uniformly across all $j$. This is established in the following lemma, the proof of which is deferred to Appendix  \ref{sec:Proof-of-Lemma-fitdiff1}.
\begin{lemma}\label{lemma: fitdiff1}
	Let $\hbbeta$ be the full model MLE and $\hbbetaimin$ be the MLE when the $i{\text{th}}$ observation is dropped. Let  $q_i$ be as described in Lemma \ref{lemma: hb} and $K_n, H_n$ be as in \eqref{eq: HnKn}. Then there exist universal constants $C_1,C_2,C_3,C_4,c_2,c_3>0$ such that
\begin{multline}
\opP	\left(\sup_{j \neq i}\left| \bX_j^{\tp}\hbbetaimin - \bX_j^{\tp}\hbbeta \right| \leq C_1 \frac{K_n^2 H_n}{\sqrt{n}} \right)  \\
	\geq 1-C_2n\exp\left(-c_2
          H_n^2\right)-C_3\exp\left(-c_3K_n^2\right)	-
        \exp\left(-C_4n\left(1+o(1)\right)\right)=1-o(1), \label{eq:lem-fitdiff1-1}
\end{multline}
\begin{multline}
\opP	 \left(\sup_i |\bX_i^{\tp}\hbbeta - \prox_{q_i \rho}(\bX_i^{\tp}\hbbetaimin)| \leq C_1 \frac{K_n^2H_n}{\sqrt{n}}\right)  \\
	\geq 1- C_2n\exp\left(-c_2 H_n^2\right)-C_3\exp\left(-c_3K_n^2\right)  - \exp\left(-C_4n\left(1+o(1)\right)\right) =1-o(1) . \label{eq:lem-fitdiff1-2}
\end{multline} 
\end{lemma}
Invoking this lemma, we see that the spectral norm of (\ref{eq:diff-G}) is bounded above by some constant times $$ \frac{ K_n^2 H_n} { \sqrt{n} } \Big\| \sum_{j: j \neq i} \tbX_j \tbX_j^{\tp}/n \Big\|$$ with high probability as specified in \eqref{eq:lem-fitdiff1-1}. From Lemma \ref{lemma: singval}, the spectral norm here is bounded by some constant with probability at least $1-c_1 \exp(-c_2n)$. These observations together with \eqref{eq: Anprob} and the fact that on $\A_n$ the minimum eigenvalues of $\tbGi$ and $\tbGimin$ are bounded by $\lambda_{\mathrm{lb}}$ yield that 
\[
	\opP \left( \big\| \tbGi^{-1} -\tbGimin^{-1} \big\| \leq C_1 \frac{K_n^2 H_n}{\sqrt{n}} \right) 
	 \geq 1 - C_2n\exp\left(-c_2 H_n^2\right)-C_3\exp\left(-c_3K_n^2\right)- \exp\left(-C_4n\left(1+o(1)\right)\right) .
\]
This is true for any $i$. Hence, taking the union bound we obtain 
\begin{multline}
	 \opP \left(\sup_i \big\| \tbGi^{-1} -\tbGimin^{-1} \big\| \leq C_1 \frac{K_n^2 H_n}{\sqrt{n}} \right) \\
	\geq 1 - C_2n^2\exp\left(-c_2 H_n^2\right)-C_3n\exp\left(-c_3K_n^2\right)- \exp\left(-C_4n\left(1+o(1)\right)\right) . \label{eq: specnormdiff1}
\end{multline}  
In order to establish the first result, note that 
\[\sup_i\frac{1}{n} \left| \tbX_i^{\tp} \tbGi^{-1} \tbX_i - \tbX_i^{\tp} \tbGimin^{-1} \tbX_i \right| \leq \sup_i \frac{\|\tbX_i \|^2}{n} \sup_i\| \tbGi^{-1} -\tbGimin^{-1} \|. \]
To obtain the second result, note that
\[\sup_i \left|\frac{1}{n} \tr(\tbGi^{-1}) - \frac{1}{n}\tr(\tbGimin^{-1})  \right| \leq  \frac{p-1}{n} \sup_i\| \tbGi^{-1} -\tbGimin^{-1} \|.  \]
Therefore, combining \eqref{eq: specnormdiff1} and Lemma \ref{lemma:
  singval} gives the desired result.
\end{proof}

\begin{proof}[Proof of Lemma \ref{lemma: tracediff}] We restrict
  ourselves to the event $\A_n$ throughout. Recalling \eqref{eq:
    SMWapp}, one has
\begin{align*}
	\tr(\tbGi^{-1}) - \tr(\tbG^{-1}) & = \frac{\rho''(\tbX_i^{\tp}\tbbeta)}{n} \frac{\tbX_i^{\tp}\tbGi^{-2}  \tbX_i}{1+ \frac{\rho''(\tbX_i^{\tp}\tbbeta)}{n}\tbX_i^{\tp}\tbGi^{-1}\tbX_i}.
\end{align*}
In addition, on $\A_n$ we have 
\begin{align*}
	\frac{1}{\lambda_{\mathrm{lb}}}  \tbX_i^{\tp} \tbGi^{-1} \tbX_i - \tbX_i^{\tp} \tbGi^{-2} \tbX_i & = \frac{1}{\lambda_{\mathrm{lb}}} \tbX_i^{\tp}\tbGi^{-1}\left( \tbGi - \lambda_{\mathrm{lb}} \Id \right) \tbGi^{-1}\tbX_i \geq 0.
\end{align*}
 Combining these results and recognizing that $\rho''>0$, we  get
 \begin{equation} \label{eq: approxfinal2}
	  \left| \tr(\tbGi^{-1}) - \tr(\tbG^{-1}) \right|  \leq    \frac{\rho''(\tbX_i^{\tp}\tbbeta)}{n}\frac{\frac{1}{\lambda_{\mathrm{lb}}}\tbX_i^{\tp}\tbGi^{-1}\tbX_i}{1+ \frac{\rho''(\tbX_i^{\tp}\tbbeta)}{n}\tbX_i^{\tp}\tbGi^{-1}\tbX_i}  \leq  \frac{1}{\lambda_{\mathrm{lb}} }
  \end{equation}
	as claimed. 
  %
  \end{proof}

\section{Proof of Lemma \ref{lemma: approx}}
\label{sub:Proof-of-Lemmas-nablal-approx}


Again, we restrict ourselves to the event $\A_n$ on which $\tbG \succeq \lambda_{\mathrm{lb}}\bm{I}$.  Note that 
\[
	\tbX_i^{\tp} \tbG^{-1}\bw = \frac{1}{n} \tbX_i^{\tp} \tbG^{-1} \tbX^{\tp}\Dtbeta \Xcone . 
\]
Note that $\{\tbG, \tbX\}$ and $\Xcone$ are independent. Conditional
on $\tbX$, the left-hand side is Gaussian with mean zero and variance
$\frac{1}{n^2}\tbX_i^{\tp} \tbG^{-1} \tbX^{\tp}\Dtbeta^2 \tbX
\tbG^{-1} \tbX_i$. The variance is bounded above by
\begin{align}
	\sigma_{X}^2 :&=\frac{1}{n^2}\tbX_i^{\tp} \tbG^{-1} \tbX^{\tp}\Dtbeta^2 \tbX \tbG^{-1} \tbX_i   ~\leq~ \sup_i\big|\rho''(\tbX_i^{\tp}\tbbeta)\big| \cdot \frac{1}{n^2}\tbX_i^{\tp} \tbG^{-1} \tbX^{\tp}\Dtbeta \tbX \tbG^{-1} \tbX_i \nonumber\\
	&= \frac{1}{n} \sup_i \big| \rho''(\tbX_i^{\tp}\tbbeta)\big| \cdot \tbX_i^{\tp}\tbG^{-1} \tbX_i  \lesssim \frac{1 }{n} \|\tbX_i\|^2 
\end{align}
In turn, Lemma \ref{lemma: singval} asserts that $n^{-1} \|\tbX_i\|^2$
is bounded by a constant with high probability. As a result, applying
Gaussian concentration results \cite[Theorem 2.1.12]{tao2012topics}
gives
\[
   |\tbX_i ^{\tp} \tbG^{-1} \bw | \lesssim H_n
\]
with probability exceeding $1-C\exp\left(-cH_n^2\right)$, where $ C, c >0$ are universal constants. 

In addition, $\sup_i|X_{i1}| \lesssim H_n$ holds with probability
exceeding $1-C\exp\left(-cH_n^2\right)$. Putting the above results
together, applying the triangle inequality
$|X_{i1}-\tbX_i^{\tp} \tbG^{-1}\bw|\leq |X_{i1}|+ |\tbX_i^{\tp}
\tbG^{-1}\bw|$, and taking the union bound, we obtain
\[\opP(\sup_{1\leq i \leq n}|X_{i1} - \tbX_i^{\tp}\tbG^{-1}\bw| \lesssim H_n) \geq 1-Cn\exp\left(-cH_n^2\right) = 1-o(1). \]

\section{Proof of Lemma \ref{lemma: fitdiff1}}
\label{sec:Proof-of-Lemma-fitdiff1}

The goal of this section is to prove Lemma \ref{lemma: fitdiff1}, which relates the full-model MLE $\hbbeta$ and the MLE $\hbbetaimin$.
To this end, we establish the key lemma below.
\begin{lemma}\label{lemma: hb}
	Suppose $\hbbetaimin$ denote the MLE when the $i{\text{th}}$ observation is dropped. Further let $\bGimin$ be as in (\ref{eq: bGimin}), and define
	 $q_i$ and $\hbb$  as follows:  
\begin{align}
  q_i & =\frac{1}{n} \bX_i^{\tp}\bGimin^{-1} \bX_i  ; \nonumber \\
	\hbb & = \hbbetaimin  - \frac{1}{n} \bGimin^{-1} \bX_i \left(  \rho' \Big( \prox_{q_i \rho} \big(\bX_i^{\tp}\hbbetaimin\big) \Big) \right) .
	\label{eq: hbb}
\end{align}
Suppose $K_n, H_n$ are diverging sequences as in \eqref{eq: HnKn}.
Then there exist universal constants $C_1, C_2, C_3 > 0$ such that
\begin{equation}
	 \opP \left( \|\hbbeta - \hbb \| \leq C_1 \frac{K_n^2 H_n}{n} \right)  
	\geq 1-C_2n\exp(-c_2 H_n^2)-C_3\exp(-c_3K_n^2)- \exp(-C_4n(1+o(1))); \label{eq: hb}
\end{equation}
\begin{multline}
	\opP \left( \sup_{j \neq i} \big| \bX_j^{\tp} \hbbetaimin - \bX_j^{\tp} \hbb \big|  \leq C_1 \frac{K_n H_n}{\sqrt{n}} \right) \\
	 \geq 1-C_2n\exp\left(-c_2 H_n^2\right)-C_3\exp\left(-c_3K_n^2\right)- \exp\left(-C_4n\left(1+o(1)\right)\right). \label{eq: Xb-2}
\end{multline}
\end{lemma}
The proof ideas are inspired by the
leave-one-observation-out approach of \cite{el2015impact}. We however
emphasize once more that the adaptation of these ideas to our setup is
not straightforward and crucially hinges on Theorem \ref{thm:
  normbound-MLE}, Lemma \ref{lemma: eigenvalue} and properties of the
effective link function.

\begin{proof}[Proof of Lemma \ref{lemma: hb}]
	Invoking techniques similar to that for establishing Lemma \ref{lemma: eigenvalue}, it can be shown that  
\begin{equation}
	\label{eq:min-evalue-GG}
	 \frac{1}{n} \sum_{i=1}^n\rho''(\gamma_i^{*}) \bX_i \bX_i^{\tp} \succeq  \lambda_{\mathrm{lb}} \bm{I}
\end{equation}
with probability at least $1- \exp(\Omega(n))$, where $\gamma_i^{*}$
is between $\bX_i^{\tp} \hbb$ and $\bX_i^{\tp} \hbbeta$. Denote by
$\mathcal{B}_n$ the event where (\ref{eq:min-evalue-GG})
holds. Throughout this proof, we work on the event
$\mathcal{C}_n :=\A_n \cap \mathcal{B}_n$, which has probability
$1- \exp\left(-\Omega(n)\right)$. As in \eqref{eq:
  diffnabla} then,
\begin{equation}\label{eq: hbbetahbbdiff}
	\| \hbbeta - \hbb \| \leq \frac{1}{n\lambda_{\mathrm{lb}}} \big\| {\nabla \Lcal(\hbb)} \big\| . 
\end{equation}
Next, we simplify (\ref{eq: hbbetahbbdiff}). To this end, recall the defining relation of the proximal operator
\[
   b\rho'(\prox_{b\rho}(z)) + \prox_{b\rho}(z) = z,
\]
which together with the definitions of $\hbb$ and $q_i$ gives
\begin{equation}
\bX_i^{\tp} \hbb = \prox_{q_i \rho} \left( \bX_i^{\tp}\hbbetaimin \right) \label{eq: proxrel2}.
\end{equation} 
	Now, let $\Lcalimin$ denote the negative log-likelihood function when the $i{\text{th}}$ observation is dropped, and hence $\nabla \Lcalimin \big( \hbbetaimin \big) = \bm{0}$.
Expressing $\nabla \Lcal (\hbb)$ as $\nabla \Lcal (\hbb)-\nabla \Lcalimin \big( \hbbetaimin \big)$, applying the mean value theorem, and using the analysis similar to that in \cite[Proposition 3.4]{el2015impact}, we obtain
 \begin{equation} \label{eq: gradrel}
\frac{1}{n}\nabla \Lcal (\hbb) = \frac{1}{n} \sum_{j: j \neq i } \left[\rho''(\gamma_j^{*})-\rho''(\bX_j^{\tp}\hbbetaimin)\right]  \bX_j \bX_j^{\tp} \left(\hbb - \hbbetaimin\right),
\end{equation} 
where $\gamma_j^{*}$ is between $\bX_j^{\tp}\hbb$ and $\bX_j^{\tp}\hbbetaimin$.
Combining \eqref{eq: hbbetahbbdiff} and \eqref{eq: gradrel} leads to the upper bound
\begin{equation}\label{eq: onestepnormdiff}
\|\hbbeta   - \hbb\| \leq \frac{1}{\lambda_{\mathrm{lb}}} \left\|\frac{1}{n} \sum_{j: j \neq i}\bX_j \bX_j ^{\tp} \right\| \cdot \sup_{j \neq i} \Big|\rho''(\gamma_j^{*})-\rho''\big(\bX_j^{\tp}\hbbetaimin\big)\Big|  \cdot  \Bigg\|\frac{1}{n} \bGimin^{-1} \bX_i \Bigg\| \cdot \left|  \rho'\left(\prox_{q_i \rho}(\bX_i^{\tp}\hbbetaimin) \right) \right| .
\end{equation}

We need to control each term in the right-hand side. To start with,
the first term is bounded by a universal constant with probability
$1- \exp(-\Omega(n))$ (Lemma \ref{lemma: singval}).
For the second term, since $\gamma_{j}^{*}$ is between
$\bX_j^{\tp}\hbb$ and $\bX_j^{\tp}\hbbetaimin$ and
$\|\rho'''\|_\infty <\infty$, we get
\begin{align}
\sup_{j \neq i} \big|\rho''(\gamma_j^{*})-\rho''(\bX_j^{\tp}\hbbetaimin)\big| 
&\leq \|\rho''' \|_{\infty} \| \bX_j^{\tp}\hbb - \bX_j^{\tp}\hbbetaimin \|  
 \label{eq: ineqtwo-2} \\
&\leq \|\rho''' \|_{\infty} 
\left| \frac{1}{n} \bX_j^{\tp} \bGimin^{-1} \bX_i  \rho' \Big( \prox_{q_i \rho} \big(\bX_i^{\tp}\hbbetaimin\big) \Big)  \right|
 \\
&\leq \|\rho''' \|_{\infty} \frac{1}{n}\sup_{j \neq i} \left|\bX_j^{\tp}\bGimin^{-1}\bX_i \right| \cdot \left|\rho' \left( \prox_{q_i \rho}(\bX_i^{\tp} \hbbetaimin) \right) \right| .
\end{align}
Given that $\{\bX_j, \bGimin\}$ and $\bX_i$ are independent for all $j \neq i$, conditional on $\{\bX_j,\bGimin\}$ one has 
$$\bX_j^{\tp}\bGimin^{-1}\bX_i \sim \dnorm\left( 0, \bX_j^{\tp}\bGimin^{-2}\bX_j \right).$$ 
In addition, the variance satisfies 
\begin{equation}
\label{eq: var-xGx}
	|\bX_j^{\tp}\bGimin^{-2} \bX_j| \leq \frac{\|\bX_j \|^2}{\lambda_{\mathrm{lb}}^2 } \lesssim n
\end{equation}
with probability at least $1-\exp(-\Omega(n))$. Applying standard
Gaussian concentration results \cite[Theorem 2.1.12]{tao2012topics},
we obtain
\begin{align}
\label{eq: crossterm}
\opP\left(\frac{1}{\sqrt{p}} \left| \bX_j ^{\tp } \bGimin^{-1} \bX_i \right| \geq C_1 H_n  \right) 
& \leq  C_2\exp\left(- c_2 H_n^2 \right) +  \exp\left(- C_3 n\left(1+o(1)\right) \right) .
\end{align}
By the union bound 
\begin{equation}
\label{eq: ineqthree}
  \opP\left(\frac{1}{\sqrt{p}}\sup_{j \neq i} \big| \bX_j ^{\tp } \bGimin^{-1} \bX_i \big| \leq C_1 H_n  \right)\geq 1- nC_2\exp\left(-c_2H_n^2 \right) -  \exp\left(-C_3 n \left(1+o(1)\right) \right) .
\end{equation}
Consequently,
\begin{align}
\label{eq: ineqtwo}
\sup_{j \neq i} \big|\rho''(\gamma_j^{*})-\rho''(\bX_j^{\tp}\hbbetaimin)\big| 
&\lesssim \sup_{j \neq i} \| \bX_j^{\tp}\hbb - \bX_j^{\tp}\hbbetaimin \|  \lesssim \frac{1}{\sqrt{n}} H_n \left|\rho' \left( \prox_{q_i \rho}(\bX_i^{\tp} \hbbetaimin) \right) \right| .
\end{align}
In addition, 
the third term in the right-hand side of \eqref{eq: onestepnormdiff}
can be upper bounded as well since 
\begin{equation}\label{eq: ineqone}
	\frac{1}{n}\|\bGimin^{-1} \bX_i \| = \frac{1}{n} \sqrt{|\bX_i^{\tp}\bGimin^{-2} \bX_i| } \lesssim \frac{1}{\sqrt{n}}
\end{equation}
with high probability. 

It remains to bound $\left|\rho' \left( \prox_{q_i \rho}(\bX_i^{\tp} \hbbetaimin) \right) \right|$. To do this, we begin by considering $\rho'(\prox_{c \rho}(Z))$ for any constant $c>0$ (rather than a random variable $q_i$). 
Recall that for any constant $c > 0$ and any $Z \sim \dnorm(0,\sigma^2)$ with finite variance, the random variable
 $\rho'(\prox_{c \rho}(Z))$ is sub-Gaussian. Conditional on $\hbbetaimin $, one has $\bX_i^{\tp}\hbbetaimin \sim \dnorm \big( 0,\|\hbbetaimin \|^2 \big)$. This yields 
\begin{align}\label{eq: rhoprimeprox}
\opP \left(\rho' \left(\prox_{c \rho}(\bX_i^{\tp}\hbbetaimin) \right) \geq  C_1 K_n \right ) & \leq C_2\E\left[ \exp \left(-\frac{C_3^2 K_n^2}{ \|\hbbetaimin \|^2}  \right) \right] \nonumber \\
& \leq C_2 \exp\left(-C_3 K_n^2\right) + C_4\exp\left(-C_5n\right)
\end{align}
for some constants $C_1,C_2,C_3,C_4, C_5>0$ since$\|\hbbetaimin \|$ is bounded with high probability (see Theorem \ref{thm: normbound-MLE}).

Note that $\frac{\partial \prox_{b\rho}(z)}{\partial b}\leq 0$ by \cite[Proposition 6.3]{donoho2013high}. Hence, in order to move over from the above concentration result established for  a fixed constant $c$ to the random variables $q_i$, it suffices to establish a uniform lower bound for $q_i$ with high probability. Observe that for each $i$, 
\[
	q_i \geq  \frac{\|\bX_i \|^{2}}{n} \frac{1}{\big\| \bGimin \big\|} \geq C^* 
\]
with probability $1-\exp(-\Omega(n))$, where $C^*$ is some universal constant. On this event, one has
%
$$ \rho' \left( \prox_{q_i \rho} \left( \bX_i^{\tp} \hbbetaimin\right) \right) \leq \rho' \left( \prox_{C^{*} \rho} \left(\bX_i^{\tp} \hbbetaimin \right) \right) .$$ 
This taken collectively with (\ref{eq: rhoprimeprox}) yields
\begin{align}
  \opP \left(\rho'(\prox_{q_i \rho}(\bX_i^{\tp}\hbbetaimin)) \leq C_1 K_n  \right ) & 
  \geq ~ \opP \left(\rho'(\prox_{C^* \rho}(\bX_i^{\tp}\hbbetaimin)) \leq C_1 K_n  \right ) \\
  & \geq~ 1-C_2 \exp\left(-C_3 K_n^2\right) - C_4 \exp\left(-C_5 n\right) .
\label{eq: rhoprimeproxqi}
 \end{align}
 %
This controls the last term.
	
To summarize, if $\{K_n\}$ and $\{H_n\}$ are diverging sequences satisfying the assumptions in \eqref{eq: HnKn}, combining \eqref{eq: onestepnormdiff} and the bounds for each term in the right-hand side 
finally gives \eqref{eq: hb}. On the other hand, combining \eqref{eq:
  ineqthree} and \eqref{eq: rhoprimeproxqi} yields \eqref{eq: Xb-2}.
\end{proof}

With the help of Lemma \ref{lemma: hb} we are ready to prove Lemma
\ref{lemma: fitdiff1}.  Indeed, observe that
\[
	\big|\bX_j^{\tp}(\hbbetaimin - \hbbeta) \big| \leq \big| \bX_j^{\tp} (\hbb - \hbbeta) \big| + \big| \bX_j^{\tp} (\hbbetaimin - \hbb) \big| , 
\]
and hence by combining Lemma \ref{lemma: singval} and Lemma \ref{lemma: hb},
we establish the first claim (\ref{eq:lem-fitdiff1-1}). 
The second claim \eqref{eq:lem-fitdiff1-2} follows directly from Lemmas \ref{lemma: singval}, \ref{lemma: hb} and \eqref{eq: proxrel2}.

\section{Proof of Theorem \ref{thm:mainthm}(b)} \label{sec: appendix8}
This section proves that the random sequence
$\talpha = \tr\big(\tbG^{-1}\big)/n$ converges in probability to the
constant $\bs$ defined by the system of equations \eqref{eq:
  sysofeq-tau} and \eqref{eq: sysofeq-b}.  To begin with, we claim
that $\talpha$ is close to a set of auxiliary random
variables $\{\tq_i\}$ defined below.

\begin{lemma} \label{lemma: qialphaclose}
Define $\tq_i$  to be 
\[
	\tq_i =\frac{1}{n} \tbX_i^{\tp}\tbGimin^{-1} \tbX_i, 
\]
	where $\tbGimin$ is defined in  (\ref{eq: tbGimin}). 
	Then there exist universal constants $C_1, C_2, C_3, C_4, c_2,c_3 >0$ such that
	\begin{multline*}
		 \opP \left(\sup_i |\tq_i - \talpha| \leq  C_1 \frac{K_n^2 H_n}{\sqrt{n}} \right)  \\
		 \geq 1-C_2n^2\exp\left(c_2 H_n^2\right)-C_3n\exp\left(-c_3 K_n^2\right)-\exp\left(-C_4n\left(1+o(1)\right)\right)=1-o(1), 
	\end{multline*}
where $K_n, H_n$ are as in \eqref{eq: HnKn}.
\end{lemma}
\begin{proof}
This result follows directly from Proposition \ref{prop: etai} and equation \eqref{eq: quadapprox1}. 

\end{proof}
A consequence is that $\prox_{\tq_{i} \rho}\left( \bX_i^{\tp} \hbbetaimin \right)$ becomes close to $\prox_{\talpha \rho} \left(\bX_i ^ {\tp} \hbbetaimin \right)$.
\begin{lemma}\label{lemma: proxclose}
Let $\tq_i$ and $\talpha$ be as defined earlier. Then one has 
\begin{multline}
	 \opP \left( \sup_i\left| \prox_{\tq_{i} \rho}\left( \bX_i^{\tp} \hbbetaimin \right) - \prox_{\talpha \rho} \left(\bX_i ^ {\tp} \hbbetaimin \right) \right|  \leq C_1 \frac{K_n^3 H_n}{\sqrt{n}}\right) 
	 \\
	\geq 1- C_2n^2\exp\left(-c_2 H_n^2\right)-C_3n\exp\left(-c_3 K_n^2\right)  -\exp\left(-C_4n\left(1+o(1)\right)\right) =1-o(1),\label{eq: proxclose}
\end{multline}
where $K_n, H_n$ are as in \eqref{eq: HnKn}. 
\end{lemma}

The key idea behind studying
$\prox_{\talpha \rho} \left(\bX_i ^ {\tp} \hbbetaimin \right)$ is that
it is connected to a random function $\delta_n(\cdot)$ defined below,
which happens to be closely related to the equation (\ref{eq:
  sysofeq-b}). In fact, we will show that $\delta_n(\talpha)$
converges in probability to 0; the proof relies on the
connection between
$\prox_{\talpha \rho} \left(\bX_i ^ {\tp} \hbbetaimin \right)$ and the
auxiliary quantity
$\prox_{\tq_{i} \rho}\left( \bX_i^{\tp} \hbbetaimin \right)$. The
formal results is this:
\begin{proposition}
\label{prop:delta-0}
For any index $i$, let $\hbbetaimin$ be the MLE obtained on dropping the $i{\text{th}}$ observation. Define $\delta_n(x)$ to be the random function 
\begin{equation}\label{eq: deltanx}
\delta_n(x) := \frac{p}{n} - 1 + \frac{1}{n} \sum_{i=1}^n \frac{1}{1+x \rho''\left(\prox_{x \rho}\left(\bX_i^{\tp} \hbbetaimin\right)\right)} .
\end{equation}
Then one has
 $ \delta_n(\talpha) ~\convP~ 0$.  
\end{proposition}


Furthermore, the random function $\delta_n(x)$ converges to a deterministic function $\Delta(x)$ defined by
\begin{equation} 
	\label{eq: Deltax}
	\Delta(x) = \kappa - 1 + \E_{Z} \left[ \frac{1}{1+x \rho''(\prox_{x \rho}(\taus Z))} \right ],
\end{equation}
where $Z \sim \dnorm(0,1)$, and $\taus$ is such that $(\taus,\bs)$ is the unique solution to \eqref{eq: sysofeq-tau} and \eqref{eq: sysofeq-b}. 
\begin{proposition}
\label{prop: Deltax}
	With $\Delta(x) $ as in (\ref{eq: Deltax}),  
	$\Delta(\talpha)  \convP 0.$
\end{proposition}

In fact, one can easily verify that
\begin{equation}
   \Delta(x) = \kappa -  \mathbb{E}\big[\Psi'\left(\taus Z; \hspace{0.2em} x\right)\big],
\end{equation}
and hence by Lemma \ref{lemma: monotonicity}, the solution to $\Delta(x)=0$ is exactly
$\bs$. As a result, putting the above claims together, we show that $\tilde{\alpha}$ converges in probability to $\bs$.

It remains to formally prove the preceding lemmas and propositions, which is the goal of the rest of this section.

\begin{proof}[Proof of Lemma \ref{lemma: proxclose}]
	By \cite[Proposition 6.3]{donoho2013high}, one has
\[
\frac{\partial\prox_{b\rho}(z)}{\partial b}=-\left.\frac{\rho'(x)}{1+b\rho''(x)}\right|_{x=\prox_{b\rho}(z)} ,
\]
which yields 
\begin{align}
 	& \sup_i\left| \prox_{\tq_i \rho}\left( \bX_i^{\tp} \hbbetaimin \right) - \prox_{\talpha \rho}\left( \bX_i^{\tp} \hbbetaimin \right) \right| \nonumber \\
	& \qquad  = \sup_i
\left[  \left| \left. \frac{\rho'(x)}{1+q_{\talpha,i}\rho''(x)}   \right|_{x=\prox_{q_{\talpha,i}\rho}\big(\bX_i^{\tp} \hbbetaimin\big)}  \right| \cdot |\tq_i - \talpha| \right] \nonumber\\
	& \qquad \leq \sup_i \left|  \rho'\left(\prox_{q_{\talpha,i}}(\bX_i^{\tp}\hbbetaimin) \right)  \right|  \cdot \sup_i |\tq_i - \talpha |, \label{eq:UB3}
\end{align}
	where $q_{\talpha,i}$ is between $\tq_{i}$ and $\talpha$. Here, the last inequality holds since $q_{\talpha,i},  \rho''\geq 0 $.  
	
	In addition, just as in the proof of Lemma \ref{lemma: hb},
        one can show that $q_{i}$ is bounded below by some constant
        $C^{*}>0$ with probability $1- \exp(- \Omega( n) )$.  Since
        $q_{\talpha,i} \geq \min\{\tq_i,\talpha\} $, on the event
        $\sup_{i}|\tq_i - \talpha| \leq C_1 K_n^2 H_n /\sqrt{n}$,
        which happens with high probability (Lemma \ref{lemma: qialphaclose}),
        $q_{\talpha,i} \geq C_{\alpha}$ for some universal constant
        $C_\alpha >0$. Hence, by an argument similar to that 
        establishing \eqref{eq: rhoprimeproxqi}, we have
\begin{multline*}
	 \opP \left(\sup_i \left|\rho'\left( \prox_{q_{\talpha,i}} \left(\bX_i^{\tp}\hbbetaimin \right) \right) \right| \geq C_1 K_n \right)  \\
	\leq  C_2n^2 \exp\left(-c_2 H_n^2\right)+C_3n\exp\left(- c_3 K_n^2\right) +\exp\left(-C_4n\left(1+o(1)\right)\right).
\end{multline*}
This together with (\ref{eq:UB3}) and Lemma \ref{lemma: qialphaclose} concludes the proof. 
\end{proof}

\begin{proof}[Proof of Proposition \ref{prop:delta-0}]
To begin with, recall from \eqref{eq: traceapprox2} and \eqref{eq: triminh} that on $\A_n$, 
\begin{equation}
	\label{eq:p-n-2}
	\frac{p-1}{n} = \sum_{i=1}^n \frac{\frac{\rho''(\tbX_i^{\tp}\tbbeta)}{n}\tbX_i^{\tp}\tbGi^{-1}\tbX_i}{1+\frac{\rho''(\tbX_i^{\tp}\tbbeta)}{n}\tbX_i^{\tp}\tbGi^{-1}\tbX_i} = 1- \frac{1}{n} \sum_{i=1}^n \frac{1}{1+\frac{\rho''(\tbX_i^{\tp}\tbbeta)}{n}\tbX_i^{\tp}\tbGi^{-1}\tbX_i} .
\end{equation}
Using the fact that $\big| \frac{1}{1+x} - \frac{1}{1+y} \big| \leq |x-y|$ for $x, y \geq  0$, we obtain  
\begin{align*}
 &\left|\frac{1}{n} \sum_{i=1}^n \frac{1}{1+\frac{\rho''(\tbX_i^{\tp}\tbbeta)}{n}\tbX_i^{\tp}\tbGi^{-1}\tbX_i}  - \frac{1}{n} \sum_{i=1}^n \frac{1}{1+\rho''(\tbX_i^{\tp}\tbbeta) \talpha}   \right|  \\
 & \quad  \leq \frac{1}{n} \sum_{i=1}^n  \rho''(\tbX_i^{\tp}\tbbeta) \left| \frac{1}{n}\tbX_i^{\tp}\tbGi^{-1}\tbX_i - \talpha \right| 
	~\leq~ \|\rho'' \|_{\infty} \sup_i \left| \frac{1}{n}\tbX_i^{\tp}\tbGi^{-1}\tbX_i - \talpha \right|  \\
 & \quad = \|\rho'' \|_{\infty} \sup_i |\eta_i| 
	\leq C_1 \frac{K_n^2 H_n}{\sqrt{n}}, 
  \end{align*}
  with high probability (Proposition \ref{prop: etai}). This combined
  with (\ref{eq:p-n-2}) yields
\begin{multline*} 
	\opP\left( \left|\frac{p-1}{n} - 1 +\frac{1}{n} \sum_{i=1}^n \frac{1}{1+ \rho''(\tbX_i^{\tp} \tbbeta) \talpha}  \right| 
	\geq C_1 \frac{K_n^2 H_n}{\sqrt{n}} \right)  \\
	 \leq  C_2n^2 \exp\left(-c_2 H_n^2\right)+C_3n\exp\left(- c_3 K_n^2\right)
	  +\exp\left(-C_4n\left(1+o(1)\right)\right) .
\end{multline*}

The above bound concerns
$\frac{1}{n} \sum_{i=1}^n \frac{1}{1+ \rho''(\tbX_i^{\tp} \tbbeta)
  \talpha} $, and it remains to relate it to
$\frac{1}{n} \sum_{i=1}^n \frac{1}{1+ \rho''\left( \prox_{\talpha
      \rho} \left( \tbX_i^{\tp} \tbbeta \right) \right) \talpha} $.
To this end, we first get from the uniform boundedness of $\rho'''$
and Lemma \ref{lemma: fitdiff1} that
\begin{multline}
	 \opP \left(\sup_i \left| \rho''(\tbX_i^{\tp}\tbbeta) - \rho''\left( \prox_{\tq_i \rho}(\tbX_i^{\tp}\tbbetaimin) \right) \right|  \geq C_1 \frac{K_n^2 H_n}{\sqrt{n}} \right) \\
	\leq  C_2n \exp(-c_2 H_n^2)+C_3\exp(- c_3 K_n^2)
 +\exp(-C_4n(1+o(1))).\label{eq: proxapprox1}
\end{multline}
 Note that
\begin{align*}
	& \left|\frac{1}{n}\sum_{i=1}^n \frac{1}{1+ \rho''(\tbX_i^{\tp} \tbbeta) \talpha}  -\frac{1}{n}\sum_{i=1}^n\frac{1}{1+ \rho''(\prox_{\talpha \rho}( \tbX_i ^{\tp} \tbbetaimin)) \talpha} \right| \\
	& \quad\leq   |\talpha| \sup_i \left| \rho''(\tbX_i^{\tp}\tbbeta) - \rho''\left( \prox_{\talpha \rho}( \tbX_i ^{\tp} \tbbetaimin) \right) \right| \\
	& \quad \leq |\talpha| \sup_i \left\{ \left| \rho''\left(\tbX_i^{\tp}\tbbeta \right) -  \rho''\left(\prox_{\tq_i \rho}( \tbX_i ^{\tp} \tbbetaimin) \right) \right|
	+ \left| \rho'' \left( \prox_{\tq_i \rho}( \tbX_i ^{\tp} \tbbetaimin) \right) - \rho'' \left(\prox_{\talpha \rho}( \tbX_i ^{\tp} \tbbetaimin) \right) \right| \right\}.
\end{align*}
By the bound \eqref{eq: proxapprox1}, an application of 
	Lemma \ref{lemma: proxclose}, and the fact that  $\talpha \leq p/(n\lambda_{\mathrm{lb}})$ (on $\A_n$), we obtain
\begin{multline*}
	 \opP \left( \left| \frac{p}{n} -1 + \frac{1}{n} \sum_{i=1}^n  \frac{1}{1+ \rho''
	\big( \prox_{\talpha \rho}( \bX_i ^{\tp} \hbbetaimin) \big) \talpha }  \right| \geq C_1  \frac{K_n^3H_n}{\sqrt{n}} \right)  \\
 \leq   C_2n^2 \exp\big(-c_2 H_n^2\big)+C_3n\exp\big(- c_3 K_n^2\big) 
	 +\exp\big(-C_4n(1+o(1))\big).
\end{multline*} 
This establishes that $\delta_n(\talpha) \convP 0$. 
\end{proof}

\begin{proof}[Proof of Proposition \ref{prop: Deltax}]
  Note that since $0 < \alpha \leq p/(n\lambda_{\mathrm{lb}}):=B$ on
  $\A_n$, it suffices to show that
\[
	\sup_{x \in [0,B]} |\delta_n(x)-\Delta(x)| ~\convP~ 0.
\]
We do this by following three steps. Below, $M>0$ is some sufficiently
large constant.
\begin{enumerate}
\item First we truncate the random function $\delta_n(x)$ and define
\[
	\tdelta_n(x) = \frac{p}{n} - 1+\sum_{i=1}^n \frac{1}{1+ x \rho''\left( \prox_{x\rho}\left( \bX_i^{\tp}\hbbetaimin\bone_{\{\|\hbbetaimin \| \leq M\}} \right) \right)} .
\]
The first step is to show that $\sup_{x \in [0,B]} \left| \tdelta_n(x) - \delta_n(x) \right| \convP 0.$ We stress that this truncation does not arise in \cite{el2015impact}, and we keep track of the truncation throughout the rest of the proof.
\item Show that $\sup_{x \in [0,B]} \left| \tdelta_n(x) - \E \big[\tdelta_n(x)\big] \right| \convP 0$. 
\item Show that $\sup_{x \in [0,B]} \left| \E \big[ \tdelta_n(x) \big] - \Delta(x) \right| \convP 0$. 
\end{enumerate}

To argue about the first step, observe that
$|\frac{1}{1+y} - \frac{1}{1+z}| \leq |y-z|$ for any $y,z>0$ and that
$\big| \frac{\partial\prox_{c \rho}(x)}{\partial x} \big| \leq 1$
\cite[Proposition 6.3]{donoho2013high}. Then 
	\[ 
		|\delta_n(x) - \tdelta_n(x)| \leq |x| \cdot \|\rho''' \|_{\infty} \cdot \sup_i \left|\bX_i^{\tp}\hbbetaimin -\bX_i^{\tp}\hbbetaimin  \bone_{\{\|\hbbetaimin \|\leq M\}} \right| .
	\]
	For a sufficiently large constant $M>0$,
        $\opP(\|\hbbetaimin \|\geq M) \leq \exp(-\Omega(n))$ by Theorem
        \ref{thm: normbound-MLE}.  Hence, for any $\epsilon > 0$,
\begin{align}
	\opP \left(\sup_{x \in [0,B]} |\delta_n(x) - \tdelta_n(x)| \geq \epsilon \right) &\leq \opP\left(\sup_i \left|\bX_i^{\tp}\hbbetaimin \bone_{\{\|\hbbetaimin \| \geq M\}} \right| \geq \frac{\epsilon}{B \|\rho''' \|_{\infty}} \right)  \nonumber \\
	& \leq \sum_{i=1}^{n}\opP\left( \|\hbbetaimin \| \geq M\right) ~=~ o(1),
\end{align}
establishing Step 1. 

To argue about Step 2, note that for any $x$ and $z$,
\begin{align*} 
	\left| \tdelta_n(x) - \tdelta_n(z) \right| & \leq  \frac{1}{n}\sum_{i=1}^{n}\left|\frac{1}{1+x \rho''\left(\prox_{x \rho}\left(\bX_i^{\tp}\hbbetaimin\bone_{\{\|\hbbetaimin\| \leq M\}}\right)\right)}-\frac{1}{1+z \rho''\left(\prox_{z \rho}\left(\bX_i^{\tp}\hbbetaimin\bone_{\{\|\hbbetaimin\| \leq M\}}\right)\right)}   \right| \\ 
	& \leq  \frac{1}{n}\sum_{i=1}^{n} \left|x \rho''\left(\prox_{x \rho}\left(\bX_i^{\tp}\hbbetaimin\bone_{\{\|\hbbetaimin\| \leq M\}}\right)\right) - z\rho''\left(\prox_{z\rho}\left(\bX_i^{\tp}\hbbetaimin\bone_{\{\|\hbbetaimin\| \leq M\}}\right)\right) \right| \\
	& \leq \frac{1}{n}\sum_{i=1}^{n} \left(\|\rho'' \|_{\infty} |x-z| + |z| \cdot \|\rho''' \|_{\infty} \left| \left. \frac{\rho'(x)}{1+\tilde{z} \rho''(x)} \right|_{x = \prox_{\tilde{z}\rho} \big( \bX_i^{\tp}\hbbetaimin \bone_{\{\|\hbbetaimin\| \leq M\}} \big)} \right|  |x-z| \right) \\
	& \leq |x-z| \left(\|\rho'' \|_{\infty } + |z| \cdot \|\rho''' \|_{\infty} \frac{1}{n} \sum_{i=1}^n \left| \rho'\left(\prox_{\tilde{z}\rho}\left(\bX_i^{\tp}\hbbetaimin\bone_{\{\|\hbbetaimin\| \leq M\}}\right)\right) \right| \right) ,
\end{align*}
where $\tilde{z} \in (x,z).$ 
	Setting $$Y_n  := \|\rho'' \|_{\infty } + B \|\rho''' \|_{\infty} \frac{1}{n} \sum\nolimits_{i=1}^n \left|\rho'\left( \prox_{\tilde{z}\rho}\left(\bX_i^{\tp}\hbbetaimin\bone_{\{\|\hbbetaimin\| \leq M\}}\right) \right) \right| ,$$ then for any $\epsilon, \eta >0$ we have  
\begin{align}\label{eq: deltan}
	\opP\left(\sup_{x, z \in (0,B], |x-z| \leq \eta} |\tdelta_n(x) - \tdelta_n(z)| \geq \epsilon  \right) ~\leq~ 
	\opP\left(Y_n \geq \frac{\epsilon}{\eta} \right)  
	  ~\leq~  \frac{\eta}{\epsilon} \E [Y_n]  ~\leq~ \eta C_1(\epsilon),
\end{align} 
where $C_1(\epsilon)$ is some function independent of $n$. The
inequality \eqref{eq: deltan} is an analogue of \cite[Lemma
3.24]{el2015impact}. We remark that the truncation is particularly
important here in guaranteeing that $\E[Y_n] < \infty$.



	Set $$G_n(x) := \E\big[\tdelta_n(x)\big],$$  and observe that 
\[
	\big|G_n(x) - G_n(z)\big| \leq |x-z| \left(\|\rho'' \|_{\infty } + |z| \|\rho''' \|_{\infty} \E \left[ \left| \rho' \left(\prox_{\tilde{z}\rho}\left(\bX_i^{\tp}\hbbetaimin\bone_{ \{\|\hbbetaimin \| \leq M \}}\right) \right) \right| \right] \right) .
\]
A similarly inequality applies to $\Delta(x)$ in which 
$\bX_i^{\tp}\hbbetaimin \bone_{\{\|\hbbetaimin \| \leq M\}}$ is replaced
by $\taus Z$. In either case, 
\begin{equation}
	\sup_{x, z \in (0,B], |x-z| \leq \eta} |G_n(x) - G_n(z)| \leq C_2 \eta \qquad \text{and} \qquad \sup_{x, z \in (0,B], |x-z| \leq \eta} |\Delta(x) - \Delta(z)| \leq C_3\eta 
\end{equation}
for any $\eta > 0$.

%

For any $\epsilon' > 0$, set $K = \max\{C_1(\epsilon'), C_2 \}
$. Next, divide $[0,B]$ into finitely many segments
$$[0,x_1),~[x_1,x_2), ~\hdots, ~[x_{K-1},x_K:=B]$$ such that the
length of each segment is $\eta/K$ for any $\eta > 0$. Then for every
$x\in[0,B]$, there exists $l$ such that $|x - x_l| \leq \eta/K$.
As a result, for any $x \in [0,B]$,
\begin{align*}
\sup_{x \in (0,B]} \left|\tdelta_n(x) - G_n(x) \right| & ~\leq~ \eta+ \sup_{x,x_l \in (0,B], |x - x_l|\leq \eta/K} \left|\tdelta_n(x) - \tdelta_n(x_l)\right| + \sup_{1 \leq l \leq k} \left|\tdelta_n(x_l) - G_n(x_l) \right| .
\end{align*}
Now fix $\delta > 0, \epsilon > 0$.  Applying the above inequality gives 
 \begin{align}
 \opP \left(\sup_{x \in (0,B]}|\tdelta_n(x) - G_n(x)| \geq \delta \right) & \leq \opP \left( \sup_{x,x_l \in (0,B], |x - x_l|\leq \eta/K}|\tdelta_n(x) - \tdelta_n(x_l)|  \geq \frac{\delta - \eta}{2}\right) \nonumber\\ 
 & \qquad\qquad + \opP \left( \sup_{1 \leq l \leq k}|\tdelta_n(x_l) - G_n(x_l)| \geq \frac{\delta - \eta }{2} \right) \nonumber\\
	 & \leq \frac{\eta}{K} C_1\left(\frac{\delta-\eta}{2} \right) + \opP \left( \sup_{1 \leq l \leq k}|\tdelta_n(x_l) - G_n(x_l)| \geq \frac{\delta - \eta }{2} \right). \label{eq:delta-UB-1}
\end{align}
Choose
$\eta < \min\{\epsilon/2,\delta \}, K = \max\{C_1 (
\frac{\delta-\eta}{2} ) ,C_2\}$. Then the first term in the right-hand
side is at most $\epsilon/2$. Furthermore, suppose one can establish
that for any fixed $x$,
\begin{equation}
	|\tdelta_n(x)- G_n(x)| ~\convP ~0.  \label{eq: toshow}
\end{equation}
Since the second term in the right-hand side is a supremum over
finitely many points, there exists an integer $N$ such that for all
$ n \geq N$, the second term is less than or equal to
$\epsilon/2$. Hence, for all $n \geq N$, right-hand side is at most
$\epsilon$, which proves Step $2$.
Thus, it remains to prove (\ref{eq: toshow}) for any fixed $x\geq
0$. We do this after the analysis of Step 3.

We argue about Step 3 in a similar fashion and letting
$K = \max \{C_2, C_3 \}$, divide $(0,B]$ into segments of length
$\eta/K$. For every $x \in (0,B]$ there exists $x_l$ such that
$|x-x_l |\leq \eta/K$. Then
\begin{align*}
 |G_n(x) -\Delta(x)| & \leq~ |G_n(x) - G_n(x_l)| + | G_n(x_l) - \Delta(x_l) | + |\Delta(x_l) - \Delta(x)| \\
 & \leq~ 2\eta + |G_n(x_l) - \Delta(x_l)| ,
\end{align*}
\begin{align*}
 \implies \qquad \sup_{[0,B]} | G_n(x) - \Delta(x)| & \leq 2\eta + \sup_{1 \leq l \leq k} |G_n(x_l) - \Delta(x_l)|. 
 \end{align*}
 Hence, it suffices to show that for any fixed $x$,
 $|G_n(x) - \Delta(x)| \rightarrow 0.$ To this end, observe that for
 any fixed $x$, 
 \begin{multline*}
   |G_n(x) - \Delta(x)| \\
   \leq \left|\frac{p}{n} - \kappa \right| + \left|\E_{\bX} \left[ \frac{1}{1+x \rho''\left(\prox_{x \rho}\left( \bX_{1}^{\tp}\hbbetaonemin \bone_{ \{  \| \hbbetaonemin\| \leq M\}  }\right)\right)} \right]  - \E_{Z} \left[ \frac{1}{1+x \rho''\left(\prox_{x \rho}(\taus Z)\right)} \right]  \right|. \end{multline*}
Additionally, 
$$\bX_1^{\tp } \hbbetaonemin \bone_{\{\| \hbbetaonemin\| \leq M \}} = \|\hbbetaonemin \| \bone_{\{\|\hbbetaonemin \| \leq M \}} \tilde{Z},$$ 
where $\tilde{Z} \sim \dnorm(0,1)$. Since
$\|\hbbetaonemin \|\bone_{ \{\|\hbbetaonemin \| \leq M \}} \convP
\taus$, by Slutsky's theorem,
$\bX_1^{\tp } \hbbetaonemin \bone_{ \{\| \hbbetaonemin\| \leq M \} }$
converges weakly to $\taus \tilde{Z}$. Since $t \mapsto 
1/(1+x\rho"(\prox_{x\rho}(t)))$ is bounded, one directly
gets $$G_n(x) -\Delta(x) ~\convP~ 0$$ for every $x$.


Finally, we establish \eqref{eq: toshow}. To this end, we will prove instead that for any given $x$, $$\delta_n(x) -G_n(x) ~\convLtwo~ 0 .$$ Define 
\[
	M_i := \bX_i^{\tp}\hbbetaimin\bone_{ \{\|\hbbetaimin \| \leq M \} } \qquad \text{and} \qquad f(M_i) := \frac{1}{1+x\rho''\left(\prox_{x\rho}(M_i)\right)} - \E\left[ \frac{1}{1+x\rho''\left(\prox_{x\rho}(M_i)\right)}\right]. 
\]
Then $\delta_n(x)-G_n(x)= \frac{1}{n} \sum_{i=1}^n f(M_i)$. Hence, for any $x \in [0,B]$,  
\begin{align*}
	\mathsf{Var}[ \delta_n(x) ]  & = \frac{1}{n^2} \sum_{i=1}^n \E \big[ f^2(M_i) \big] + \frac{1}{n^2} \sum_{i \neq j} \E\big[ f(M_i)f(M_j) \big]\\
	& = \frac{\E\big[ f^2(M_1) \big]}{n} + \frac{n(n-1)}{n^2} \E\big[ f(M_1)f(M_2) \big].
\end{align*} 
The first term in the right-hand side is at most $1/n$. Hence, it
suffices to show that $\E\big[ f(M_1)f(M_2) \big] \rightarrow 0.$ Let
$\hbbetaonetwomin$ be the MLE when the $1{\text{st}}$ and
$2{\text{nd}}$ observations are dropped. Define
\begin{align*}
	&\bGonetwomin  :=  \frac{1}{n} \sum_{j \neq 1,2} \rho''\left(\bX_j^{\tp}\hbbetaonetwomin\right)\bX_j \bX_j^{\tp}, \qquad  q_2 := \frac{1}{n} \bX_2^{\tp}\bGonetwomin^{-1} \bX_2, \\
	&\qquad \hbbonemin := \hbbetaonetwomin  + \frac{1}{n} \bGonetwomin^{-1} \bX_2 \left(- \rho' \left(\prox_{q_2 \rho}\left(\bX_2^{\tp}\hbbetaonetwomin\right)\right) \right). 
 \end{align*}
By an application of Lemma \ref{lemma: hb},
\begin{equation}\label{eq: beta1min}
\opP \left( \big\| \hbbetaonemin - \hbbonemin \big\| \geq  C_1 \frac{K_n^2 H_n}{n}  \right) \leq C_2n\exp\left(-c_2 H_n^2\right)-C_3\exp\left(-c_3K_n^2\right)- \exp\left(-C_4n(1+o(1))\right).
\end{equation}
Also, by the triangle inequality,
\begin{align*}
\left|\bX_1^{\tp}(\hbbetaonemin - \hbbetaonetwomin)\right| & \leq \left|\bX_1^{\tp}(\hbbetaonemin - \hbbonemin)\right| + \frac{1}{n} \left| \bX_1^{\tp} \bGonetwomin^{-1} \bX_2\left(-\rho'\left(\prox_{q_2 \rho}\left(\bX_2^{\tp}\hbbetaonetwomin\right)\right)\right) \right|. 
\end{align*}
Invoking Lemma \ref{lemma: singval}, \eqref{eq: beta1min}, and an
argument similar to that leading to \eqref{eq: rhoprimeproxqi} and
\eqref{eq: crossterm}, we obtain 
\begin{multline*}
    \opP \left(\left|\bX_1^{\tp} \left(\hbbetaonemin - \hbbetaonetwomin\right)\right|  \geq C_1\frac{K_n^2 H_n}{\sqrt{n}}  \right) \\
 \leq C_2n\exp\left(-c_2 H_n^2\right)-C_3\exp\left(-c_3K_n^2\right)- \exp\left(-C_4n\left(1+o(1)\right)\right).
\end{multline*}

The event $\{\|\hbbetaonemin \| \leq M \}\cap \{ \|\hbbetaonetwomin \| \leq M \} $ occurs with probability at least $1-C\exp(-cn)$. Hence, one obtains  
\begin{multline}
	\opP \left(\left|\bX_1^{\tp} \left(\hbbetaonemin \bone_{\|\hbbetaonemin \| \leq M} - \hbbetaonetwomin  \bone_{\|\hbbetaonetwomin \| \leq M}\right)\right| \leq C_1\frac{K_n^2 H_n}{\sqrt{n}}  \right) \\
	 \qquad \geq C_2n\exp\left(-c_2 H_n^2\right)-C_3\exp\left(-c_3K_n^2\right)  - \exp\left(-C_4n\left(1+o(1)\right)\right).\label{eq: betaoneapprox}
\end{multline}

A similar statement continues to hold with $\bX_1$ replaced by $\bX_2$ and $\hbbetaonemin$ replaced by $\hbbetatwomin$. Some simple computation yields that $\|f'\|_{\infty}$ is bounded by some constant times $|x|$. By the mean value theorem and the fact that $\|f \|_{\infty} \leq 1,$
\begin{align*}
	& \left| f(M_1) f(M_2)- f\left( \bX_1^{\tp}\hbbetaonetwomin \bone_{\{\|\hbbetaonetwomin \| \leq M\}} \right) f \left( \bX_2^{\tp}\hbbetaonetwomin \bone_{\{\|\hbbetaonetwomin \| \leq M\}} \right) \right| \\
	& \qquad \leq \| f\|_{\infty} \left\{ \left|f(M_1)-f \left( \bX_1^{\tp}\hbbetaonetwomin\bone_{\{\|\hbbetaonetwomin \| \leq M\}} \right) \right| \right.\\
	& \qquad\qquad + \left. \left| f(M_2)-f \left( \bX_2^{\tp}\hbbetaonetwomin\bone_{\{\|\hbbetaonetwomin \| \leq M\}} \right) \right| \right\} \\
& \qquad \leq  C|x| \cdot
	\left| \bX_1^{\tp}\hbbetaonemin\bone_{\{\| \hbbetaonemin\| \leq M\}} -\bX_1^{\tp}\hbbetaonetwomin \bone_{\{\|\hbbetaonetwomin \| \leq M\}} \right| \\
	& \qquad\qquad +|x| \cdot \left |\bX_2^{\tp}\hbbetatwomin\bone_{\{\|\hbbetatwomin\| \leq M\}} -\bX_2^{\tp}\hbbetaonetwomin \bone_{\{\|\hbbetaonetwomin \| \leq M\}} \right|.
\end{align*}
Consequently, $$f(M_1) f(M_2)- f \left( \bX_1^{\tp}\hbbetaonetwomin\bone_{\{\|\hbbetaonetwomin \| \leq M\}} \right)f \left(\bX_2^{\tp}\hbbetaonetwomin\bone_{\{\|\hbbetaonetwomin \| \leq M\}} \right) ~\convP~ 0.$$ As $\| f\| _{\infty} \leq 1$, this implies convergence in $\mathrm{L}_1$. Thus, it simply suffices to show that 
$$  \E \left[ f\left(\bX_1^{\tp}\hbbetaonetwomin \bone_{\{\|\hbbetaonetwomin \| \leq M \}}\right)  f\left(\bX_2^{\tp}\hbbetaonetwomin\bone_{\{\|\hbbetaonetwomin \| \leq M \}}\right) \right] ~\rightarrow~ 0.  $$
Denote the design matrix on dropping the first and second row as $\bXonetwomin$. Note that conditional on $\bXonetwomin$, $\bX_1^{\tp}\hbbetaonetwomin\bone_{\{\|\hbbetaonetwomin \| \leq M\}}$ and $\bX_2^{\tp}\hbbetaonetwomin\bone_\{\|\hbbetaonetwomin \| \leq M\}$ are independent and have distribution $$\dnorm\left(0, \| \hbbetaonetwomin\|^2\bone_{\{\|\hbbetaonetwomin \| \leq M\}} \right).$$ Using this and by arguments similar to \cite[Lemma 3.23]{el2015impact}, one can show that  
\begin{multline}
	 \E\left[ e^{i \left( t\bX_1^{\tp}\hbbetaonetwomin\bone_{\{\|\hbbetaonetwomin \| \leq M\}}+w\bX_2^{\tp}\hbbetaonetwomin\bone_{\{\|\hbbetaonetwomin \| \leq M\}} \right)}\right] \\ 
	 -\E\left[e^{it\bX_1^{\tp}\hbbetaonetwomin\bone_{\{\|\hbbetaonetwomin \| \leq M\}}}\right] 
	  \E\left[e^{iw\bX_2^{\tp}\hbbetaonetwomin\bone_{\{\|\hbbetaonetwomin \| \leq M\}}}\right] ~\rightarrow~ 0.
	\label{eq: ch3}
 \end{multline}
On repeated application of the multivariate inversion theorem for obtaining densities from characteristic functions, we get that 
\begin{multline*}
	 \E\left[ f\left( \bX_1^{\tp}\hbbetaonetwomin\bone_{\{\|\hbbetaonetwomin \| \leq M\}} \right) f\left( \bX_2^{\tp}\hbbetaonetwomin\bone_{\{\|\hbbetaonetwomin \| \leq M\}} \right)   \right]  \\
	 - \E \left[f\left(\bX_1^{\tp}\hbbetaonetwomin\bone_{\{\|\hbbetaonetwomin \| \leq M\}} \right)\right]  \E \left[f\left( \bX_2^{\tp}\hbbetaonetwomin\bone_{\{\|\hbbetaonetwomin \| \leq M\}} \right)\right] ~\rightarrow~ 0.
\end{multline*}
Since $f$ is centered, this completes the proof.
\end{proof}